\documentclass[11pt]{report}

\usepackage[english]{babel} 
\usepackage[utf8]{inputenc} 
\usepackage[T1]{fontenc} 
\usepackage{fullpage}
\usepackage{amsfonts}
\usepackage{amsthm}
\usepackage{amsmath}
\usepackage{lmodern}
\usepackage{amssymb}
\usepackage{mathtools}
\usepackage{enumerate}
\usepackage{tikz-cd}
\usepackage{tikz}
\usepackage{graphicx}
\usepackage{stmaryrd}
\usepackage{longtable,graphics,multirow}
\usepackage{mathrsfs}
\usepackage{hyperref}
\usepackage{xcolor}
\usepackage[justification=centering]{caption}
\usepackage{makecell}
\usepackage{pdfpages}

\hypersetup{
colorlinks=true,
linkcolor=[rgb]{0,0.4,0.6},
filecolor=blue,
urlcolor=cyan,
citecolor=[rgb]{0,0.6,0},%escala=255
anchorcolor=blue
}

\newtheorem{teor}{Theorem}
\numberwithin{teor}{section}

\newtheorem{lemma}[teor]{Lemma}

\newtheorem{prop}[teor]{Proposition}

\newtheorem{coro}[teor]{Corollary}

\newtheorem{remark}[teor]{Remark}

\theoremstyle{definition}\newtheorem{defi}[teor]{Definition}

\theoremstyle{definition}\newtheorem{example}{Example}
\numberwithin{example}{subsection}

\definecolor{customPink}{rgb}{1, 0.2, 0.4}%escala=255
\definecolor{customGreen}{rgb}{0, 0.6, 0}
\definecolor{customBlue}{rgb}{0, 0.4, 0.6}

\def\P{\mathbb{P}}

\def\A{\mathbb{A}}
\def\C{\mathbb{C}}

\def\Z{\mathbb{Z}}

\def\Q{\mathbb{Q}}

\def\Div{\text{Div}}
\def\NS{\text{NS}}

\def\Q{\Bbb{Q}}

\newcommand{\EEight}{
\begin{tikzpicture}[thick,scale=1.4, every node/.style={scale=1}]

\draw (0,0) node{$\bullet$};

\draw (0,0) node[above]{\small{$2\Theta_7$}};	

\draw (0.6,0) node{$\bullet$};

\draw (0.6,0) node[above]{\small{$4\Theta_6$}};	

\draw (1.2,0) node{$\bullet$};

\draw (1.2,0) node[above]{\small{$6\Theta_5$}};	

\draw (1.2,-0.5) node{$\bullet$};

\draw (1.2,-0.5) node[right]{\small{$3\Theta_8$}};	

\draw (1.8,0) node{$\bullet$};

\draw (1.8,0) node[above]{\small{$5\Theta_4$}};	

\draw (2.4,0) node{$\bullet$};

\draw (2.4,0) node[above]{\small{$4\Theta_3$}};	

\draw (3,0) node{$\bullet$};

\draw (3,0) node[above]{\small{$3\Theta_2$}};	

\draw (3.6,0) node{$\bullet$};

\draw (3.6,0) node[above]{\small{$2\Theta_1$}};	

\draw (4.2,0) node{$\bullet$};

\draw (4.2,0) node[above]{\small{$\Theta_0$}};	

\draw (0,0) -- (0.6,0);

\draw (0.6,0) -- (1.2,0);

\draw (1.2,0) -- (1.8,0);

\draw (1.2,0) -- (1.2,-0.5);

\draw (1.8,0) -- (2.4,0);

\draw (2.4,0) -- (3,0);

\draw (3,0) -- (3.6,0);

\draw (3.6,0) -- (4.2,0);

\end{tikzpicture}
}

\newcommand{\ESeven}{
\begin{tikzpicture}[thick,scale=1.4, every node/.style={scale=1}]
\draw (0,0) node{$\bullet$};

\draw (0,0) node[above]{\small{$\Theta_0$}};	

\draw (0.6,0) node{$\bullet$};

\draw (0.6,0) node[above]{\small{$2\Theta_1$}};	

\draw (1.2,0) node{$\bullet$};

\draw (1.2,0) node[above]{\small{$3\Theta_2$}};	

\draw (1.8,0) node{$\bullet$};

\draw (1.8,0) node[above]{\small{$4\Theta_3$}};	

\draw (1.8,-0.5) node{$\bullet$};

\draw (1.8,-0.5) node[right]{\small{$2\Theta_7$}};	

\draw (2.4,0) node{$\bullet$};

\draw (2.4,0) node[above]{\small{$3\Theta_4$}};	

\draw (3,0) node{$\bullet$};

\draw (3,0) node[above]{\small{$2\Theta_5$}};	

\draw (3.6,0) node{$\bullet$};

\draw (3.6,0) node[above]{\small{$\Theta_6$}};	

\draw (0,0) -- (0.6,0);

\draw (0.6,0) -- (1.2,0);

\draw (1.2,0) -- (1.8,0);

\draw (1.8,0) -- (2.4,0);

\draw (1.8,0) -- (1.8,-0.5);

\draw (2.4,0) -- (3,0);

\draw (3,0) -- (3.6,0);
\end{tikzpicture}
}

\newcommand{\ESix}{
\begin{tikzpicture}[thick,scale=1.4, every node/.style={scale=1}]
\draw (0,0) node{$\bullet$};

\draw (0,0) node[above]{\small{$\Theta_0$}};	

\draw (0.6,0) node{$\bullet$};

\draw (0.6,0) node[above]{\small{$2\Theta_1$}};	

\draw (1.2,0) node{$\bullet$};

\draw (1.2,0) node[above]{\small{$3\Theta_2$}};	

\draw (1.2,-0.5) node{$\bullet$};

\draw (1.2,-0.5) node[right]{\small{$2\Theta_5$}};	

\draw (1.2,-1) node{$\bullet$};

\draw (1.2,-1) node[right]{\small{$\Theta_6$}};

\draw (1.8,0) node{$\bullet$};

\draw (1.8,0) node[above]{\small{$2\Theta_3$}};	

\draw (2.4,0) node{$\bullet$};

\draw (2.4,0) node[above]{\small{$\Theta_4$}};	

\draw (0,0) -- (0.6,0);

\draw (0.6,0) -- (1.2,0);

\draw (1.2,0) -- (1.2,-0.5);

\draw (1.2,-0.5) -- (1.2,-1);

\draw (1.2,0) -- (1.8,0);

\draw (1.8,0) -- (2.4,0);
\end{tikzpicture}
}

\newcommand{\Dn}{
\begin{tikzpicture}[thick,scale=1.4, every node/.style={scale=1}]
\draw (0.1,0.5) node{$\bullet$};

\draw (0.1,0.5) node[left]{\small{$\Theta_0$}};	

\draw (0.1,-0.5) node{$\bullet$};

\draw (0.1,-0.5) node[left]{\small{$\Theta_1$}};	

\draw (0.5,0) node{$\bullet$};

\draw (0.5,0) node[left]{\small{$2\Theta_4$}};	

\draw (1,0) node{$\bullet$};

\draw (1,0) node[below]{\small{$2\Theta_5$}};	

\draw (2,0) node{$\bullet$};

\draw (2,0) node[below]{\small{$2\Theta_{n+3}$}};	

\draw (2.5,0) node{$\bullet$};

\draw (2.5,0) node[right]{\small{$2\Theta_{n+4}$}};	

\draw (2.9,0.5) node{$\bullet$};

\draw (2.9,0.5) node[right]{\small{$\Theta_2$}};	

\draw (2.9,-0.5) node{$\bullet$};

\draw (2.9,-0.5) node[right]{\small{$\Theta_3$}};	

\draw (0.1,0.5) -- (0.5,0);

\draw (0.1,-0.5) -- (0.5,0);

\draw (0.5,0) -- (1,0);

\draw[dashed] (1,0) -- (2,0);

\draw (2,0) -- (2.5,0);

\draw (2.5,0) -- (2.9,0.5);

\draw (2.5,0) -- (2.9,-0.5);
\end{tikzpicture}
}

\newcommand{\An}{
\begin{tikzpicture}[thick,scale=1.4, every node/.style={scale=1}]

\draw (0,0) node{$\bullet$}; 

\draw (0,0) node[left]{\small{$\Theta_0$}};	

\draw (0.5,0.5) node{$\bullet$};	

\draw (0.5,0.5) node[above]{\small{$\Theta_1$}};	

\draw (1, 0.25) node{$\bullet$};

\draw (1, 0.25) node[right]{\small{$\Theta_2$}};	

\draw (1, -0.25) node{$\bullet$};

\draw (1, -0.25) node[right]{\small{$\Theta_3$}};	

\draw (0.5, -0.5) node{$\bullet$};

\draw (0.5, -0.5) node[below]{\small{$\Theta_{n-1}$}};	

\draw (0,0)--(0.5, 0.5);

\draw (0,0)--(0.5, -0.5);

\draw (0.5, 0.5)--(1, 0.25);

\draw (1, 0.25)--(1, -0.25);

\draw[dashed] (1, -0.25)--(0.5,-0.5);

\draw (0.5, -0.5)--(0,0);
\end{tikzpicture}
}

\newcommand{\TypeA}{
\begin{tikzpicture}[thick,scale=1.4, every node/.style={scale=1}]

\draw[fill=white] (0,0) circle (1.5	pt);

\draw (0,0) node[above]{\small{$1$}};	

\draw (0.5,0) node{$\bullet$}; 

\draw (0.5,0) node[above]{\small{$1$}};	

\draw (1,0) node{$\bullet$};

\draw (1,0) node[above]{\small{$1$}};	

\draw (1.5,0) node{$\bullet$};

\draw (1.5,0) node[above]{\small{$1$}};

\draw (2,0) node{$\bullet$};

\draw (2,0) node[above]{\small{$1$}};

\draw[fill=white] (2.5,0) circle (1.5pt);

\draw (2.5,0) node[above]{\small{$1$}};

\draw (0.05,0) -- (0.5,0);

\draw (0.5,0) -- (1,0);

\draw[dashed] (1,0) -- (1.5,0);

\draw (1.5,0) -- (2,0);

\draw (2,0) -- (2.45,0);
\end{tikzpicture}
}

\newcommand{\TypeB}{
\begin{tikzpicture}[thick,scale=1.4, every node/.style={scale=1}]

\draw[fill=white] (0,0) circle (1.5	pt);

\draw (0,0) node[above]{\small{$2$}};	

\draw (0.5,0) node{$\bullet$}; 

\draw (0.5,0) node[above]{\small{$2$}};	

\draw (1,0) node{$\bullet$};

\draw (1,0) node[above]{\small{$2$}};	

\draw (1.5,0) node{$\bullet$};

\draw (1.5,0) node[above]{\small{$2$}};

\draw (2,0) node{$\bullet$};

\draw (2,0) node[above]{\small{$2$}};

\draw (2.5,0.3) node{$\bullet$};

\draw (2.5,0.3) node[above]{\small{$1$}};

\draw (2.5,-0.3) node{$\bullet$};

\draw (2.5,-0.3) node[below]{\small{$1$}};

\draw (0.05,0) -- (0.5,0);

\draw (0.5,0) -- (1,0);

\draw[dashed] (1,0) -- (1.5,0);

\draw (1.5,0) -- (2,0);

\draw (2,0) -- (2.5,0.3);

\draw (2,0) -- (2.5,-0.3);
\end{tikzpicture}
}

\newcommand{\TypeC}{
\begin{tikzpicture}[thick,scale=1.4, every node/.style={scale=1}]

\draw (0,0) node{$\bullet$};

\draw (0,0) node[above]{\small{$1$}};

\draw[fill=white] (0.5,0) circle (1.5pt);

\draw (0.5,0) node[above]{\small{$2$}};	

\draw (1,0) node{$\bullet$};

\draw (1,0) node[above]{\small{$1$}};	

\draw (0,0) -- (0.45,0);

\draw (0.55,0) -- (1,0);
\end{tikzpicture}
}

\newcommand{\TypeD}{
\begin{tikzpicture}[thick,scale=1.4, every node/.style={scale=1}]

\draw[fill=white] (0,0) circle (1.5pt);

\draw (0,0) node[above]{\small{$1$}};	

\draw[fill=white] (0.5,0) circle (1.5pt);

\draw (0.5,0) node[above]{\small{$1$}};	

\draw (0.05,0) -- (0.45,0);
\end{tikzpicture}
}

\newcommand{\TypeE}{
\begin{tikzpicture}[thick,scale=1.4, every node/.style={scale=1}]

\draw (0,0) node{$\star$};

\draw (0,0) node[above]{\small{$1$}};
\end{tikzpicture}
}

\newcommand{\DiagramOne}{
\begin{tikzpicture}[thick,scale=1.4, every node/.style={scale=1}]

\draw[fill=white] (0,0) circle (1.5	pt);

\draw (0,0) node[below]{\small{$O$}};

\draw (0,0) node[above]{\small{$2$}};	

\draw (0.5,0) node{$\bullet$}; 

\draw (0.5,0) node[below]{\small{$\Theta_0$}};

\draw (0.5,0) node[above]{\small{$2$}};	

\draw (1,0) node{$\bullet$};

\draw (1,0) node[above]{\small{$2$}};	

\draw (1,0) node[below]{\small{$\Theta_1$}};

\draw (1.5,0) node{$\bullet$};

\draw (1.5,0) node[above]{\small{$2$}};

\draw (1.5,0) node[below]{\small{$\Theta_2$}};

\draw (2,0) node{$\bullet$};

\draw (2,0) node[above]{\small{$2$}};

\draw (2,0) node[below]{\small{$\Theta_3$}};

\draw (2.5,0) node{$\bullet$};

\draw (2.5,0) node[above]{\small{$2$}};

\draw (2.5,0) node[below]{\small{$\Theta_4$}};

\draw (3,0) node{$\bullet$};

\draw (3,0) node[above]{\small $2$};

\draw (3,0) node[below]{\small{$\Theta_5$}};

\draw (3.5,0.25) node{$\bullet$};

\draw (3.5,0.25) node[right]{\small{$\Theta_8$}};

\draw (3.5,-0.25) node{$\bullet$};

\draw (3.5,-0.25) node[right]{\small{$\Theta_6$}};

\draw (0.05,0) -- (0.5,0);

\draw (0.5,0) -- (1,0);

\draw (1,0) -- (1.5,0);

\draw (1.5,0) -- (2,0);

\draw (2,0) -- (2.5,0);

\draw (2.5,0) -- (3,0);

\draw (3,0) -- (3.5,0.25);

\draw (3,0) -- (3.5,-0.25);
\end{tikzpicture}
}

\newcommand{\DiagramTwo}{
\begin{tikzpicture}[thick,scale=1.4, every node/.style={scale=1}]

\draw[fill=white] (0,0) circle (1.5	pt);

\draw (0,0) node[below]{\small{$O$}};

\draw (0,0) node[above]{\small{$2$}};	

\draw (0.5,0) node{$\bullet$}; 

\draw (0.5,0) node[below]{\small{$\Theta_0$}};

\draw (0.5,0) node[above]{\small{$2$}};	

\draw (1,0) node{$\bullet$};

\draw (1,0) node[above]{\small{$2$}};	

\draw (1,0) node[below]{\small{$\Theta_1$}};

\draw (1.5,0) node{$\bullet$};

\draw (1.5,0) node[above]{\small{$2$}};

\draw (1.5,0) node[below]{\small{$\Theta_2$}};

\draw (2,0) node{$\bullet$};

\draw (1.85,0) node[above]{\small{$2$}};

\draw (2,0) node[below]{\small{$\Theta_3$}};

\draw (2.5,0.25) node{$\bullet$};

\draw (2.5,0.25) node[right]{\small{$\Theta_7$}};

\draw (2.5,-0.25) node{$\bullet$};

\draw (2.5,-0.25) node[right]{\small{$\Theta_4$}};

\draw (0.05,0) -- (0.5,0);

\draw (0.5,0) -- (1,0);

\draw (1,0) -- (1.5,0);

\draw (1.5,0) -- (2,0);

\draw (2,0) -- (2.5,0.25);

\draw (2,0) -- (2.5,-0.25);

\end{tikzpicture}
}

\newcommand{\DiagramThree}{
\begin{tikzpicture}[thick,scale=1.4, every node/.style={scale=1}]

\draw[fill=white] (0,0) circle (1.5	pt);

\draw (0,0) node[below]{\small{$P_1$}};

\draw (0.5,0) node{$\bullet$}; 

\draw (0.5,0) node[below]{\small{$\Theta_0$}};

\draw (1,0) node{$\bullet$};

\draw (1,0) node[below]{\small{$\Theta_1$}};

\draw (1.5,0) node{$\bullet$};

\draw (1.5,0) node[below]{\small{$\Theta_2$}};

\draw (2,0) node{$\bullet$};

\draw (2,0) node[below]{\small{$\Theta_3$}};

\draw (2.5,0) node{$\bullet$};

\draw (2.5,0) node[below]{\small{$\Theta_4$}};

\draw [fill=white] (3,0) circle (1.5pt);

\draw (3,0) node[below]{\small{$P_2$}};

\draw (0.05,0) -- (0.5,0);

\draw (0.5,0) -- (1,0);

\draw (1,0) -- (1.5,0);

\draw (1.5,0) -- (2,0);

\draw (2,0) -- (2.5,0);

\draw (2.5,0) -- (2.95,0);

\end{tikzpicture}
}

\newcommand{\DiagramFour}{
\begin{tikzpicture}[thick,scale=1.4, every node/.style={scale=1}]

\draw[fill=white] (0,0) circle (1.5	pt);

\draw (0,0) node[below]{\small{$O$}};

\draw (0,0) node[above]{\small{$2$}};	

\draw (0.5,0) node{$\bullet$}; 

\draw (0.5,0) node[below]{\small{$\Theta_0$}};

\draw (0.5,0) node[above]{\small{$2$}};	

\draw (1,0) node{$\bullet$};

\draw (1,0) node[above]{\small{$2$}};	

\draw (1,0) node[below]{\small{$\Theta_1$}};

\draw (1.5,0.25) node{$\bullet$};

\draw (1.5,0.25) node[right]{\small{$\Theta_2$}};

\draw (1.5,-0.25) node{$\bullet$};

\draw (1.5,-0.25) node[right]{\small{$\Theta_3$}};

\draw (0.05,0) -- (0.5,0);

\draw (0.5,0) -- (1,0);

\draw (1,0) -- (1.5,0.25);

\draw (1,0) -- (1.5,-0.25);

\end{tikzpicture}
}

\newcommand{\DiagramFive}{
\begin{tikzpicture}[thick,scale=1.4, every node/.style={scale=1}]

\draw[fill=white] (0,0) circle (1.5	pt);

\draw (0,0) node[below]{\small{$O$}};

\draw (0,0) node[above]{\small{$2$}};	

\draw (0.5,0) node{$\bullet$}; 

\draw (0.5,0) node[below]{\small{$\Theta_0$}};

\draw (0.5,0) node[above]{\small{$2$}};	

\draw (1,0) node{$\bullet$};

\draw (1,0) node[above]{\small{$2$}};	

\draw (1,0) node[below]{\small{$\Theta_1$}};

\draw (1.5,0.25) node{$\bullet$};

\draw (1.5,0.25) node[right]{\small{$\Theta_3$}};

\draw (1.5,-0.25) node{$\bullet$};

\draw (1.5,-0.25) node[right]{\small{$\Theta_2$}};

\draw (0.05,0) -- (0.5,0);

\draw (0.5,0) -- (1,0);

\draw (1,0) -- (1.5,0.25);

\draw (1,0) -- (1.5,-0.25);

\end{tikzpicture}
}

\newcommand{\DiagramSix}{
\begin{tikzpicture}[thick,scale=1.4, every node/.style={scale=1}]

\draw[fill=white] (0,0) circle (1.5	pt);

\draw (0,0) node[above]{\small{$2$}};

\draw (0.5,0) node{$\bullet$}; 

\draw (0.5,0) node[above]{\small{$2$}};

\draw (1,0) node{$\bullet$};

\draw (1,0) node[above]{\small{$2$}};

\draw (1.5,0.25) node{$\bullet$};

\draw (1.5,-0.25) node{$\bullet$};

\draw (0.05,0) -- (0.5,0);

\draw (0.5,0) -- (1,0);

\draw (1,0) -- (1.5,0.25);

\draw (1,0) -- (1.5,-0.25);

\end{tikzpicture}
}

\newcommand{\DiagramSeven}{
\begin{tikzpicture}[thick,scale=1.4, every node/.style={scale=1}]

\draw[fill=white] (0,0) circle (1.5	pt);

\draw (0.5,0) node{$\bullet$}; 

\draw (1,0) node{$\bullet$};

\draw (1.5,0) node{$\bullet$};

\draw[fill=white] (2,0) circle (1.5	pt);

\draw (0.05,0) -- (0.5,0);

\draw (0.5,0) -- (1,0);

\draw (1,0) -- (1.5,0);

\draw (1.5,0) -- (1.95,0);

\end{tikzpicture}
}

\newcommand{\DiagramEight}{
\begin{tikzpicture}[thick,scale=1.4, every node/.style={scale=1}]

\draw (0,0) node{$\bullet$};

\draw (0,0) node[below]{\small{$\Theta_0^1$}};

\draw[fill=white] (0.5,0) circle (1.5pt); 

\draw (0.5,0) node[above]{\small{$2$}};

\draw (0.5,0) node[below]{\small{$O$}};

\draw (1,0) node{$\bullet$};

\draw (1,0) node[below]{\small{$\Theta_0^2$}};

\draw (0.05,0) -- (0.45,0);

\draw (0.55,0) -- (1,0);

\end{tikzpicture}
}

\newcommand{\DiagramNine}{
\begin{tikzpicture}[thick,scale=1.4, every node/.style={scale=1}]

\draw[fill=white] (0,0) circle (1.5pt);

\draw (0,0) node[below]{\small{$P_1$}};

\draw (0.5,0) node{$\bullet$};

\draw (1,0) node{$\bullet$};

\draw (1.5,0) node{$\bullet$};

\draw[fill=white] (2,0) circle (1.5pt);

\draw (2,0) node[below]{$P_2$};

\draw (0.05,0) -- (0.5,0);

\draw (0.5,0) -- (1,0);

\draw (1,0) -- (1.5,0);

\draw (1.5,0) -- (1.95,0);

\end{tikzpicture}
}

\newcommand{\DiagramTen}{
\begin{tikzpicture}[thick,scale=1.4, every node/.style={scale=1}]

\draw (0,0) node{$\bullet$};

\draw (0,0) node[below]{\small{$\Theta_0^1$}};

\draw[fill=white] (0.5,0) circle (1.5pt); 

\draw (0.5,0) node[above]{\small{$2$}};

\draw (0.5,0) node[below]{\small{$O$}};

\draw (1,0) node{$\bullet$};

\draw (1,0) node[below]{\small{$\Theta_0^2$}};

\draw (0,0) -- (0.45,0);

\draw (0.55,0) -- (1,0);

\end{tikzpicture}
}

\newcommand{\DiagramEleven}{
\begin{tikzpicture}[thick,scale=1.4, every node/.style={scale=1}]

\draw[fill=white] (0,0) circle (1.5pt);

\draw (0,0) node[below]{\small{$P_1$}};

\draw (0.5,0) node{$\bullet$};

\draw (1,0) node{$\bullet$};

\draw (1.5,0) node{$\bullet$};

\draw (2,0) node{$\bullet$};

\draw (2.5,0) node{$\bullet$};

\draw[fill=white] (3,0) circle (1.5pt);

\draw (3,0) node[below]{$P_2$};

\draw (0.05,0) -- (0.5,0);

\draw (0.5,0) -- (1,0);

\draw (1,0) -- (1.5,0);

\draw (1.5,0) -- (2,0);

\draw (2,0) -- (2.5,0);

\draw (2.5,0) -- (2.95,0);

\end{tikzpicture}
}

\newcommand{\DiagramX}{
\begin{tikzpicture}[thick,scale=1.4, every node/.style={scale=1}]

\draw[fill=white] (0,0) circle (1.5	pt);

\draw (0,0) node[above]{\small{$2$}};	

\draw (0.5,0) node{$\bullet$}; 

\draw (0.5,0) node[above]{\small{$n_0$}};	

\draw (1,0) node{$\bullet$};

\draw (1,0) node[above]{\small{$n_1$}};	

\draw (1.5,0) node{$\bullet$};

\draw (1.5,0) node[above]{\small{$n_m$}};	

\draw (0.05,0) -- (0.5,0);

\draw (0.5,0) -- (1,0);

\draw[dashed] (1,0) -- (1.5,0);
\end{tikzpicture}
}

\newcommand{\DiagramY}{
\begin{tikzpicture}[thick,scale=1.4, every node/.style={scale=1}]

\draw[fill=white] (0,0) circle (1.5	pt);

\draw (0,0) node[above]{\small{$2$}};	

\draw (0.5,0) node{$\bullet$}; 

\draw (0.5,0) node[above]{\small{$n_0$}};	

\draw (1,0) node{$\bullet$};

\draw (1,0) node[above]{\small{$n_1$}};	

\draw (1.5,0) node{$\bullet$};

\draw (1.5,0) node[above]{\small{$n_m$}};

\draw (2,0.3) node[right]{branch 1};

\draw (2,-0.3) node[right]{branch 2};

\draw (0.05,0) -- (0.5,0);

\draw (0.5,0) -- (1,0);

\draw[dashed] (1,0) -- (1.5,0);

\draw (1.5,0) -- (2.06,0.3);

\draw (1.5,0) -- (2.06,-0.3);
\end{tikzpicture}
}

\newcommand{\DiagramZ}{
\begin{tikzpicture}[thick,scale=1.4, every node/.style={scale=1}]

\draw[fill=white] (0,0) circle (1.5	pt);

\draw (0,0) node[above]{\small{$2$}};	

\draw (0.5,0) node{$\bullet$}; 

\draw (0.5,0) node[above]{\small{$n_0$}};	

\draw (1,0.4) node{$\bullet$};	

\draw (1,0.4) node[above]{\small{$n_1$}};	

\draw (1.5,0.2) node{$\bullet$};

\draw (1.5,0.2) node[above]{\small{$n_2$}};	

\draw (1.5,-0.2) node{$\bullet$};

\draw (1.5,-0.2) node[below]{\small{$n_3$}};	

\draw (1,-0.4) node{$\bullet$};

\draw (1,-0.4) node[below]{\small{$n_m$}};	

\draw (0.05,0)--(0.5,0);

\draw (0.5,0)--(1,0.4);

\draw (1,0.4)--(1.5,0.2);

\draw (1.5,0.2)--(1.5,-0.2);

\draw[dashed] (1.5,-0.2)--(1,-0.4);

\draw (1,-0.4)--(0.5,0);
\end{tikzpicture}
}

\begin{document}

\begin{titlepage}
	\centering
	\includegraphics[width=0.15\textwidth]{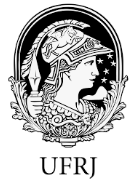}\par\vspace{1cm}
	{\Large UNIVERSIDADE FEDERAL DO RIO DE JANEIRO\par}
	{\Large CENTRO DE CIÊNCIAS MATEMÁTICAS E DA NATUREZA\par}
	{\Large INSTITUTO DE MATEMÁTICA\par}	
	\vspace{2.5cm}
	{\huge\bfseries Geometric and arithmetic aspects\\of rational elliptic surfaces\par}
	\vspace{2cm}
	{\Large Renato Dias Costa\par}
	{\Large Advisor: Cecília Salgado\par}
	\vfill
	{A thesis submitted in partial fulfilment of the requirements\\ for the degree of Doctor of Philosophy (PhD) in Mathematics.}
	\vfill
	{\large Rio de Janeiro\par}
	{\large 2022\par}
\end{titlepage}

\includepdf[pages={1}]{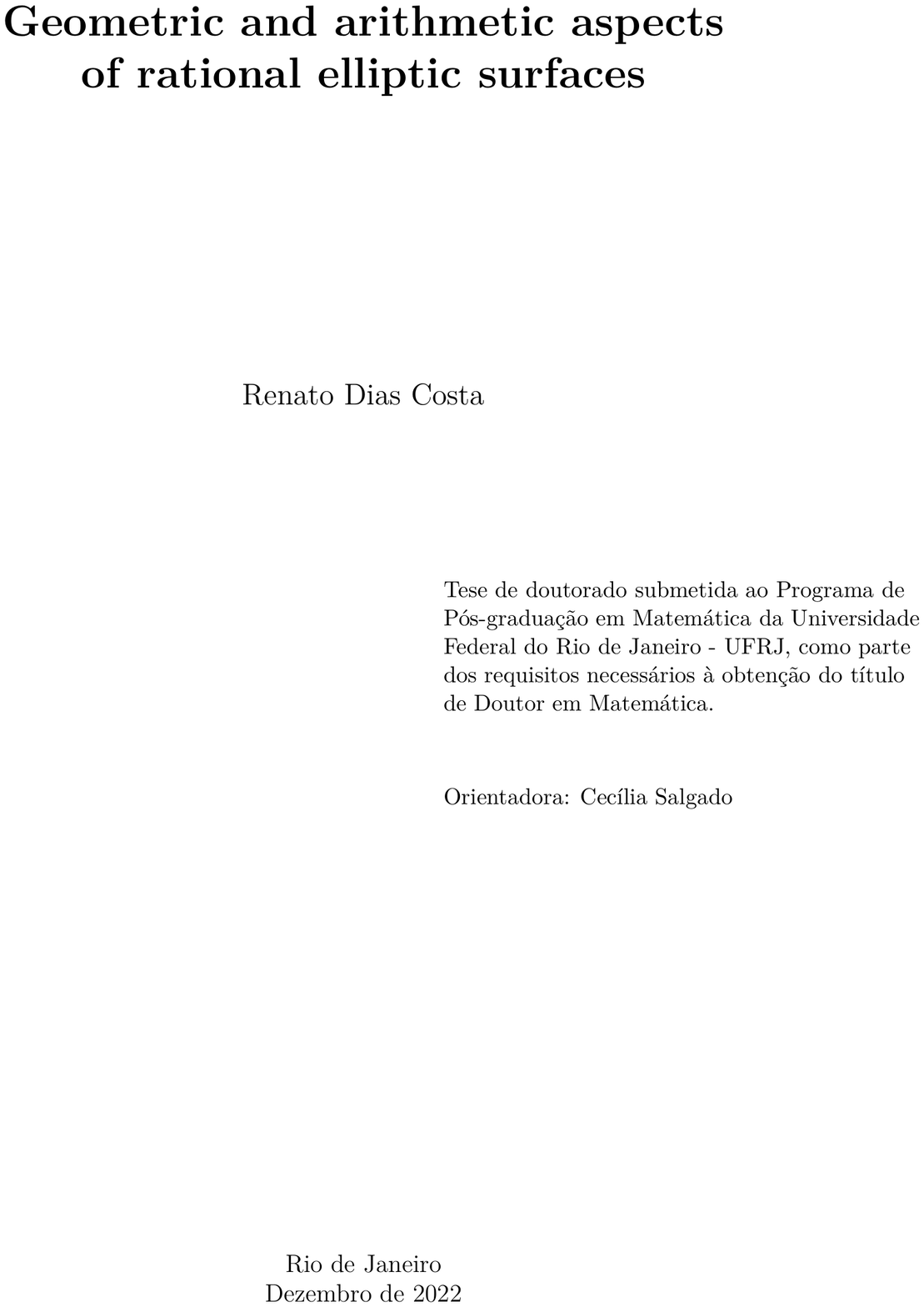}

\includepdf[pages={1}]{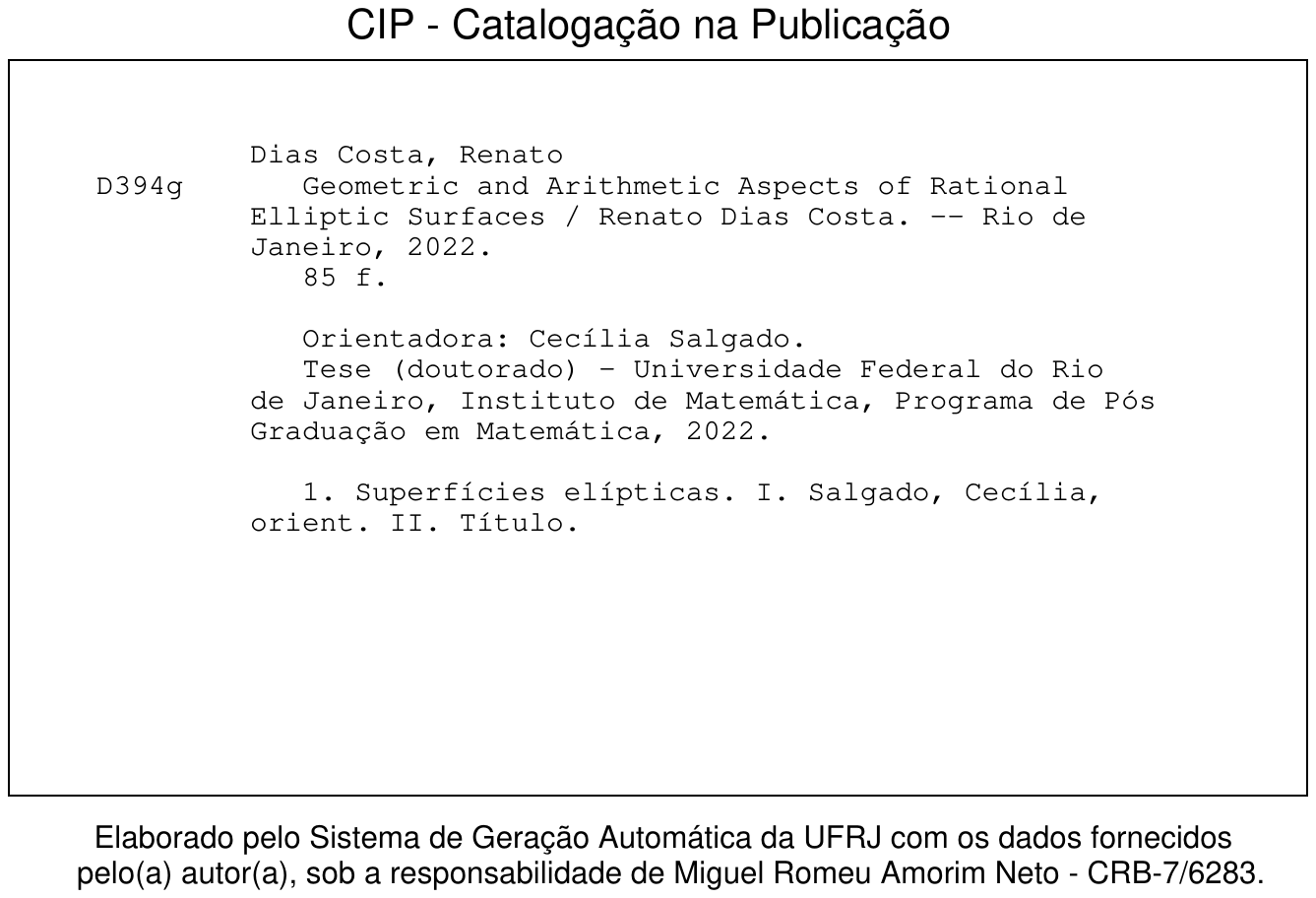}

\includepdf[pages={1}]{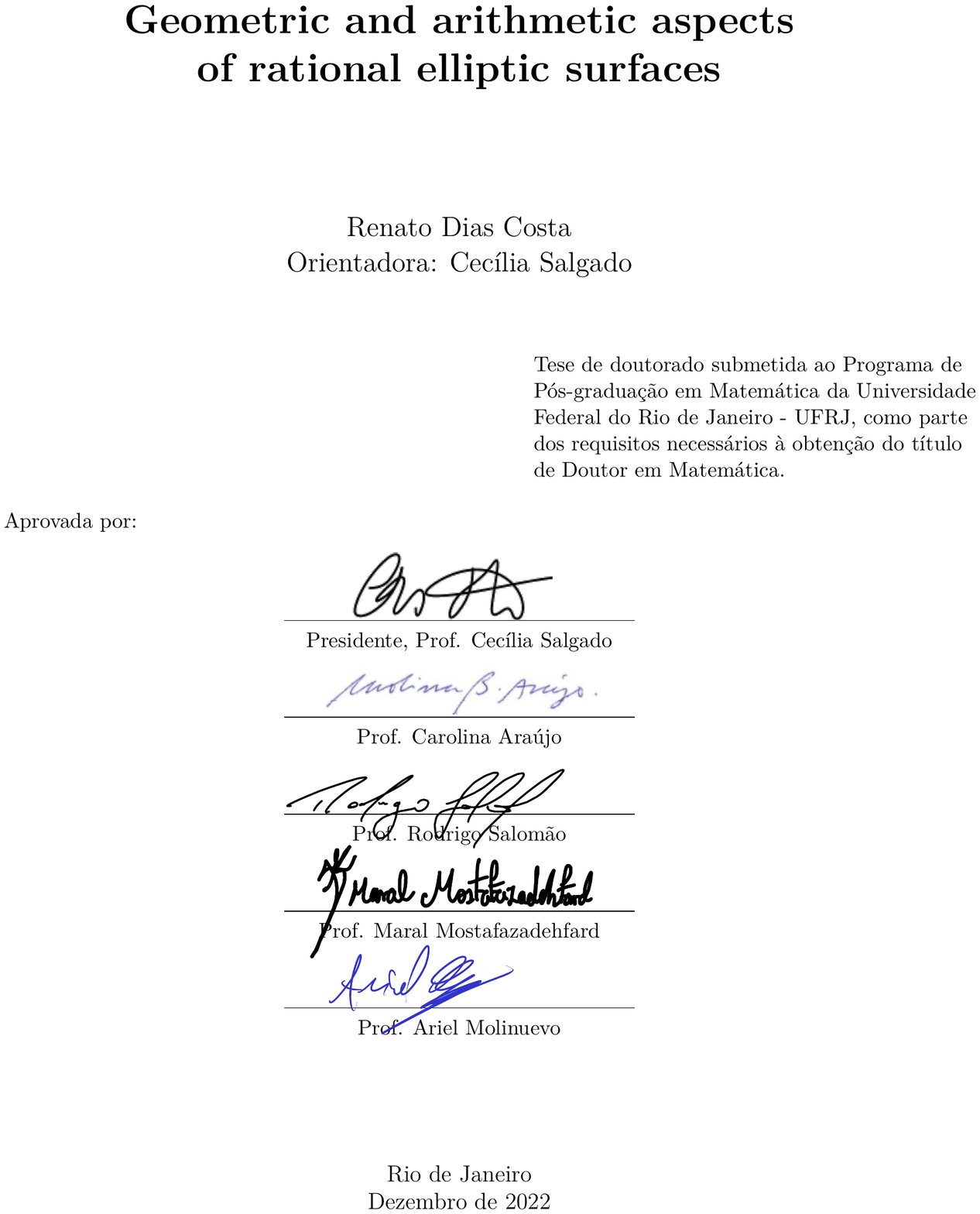}

\chapter*{Agradecimentos}\
\noindent É um dever e uma alegria agradecer àqueles que estiveram ao meu lado nos últimos quatro anos.

\thispagestyle{empty}

Tive a sorte de trabalhar com Cecília Salgado, que me deu um exemplo de como pensa uma profissional de verdade e de como é possível transitar entre a academia e a vida em família. Agradeço pela paciência, principalmente no início, de repetir as mesmas explicações mil vezes até me fazer entender que alho não é bugalho. Cecília acabou se revelando uma boa professora de redação, o que foi uma grata surpresa. Agradeço por ter acreditado em mim, às vezes mais do que eu acreditei em mim mesmo: este trabalho teria sido muito mais difícil sem esse apoio.

Sou grato a Alice Garbagnati pela generosidade de ler meu primeiro artigo com tanto cuidado e sugerir investigações posteriores. Também ao meu colega Felipe Zingali, pelas chances de compartilhar nossas dificuldades em descobrir como alguém descobre alguma coisa. 

Não posso perder a ocasião de mencionar um professor dos mais notáveis, Rodrigo José Gondim Neves. Foi a primeira pessoa que conheci com disposição e competência para entender a matemática não apenas em nome de suas aplicações mas, como escreveu Jacobi, \textit{pour l'honneur de l'esprit humain}. Suas aulas sobre polinômios ciclotômicos estão entre minhas lembranças mais queridas: devo a ele mais do que ele suspeita.

Ao meu pai e minha mãe, devo minha vida e tantas outras coisas que nunca vou poder retribuir. À minha família, agradeço pelo apoio e pela torcida, principalmente no início do doutorado, quando minha filha acabava de nascer. A presença diária da minha esposa Edlaura e da minha filha Cecília sempre manteve meus pés no chão e tem deixado minha vida cada vez mais cheia de sentido e afeto. Espero que, entre as nossas conquistas juntos, este trabalho represente uma delas.

Agradeço aos membros da banca Rodrigo Salomão, Carolina Araújo, Maral Mostafazadehfard, Ariel Molinuevo, Luciane Quoos e Miriam Abdon por seus comentários e sugestões, dos quais tirei muito proveito. 

Agradeço ainda às agências de fomento CAPES, CNPq e FAPERJ pelo auxílio financeiro, sem o qual este trabalho não teria sido possível.

\newpage

\chapter*{Abstract}\
\indent This thesis collects the material developed in the course of three lines of investigation concerning arithmetic and algebraic aspects of rational elliptic surfaces. We organize the text as follows. 

\thispagestyle{empty}

In Chapter~\ref{ch:introduction} we present the three lines of investigation and explain their connection with one another. In short, they are:
\begin{enumerate}
\item Over an algebraically closed field, classify the fibers of conic bundles on rational elliptic surfaces and describe the interplay between the fibers of the elliptic fibration and the fibers of the conic bundle.
\item Given a rational elliptic surface over an algebraically closed field, investigate the numbers that cannot occur as the intersection number of a pair of sections, which we call \textit{gap numbers}. More precisely, try to answer when gap numbers exist, how they are distributed and how to identify them.
\item Given a rational elliptic surface over a number field, study the set of fibers whose Mordell-Weil rank is higher than the generic rank. More specifically, present conditions to guarantee that the collection of fibers where the rank jumps of at least 3 is not thin.
\end{enumerate}

Chapter~\ref{ch:preliminaries} is dedicated to establishing notations, definitions and well-known results on which this work is based. The results related to each of the three topics mentioned receive an individual chapter, namely Chapters~\ref{ch:conic_bundles_on_RES}, \ref{ch:gaps} and \ref{ch:rank_jumps}. The appendix in Chapter~\ref{ch:appendix} stores data relevant to Chapter~\ref{ch:gaps}.

The results in Chapters~\ref{ch:conic_bundles_on_RES}, \ref{ch:gaps} and \ref{ch:rank_jumps} also appear separately in the following preprints:

\begin{itemize}
\item \cite{ConicBundles} R.D. Costa.\textit{Classification of conic bundles on a rational elliptic surface in any characteristic}. arXiv:2206.03549.
\item \cite{Gaps} R.D. Costa. \textit{Gaps on the intersection numbers of sections on a rational elliptic surface}. arXiv:2301.03137.
\item \cite{DiasCostaSalgado} R.D. Costa, C. Salgado. \textit{Large rank jumps on elliptic surfaces and the Hilbert property}. arXiv:2205.07801.
\end{itemize}

\noindent\textbf{Keywords:} Elliptic surfaces, elliptic fibrations, conic bundles, Mordell-Weil lattices.

\newpage

\tableofcontents

\chapter{Introduction}\label{ch:introduction}\
\indent The central object of this thesis are rational elliptic surfaces (see Definitions~\ref{def:elliptic_surfaces} and \ref{def:geometrically_rational}). Our investigation stems from three motivating problems regarding both arithmetic and geometric aspects of these surfaces. We introduce each problem in its particular context, explain our strategy for approaching it and state our results.

From a historical perspective, elliptic surfaces emerge from number theoretic problems and receive a progressively more geometric treatment over time. From the viewpoint of elliptic curves over a function field, they have first been studied over a finite field of constants by Artin \cite{Artin} in the search for an analogue of the Riemann Hypothesis, later proved in generality by Weil \cite{Weil}. In the framework of algebraic surfaces over the complex numbers, elliptic fibrations were already known to Enriques \cite{Enriques}, but only as a class of examples among others. 

A much more dedicated treatment was given by Kodaira \cite{KodII, KodIII} using the language of complex geometry, which laid the foundations for a profound study of elliptic surfaces. The algebraic theory was soon advanced by Néron \cite{Neron64} and Shafarevich \cite{Shafarevich65} and received an important contribution from Tate \cite{Tate75} with a simplified algorithm for identifying singular fibers.

A key aspect of elliptic surfaces is the bijective correspondence between sections of the elliptic fibration and points on the generic fiber (Subsection~\ref{subsection:correspondence_between_sections_and_points}), first properly emphasized by Shioda in \cite{Shioda72}. This correspondence is a fundamental ingredient in the construction of the Mordell-Weil lattice (Section~\ref{section:MW_lattice}) by Shioda \cite{Shioda89} and Elkies \cite{Elkies90} independently. This proved to be a powerful tool of both geometric and arithmetic interest, with applications in various topics such as Galois representations, moduli spaces of K3 surfaces and crystallography.

Our focus on rational elliptic surfaces is due to their convenient geometric properties (Theorem~\ref{thm:RES_distinguished_properties}), the simplicity with which examples can be produced (Section~\ref{section:examples_conic_bundles}) and the fact that their possible fiber configurations and possible Mordell-Weil lattices have been completely classified in \cite{Persson} and \cite{OguisoShioda} respectively.

In what follows $X$ is a geometrically rational elliptic surface with elliptic fibration $\pi:X\to\P^1$ over a field $k$ which is either a number field or an algebraically closed field. We let $K:=k(\P^1)$ be the function field of the base curve $\P^1$ and define the \textit{Mordell-Weil group} of $\pi$ as the group of $K$-points on the generic fiber $E$, denoted by $E(K)$. We note that every geometrically rational elliptic surface admits precisely one elliptic fibration (Proposition~\ref{thm:RES_distinguished_properties}), hence we may refer to $E(K)$ as the Mordell-Weil group \textit{of the surface} $X$. We use $r$ to denote the rank of $E(K)$, called the \textit{Mordell-Weil rank} or \textit{generic rank}. 

We present our three problems of interest.

\section{Classification of conic bundle fibers}\
\indent In a broad sense, a \textit{conic bundle} can be understood as a genus zero fibration on a variety (see Definition~\ref{def:conic_bundles} for surfaces). Conic bundles arise in different forms in many important contexts, from the classification of $k$-minimal rational surfaces to the Minimal Model Program (Section~\ref{section:conic_bundles} and Chapter~\ref{ch:conic_bundles_on_RES} for a more precise account). In this thesis we are concerned more specifically with conic bundles on a rational elliptic surface $\pi:X\to\P^1$. This has been motivated by results from \cite{Salgado12}, \cite{LoughSalgado}, where the presence of conic bundles is used as a condition to guarantee that rank jumps occur (Section~\ref{section:large_rank_jumps} and Chapter~\ref{ch:rank_jumps}); from \cite{GarbagnatiSalgado17, GarbagnatiSalgado20}, where conic bundles are used to classify elliptic fibrations on certain K3 surfaces; and from \cite{ArtebaniGarbagnatiLaface}, where conic bundles appear in the study of generators of the Cox ring of $X$.

These applications seem to justify a further study of conic bundles on rational elliptic surfaces. More specifically, we propose two questions: first, in which ways can we construct a divisor $D$ on $X$ such that the linear system $|D|$ induces a conic bundle on $X$? Second, to what extent is this construction obstructed by the fiber configuration of the elliptic fibration?

Our starting point is an observation, already implicit in \cite{ArtebaniGarbagnatiLaface}, \cite{GarbagnatiSalgado17, GarbagnatiSalgado20}, that there is a bijective correspondence between conic bundles on $X$ and certain Néron-Severi classes $[D]\in\NS(X)$, which we call \textit{conic classes} (Definition~\ref{def:conic_class}), which indicates that all information we need can be derived from the numerical behavior of a divisor $D$ representing a conic class.

We note that this study is essentially geometric, therefore it makes sense to work over an algebraic closure of the base field. Moreover, except for minor difficulties, there is no reason to restrict ourselves to characteristic zero, which is the case in \cite{Salgado12, LoughSalgado} over number fields, in \cite{ArtebaniGarbagnatiLaface, GarbagnatiSalgado17} over $\C$ or in \cite{GarbagnatiSalgado20} over a field of characteristic zero. 

Our first result is a complete classification of the fibers of a conic bundle on $X$ and answers our first question about the possibilities of $D$ such that $|D|$ induces a conic bundle $\varphi_{|D|}:X\to \P^1$.
\\ \\
\noindent\textbf{Theorem~\ref{thm:classification_conic_bundles}.} \textit{Let $\pi:X\to\P^1$ be a rational elliptic surface over an algebraically closed field and let $\varphi:X\to\P^1$ be a conic bundle. If $D$ is a fiber of $\varphi$, then the intersection graph of $D$ (multiplicities considered) fits one of the types below. Conversely, if the intersection graph of a divisor $D$ fits any of these types, then $|D|$ induces a conic bundle $\varphi_{|D|}:X\to\P^1$.}
\newpage
\begin{table}[h]
\begin{center}
\centering
\begin{tabular}{|c|c|} 
\hline
\multirow{2}{*}{\hfil Type} & \multirow{2}{*}{$\begin{matrix}\text{Intersection graph}\\ \text{(with multiplicities)}\end{matrix}$}\\
&\\
\hline
\multirow{3}{*}{\hfil $0$} & \multirow{3}{*}{\hfil \TypeE}\\ %& \multirow{3}{*}{\hfil $C$}\\
&\\
&\\
\hline
\multirow{3}{*}{\hfil $A_2$} & \multirow{3}{*}{\hfil \TypeD}\\%& \multirow{3}{*}{\hfil $P+Q$}\\
& \\
& \\
\hline
\multirow{3}{*}{\hfil $A_n$ ($n\geq 3$)} & \multirow{3}{*}{\hfil \TypeA}\\ %& \multirow{3}{*}{\hfil $P+\sum_i\Theta_i+Q$}\\
& \\ 
& \\
\hline
\multirow{3}{*}{\hfil $D_3$} & \multirow{3}{*}{\hfil \TypeC}\\ %& \multirow{3}{*}{\hfil $\Theta_0^1+2P+\Theta_0^2$}\\ 
& \\
& \\
\hline
\multirow{4}{*}{\hfil $D_m$ ($m\geq 4$)} & \multirow{4}{*}{\hfil \TypeB}\\ %& \multirow{4}{*}{\hfil $2P+\sum_{i=0}^{n-2}2\Theta_i+\Theta_{n-1}+\Theta_n$}\\
& \\ 
& \\
& \\
\hline
\end{tabular}
\end{center}
\begin{align*}
\star&\,\,\text{smooth, irreducible curve of genus }0\\
\circ&\,(-1)\text{-curve (section of }\pi)\\
\bullet&\,(-2)\text{-curve (component of a reducible fiber of }\pi)
\end{align*}
\end{table}

Our second result answers our second question and consists in a precise description of the interplay between fibers of conic bundles on $X$ and the fiber configuration of the elliptic fibration. We use Kodaira's notation for the singular fibers of $\pi$ (see Theorem~\ref{thm:Kodaira_classification}).
\\ \\
\noindent\textbf{Theorem~\ref{thm:interplay_conic_bundle_elliptic_fibration}.} \textit{Let $X$ be a rational elliptic surface with elliptic fibration $\pi:X\to\P^1$. Then the following statements hold:}
\begin{enumerate}[a)]
\item \textit{$X$ admits a conic bundle with an $A_2$ fiber $\Leftrightarrow$ $\pi$ has positive generic rank and no {\normalfont $\text{III}^*$} fiber.}
\item \textit{$X$ admits a conic bundle with an $A_{n\geq 3}$ fiber $\Leftrightarrow$ $\pi$ has a reducible fiber distinct from {\normalfont $\text{II}^*$}.}
\item \textit{$X$ admits a conic bundle with a $D_3$ fiber $\Leftrightarrow$ $\pi$ has at least two reducible fibers.}
\item \textit{$X$ admits a conic bundle with a $D_{m\geq 4}$ fiber $\Leftrightarrow$ $\pi$ has a nonreduced fiber or a fiber {\normalfont $\text{I}_{n\geq 4}$}.}
\end{enumerate}

Theorems~\ref{thm:classification_conic_bundles} and \ref{thm:interplay_conic_bundle_elliptic_fibration} are proven in Chapter~\ref{ch:conic_bundles_on_RES} and, for the most part, are also the main results in \cite{ConicBundles}.
\newpage
\section{Intersection gaps}\
\indent We note that a complete proof of item a) in Theorem~\ref{thm:interplay_conic_bundle_elliptic_fibration} is not possible without a detailed study on how sections of $\pi:X\to\P^1$ intersect one another. Indeed, in order to answer precisely when $X$ admits a conic bundle with an $A_2$ fiber one must know precisely which rational elliptic surfaces admit sections $P_1,P_2\in E(K)$ such that the intersection number $P_1\cdot P_2$ is equal to $1$. This answer was not available at the time of \cite{ConicBundles}, hence only a partial version of item a) could be proven in that context and the complete one remained as a conjecture. We present our next investigation, which provides the tools to prove item a) and several other results.

We ask the following question: as $P_1,P_2$ run through $E(K)$, what values can $P_1\cdot P_2$ attain? On surfaces in general, the computation of intersection numbers of curves can be a very difficult problem. In our case, however, we are only concerned with sections of an elliptic surface, which has the additional advantage of being rational. The tool that allows us to attack this problem is the \textit{Mordell-Weil lattice}, a notion first introduced by Elkies \cite{Elkies90} and Shioda \cite{Shioda89, Shioda90}. It involves the definition of a $\Q$-valued pairing on $E(K)$, called the \textit{height pairing}, which induces a positive-definite lattice on the quotient $E(K)/E(K)_\text{tor}$, which is the Mordell-Weil lattice (Section~\ref{section:MW_lattice}). A key aspect of its construction is the connection with the Néron-Severi lattice, so that the height pairing and the intersection pairing of sections are strongly intertwined. Fortunately, the possibilities for the Mordell-Weil lattice on a rational elliptic surface have already been classified in \cite{OguisoShioda}, which gives us a convenient starting point.

Another aspect of this investigation is the connection with a classic theme in number theory, namely the representation of integers by positive-definite quadratic forms. Indeed, since $E(K)$ has rank $r$, its free part is generated by $r$ terms, so the height $h(P):=\langle P,P\rangle$ induces a positive-definite quadratic form on $r$ variables with coefficients in $\Q$. If $O\in E(K)$ is the neutral section and $R$ is the set of reducible fibers of $\pi$, then by the height formula (\ref{equation:height_formula_P})
$$h(P)=2+2(P\cdot O)-\sum_{v\in R}\text{contr}_v(P),$$
where the sum over $v$ is a rational number which can be easily estimated. By clearing denominators, we see that the possible values of $P\cdot O$ depend on a certain range of integers represented by a positive-definite quadratic form with coefficients in $\Z$. This point of view is explored in some parts
of Chapter~\ref{ch:gaps}, where we apply results such as the classical Lagrange four-square theorem \cite[\S 20.5]{HardyWright}, the counting of integers represented by a binary quadratic form \cite[p. 91]{Bernays} and the more recent Bhargava-Hanke's 290-theorem on universal quadratic forms \cite[Thm. 1]{BhargavaHanke}.

We say that $k\in\Z_{\geq 0}$ is a \textit{gap number} of $X$, or that or that $X$ has a $k$-\textit{gap} if there are no sections $P_1\cdot P_2\in E(K)$ such that $P_1\cdot P_2=k$. We try to answer under which conditions gap numbers exist, how they are distributed and try to identify them in some cases. Our first result states that gap numbers do not exist for a big enough Mordell-Weil rank.
\\ \\
\noindent\textbf{Theorem~\ref{thm:gap_free_r>=5}.} \textit{If $r\geq 5$, then $X$ has no gap numbers.}
\\ \\
\indent On the other hand, if the rank is low enough, then gap numbers occur with probability $1$.
\newpage
\noindent\textbf{Theorem~\ref{thm:gaps_probability_1_r=1,2}.} \textit{If $r\leq 2$, then the set of gap numbers $G:=\{k\in\Bbb{N}\mid k\text{ is a gap number of }X\}$ has density $1$ in $\Bbb{N}$, i.e.
$$\lim_{n\to\infty}\frac{\#G\cap\{1,...,n\}}{n}=1.$$}

\indent As to the explicit identification of gap numbers, we point out some cases where a complete identification is possible.
\\ \\
\noindent\textbf{Theorem~\ref{thm:identification_of_gaps_r=1}.} \textit{If $E(K)$ is torsion-free with rank $r=1$, then all the gap numbers of $X$ are described below, according to the lattice $T$ associated with the reducible fibers of $\pi$ (see Definition~\ref{def:lattice_T}).}
\begin{table}[h]
\begin{center}
\centering
\begin{tabular}{cc} 
\hline
\multirow{2}{*}{$T$} & \multirow{2}{*}{$\begin{matrix}k\text{ is a gap number}\Leftrightarrow \text{none of}\\ \text{the following are perfect squares}\end{matrix}$}\\ %& \multirow{2}{*}{first gap numbers}\\ 
& \\
\hline
\multirow{2}{*}{$E_7$} & \multirow{2}{*}{$k+1$, $4k+1$}\\ %& \multirow{2}{*}{$1,4$}\\ 
& \\
\hline
\multirow{2}{*}{$A_7$} & \multirow{2}{*}{$\frac{k+1}{4}$, $16k,...,16k+9$}\\ %& \multirow{2}{*}{$8,11$}\\ 
& \\
\hline
\multirow{2}{*}{$D_7$} & \multirow{2}{*}{$\frac{k+1}{2}$, $8k+1,...,8k+4$}\\% & \multirow{2}{*}{$2,5$}\\ 
& \\
\hline
\multirow{2}{*}{$A_6\oplus A_1$} & \multirow{2}{*}{$\frac{k+1}{7}$, $28k-3,...,28k+21$}\\% & \multirow{2}{*}{$12,16$}\\ 
& \\
\hline
\multirow{2}{*}{$E_6\oplus A_1$} & \multirow{2}{*}{$\frac{k+1}{3}$, $12k+1,...,12k+9$}\\% & \multirow{2}{*}{$3,7$}\\ 
& \\
\hline
\multirow{2}{*}{$D_5\oplus A_2$} & \multirow{2}{*}{$\frac{k+1}{6}$, $24k+1,...,24k+16$}\\% & \multirow{2}{*}{$6,11$}\\ 
& \\
\hline
\multirow{2}{*}{$A_4\oplus A_3$} & \multirow{2}{*}{$\frac{k+1}{10}$, $40k-4,...,40k+25$}\\% & \multirow{2}{*}{$16,20$}\\ 
& \\
\hline
\multirow{2}{*}{$A_4\oplus A_2\oplus A_1$} & \multirow{2}{*}{$\frac{k+1}{15}$, $60k-11,...,60k+45$}\\% & \multirow{2}{*}{$22,27$}\\ 
& \\
\hline
\end{tabular}\caption{Description of gap numbers when $E(K)$ is torsion-free with $r=1$.}\label{table:description_gaps_r=1}
\end{center}
\end{table}

\indent We conclude by fixing $k=1$ and identifying all rational elliptic surfaces with a $1$-gap. This is tantamount to describing when $X$ admits a conic bundle with an $A_2$ fiber, hence we solve our motivating problem of proving item a) in Theorem~\ref{thm:interplay_conic_bundle_elliptic_fibration}.
\\ \\
\noindent\textbf{Theorem~\ref{thm:surfaces_with_a_1-gap}.} \textit{$X$ has a $1$-gap if and only if $r=0$ or $r=1$ and $\pi$ has a} III$^*$ \textit{fiber.}
\\ \\
\indent The four theorems above are proven in Chapter~\ref{ch:gaps} and are also the main results in \cite{Gaps}.
\newpage

\section{Large rank jumps and the Hilbert property}\label{section:large_rank_jumps}\
\indent Let $k$ be a number field. We address a prominent theme in Diophantine Geometry, namely the variation of the Mordell-Weil rank of a fiber $\pi^{-1}(t)$ as $t$ varies in $\P^1_k$. In the search for elliptic curves with large rank, the use of this phenomenon has been a major source of examples, as done by Mestre \cite{Mestre}, Fermigier \cite{Fermigier92, Fermigier97}, Nagao \cite{Nagao92, Nagao93, Nagao94}, Martin-McMillen \cite{MartinMcMillen98, MartinMcMillen00} and ultimately Elkies \cite{Elkies06}, whose record of rank $\geq 28$ over $\Q$ still holds.

We begin by considering two important results; first by Néron \cite[Thm. 6]{Neron52}, which states that the Mordell-Weil rank $r_t$ (over $k$) of the fiber $\pi^{-1}(t)$ is greater or equal to the generic rank $r$ for all $t\in\P^1_k$ outside a thin subset of $\P^1_k$ (see Definition~\ref{def:thin_sets}); and another by Silverman \cite[Thm. C]{Silverman}, built on the first, stating that $r_t\geq r$ outside a set of points of bounded height. In this thesis we study the possibility of having $r_t>r$, in which case we say that the \textit{rank jumps}, and study the nature of subsets of $\P^1_k$ where rank jumps occur.

This problem has received attention not only in the case of rational elliptic surfaces \cite{Billard, Salgado12, LoughSalgado} but also of K3 surfaces \cite{Salgado15} and Abelian varieties \cite{HindrySalgado}. We note that in \cite{LoughSalgado} the authors show that under some conditions the set of fibers where the rank jumps of at least $2$, i.e. $\{t\in\P^1_k\mid r_t\geq r+2\}$ is not a thin set (Definition~\ref{def:thin_sets}), which calls attention to the nature of the rank jump set and brings back the notion of thin sets in Néron's specialization theorem.

Our strategy is to use Néron's geometric methods already applied in \cite{Shioda91, Salgado12, Salgado15, HindrySalgado, ColliotThelene20} to produce rank jumps and use ideas from \cite{LoughSalgado} to study the nature of the rank jump set. More precisely, we look for conditions to guarantee that rank jumps of at least 3 occur on a non-thin set. We show in fact that it suffices to require the existence of a certain genus zero fibration on $X$, which we call an \textit{RNRF-conic bundle} (see Definitions~\ref{def:conic_bundles} and \ref{def:RNRF}). The following theorem is, to our knowledge, the largest rank jump observed in this level of generality.
\\ \\
\textbf{Theorem~\ref{thm:rank_jump_3_times}.} \textit{If $X$ admits a \textit{RNRF}-conic bundle, then $\{t\in \P^1_k\mid r_t\geq r+3\}$ is not thin.}
\\ \\
\indent Theorem~\ref{thm:rank_jump_3_times} is proven in Chapter~\ref{ch:rank_jumps} and is also the main result in \cite{DiasCostaSalgado}.

\newpage
\chapter{Preliminaries}\label{ch:preliminaries}\
\indent Throughout the text, all surfaces are projective and smooth over a field $k$, which is either a number field or an algebraically closed field of arbitrary characteristic. The ground field $k$ is specified only when necessary. We reserve the letter $X$ to denote a geometrically rational elliptic surface with elliptic fibration $\pi:X\to \P^1$ (Definitions~\ref{def:elliptic_surfaces}, \ref{def:RES}). The letter $S$ denotes a surface in general, which may or may not be rational elliptic. 

For the general theory of elliptic surfaces we refer the reader to the classical books by Miranda \cite{Miranda}, Cossec and Dolgachev \cite[Ch. V]{Dolga}, Silverman \cite[Ch. III]{SilvermanAdv}, a survey paper by Schuett and Shioda \cite{Schuett-Shioda} and the more recent book by the same authors \cite[Ch. 5]{MWL}.

\section{Elliptic surfaces}\label{section:elliptic_surfaces}
\begin{defi}\label{def:elliptic_surfaces}
We call $S$ an \textit{elliptic surface} if there is smooth projective curve $C$ and a surjective morphism $f:S\to C$, called an \textit{elliptic fibration}, such that
\begin{enumerate}[i)]
\item The fiber $f^{-1}(t)$ is a smooth genus-1 curve for all but finitely many $t\in C$.
\item (existence of a section) There is a morphism $\sigma:C\to S$ such that $f\circ\sigma=\text{id}_C$, called a \textit{section}.
\item (relative minimality) No fiber of $f$ contains an exceptional curve in its support (i.e., a smooth rational curve with self-intersection $-1$).\label{item:relative_minimality}
\end{enumerate}
\end{defi}

\begin{remark}\label{remark:relative_minimality}
\normalfont Condition~\ref{item:relative_minimality}) can be understood as an extra hypothesis on $f$. For our purposes it is a natural one, as it assures that the fibers agree with Kodaira's classification (Theorem~\ref{thm:Kodaira_classification}) and that, for rational elliptic surfaces, the fibration is uniquely determined by the anticanonical system (Theorem~\ref{thm:RES_distinguished_properties}).  

%some convenient properties hold when the surface is rational (Theorem~\ref{thm:RES_from_pencil_of_cubics} and ) and that the fibration is uniquely determined by the anticanonical linear system .
\end{remark}

\begin{defi}
Let $f:S\to C$ be an elliptic fibration. If $\eta\in C$ is the generic point of $C$, we call $E:=f^{-1}(\eta)$ the \textit{generic fiber} of $\pi$. In particular, $E$ is an elliptic curve over the function field $K:=k(C)$ and the set $E(K)$ of $K$-points has a group structure, which we call the \textit{Mordell-Weil group} of $S$.
\end{defi}

As the generic fiber is an elliptic curve over $K:=k(C)$, the elliptic surface $S$ may be locally represented in the Weierstrass form, namely
\begin{equation}\label{eq:Weierstrass}
y^2+a_1xy+a_3=x^3+a_2x^2+a_4x+a_6,\text{ where }a_i\in K\text{ for all }i,
\end{equation}
which is an affine model of the generic fiber in $\A_K^2$.
\newpage

We also introduce a notion related to that of section.

\begin{defi}\label{def:multisection}
A curve $D\subset S$ is called a \textit{multisection of degree} $d$ when $f|_D:D\to C$ is a flat, finite morphism of degree $d$. In particular, a section of $f:S\to C$ is a multisection of degree $1$.
\end{defi}
\begin{remark}\label{remark:bisection}
\normalfont In Chapter~\ref{ch:rank_jumps} we deal with multisections of degree $2$, which we call \textit{bisections}.
\end{remark}

\subsection{Correspondence between sections, curves and points on the generic fiber}\label{subsection:correspondence_between_sections_and_points}\
\indent Let $f:S\to C$ be an elliptic fibration and $E$ the generic fiber of $f$, which is an elliptic curve over the function field $K:=k(C)$. We describe the natural correspondence between sections of $f$, certain curves on $S$ and $K$-points in $E$.

Given a section $\sigma:C\to S$, the curve $\sigma(C)$ on $S$ is isomorphic to $C$ via $\sigma$. Moreover $\sigma(C)$ meets the generic fiber $E$ at exactly one $K$-point, say $P\in E(K)$. 
\begin{figure}[h]
\begin{center}
\includegraphics[scale=0.65]{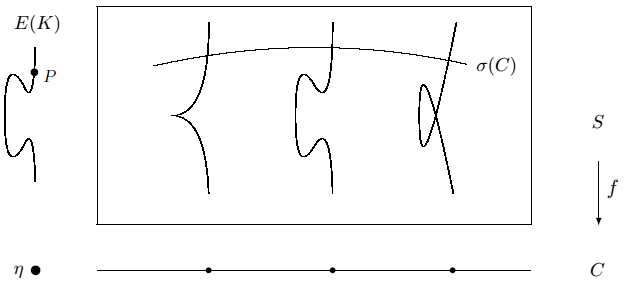}
\end{center}
\caption{Correspondence between sections $\sigma:C\to S$ and points $P\in E(K)$.}
\end{figure}

\indent Conversely, a point $P\in E(K)$ gives rise to a section in the following manner. Since $E$ is the generic fiber, $P\in E(K)$ corresponds to a point $P_v\in f^{-1}(v)$ for almost all $v\in C$. We take the scheme-theoretic closure of $\{P_v\}_v$ in $S$ and obtain a curve $\Gamma\subset S$. Because $C$ is smooth, the restriction $f|_\Gamma:\Gamma\to C$ is an isomorphism, which induces a section $\sigma:C\to S$ such that $\text{Im}(\sigma)=\Gamma$. This correspondence between sections and $K$-points in the generic fiber is, in fact, bijective.

\begin{prop}\label{prop:correspondence_sections_points}\cite[Prop. 5.4]{MWL}
The sections of $f:S\to C$ are in a natural bijection correspondence with the points in $E(K)$ via $\sigma\mapsto \sigma(C)\cap E(K)$, where $\sigma:C\to S$ is a section.
\end{prop}

\begin{remark}\label{remark:sections_as_curves}
\normalfont We also identify a section $\sigma:C\to S$ with the image $\sigma(C)$, which is a curve on $S$, hence calling it a \textit{section} as well. We note that when $S$ is rational, the sections correspond to the exceptional curves of $S$ (Theorem~\ref{thm:RES_distinguished_properties}).
\end{remark}

\subsection{Singular Fibers}\label{subsection:singular_fibers}\
\indent The singular fibers of an elliptic fibration play an important role in the study of elliptic surfaces. The first classification of singular fibers was made by Kodaira \cite{KodII} over the base field $\C$, followed by Tate \cite{Tate75}, who introduced a simplified classifying algorithm also valid over perfect fields. Although in Tate's algorithm one finds the same fiber types as Kodaira's, new fiber types may appear over non-perfect fields \cite{Szydlo}. As our base field $k$ is either a number field or algebraically closed, hence perfect, we may rely on the usual classification (Theorem~\ref{thm:Kodaira_classification}).
\\ \\
\noindent\textbf{Notation.} We follow Kodaira's notation \cite{KodII} for fiber types. A smooth fiber is labeled as $\text{I}_0$, while singular fibers receive one of the following labels: $\text{I}_n,\text{I}^*_n$ for some $n\geq 1$, $\text{II}, \text{II}^*,\text{III},\text{III}^*, \text{IV}, \text{IV}^*$, where $*$ indicates a nonreduced fiber. By Theorem~\ref{thm:Kodaira_classification}, every reducible fiber (i.e. all types except $\text{I}_1$, $\text{II}$) is associated with an ADE lattice (Figure~\ref{figure:ADE_lattices}). More precisely, the intersection graph of a reducible fiber forms an \textit{extended Dynkin diagram} \cite[I.4.7]{Humphreys} of type $\widetilde{A}_n$ ($n\geq 2$), $\widetilde{D}_k$ ($k\geq 4$) or $\widetilde{E}_\ell$ ($\ell=6,7,8$), where $\sim$ indicates that the Dynkin diagram for $A_n, D_k$ or $E_\ell$ is extended with one extra node. In Table~\ref{table:indices_reducible_fibers} the extra node corresponds to the neutral component $\Theta_0$.

\begin{teor}{\normalfont \cite[\S 6]{Tate75}}\label{thm:Kodaira_classification}
Let $f:S\to C$ be an elliptic fibration. If $F$ a singular fiber of $f$, then all possibilities for $F$ are listed below. 

\begin{center}
\includegraphics[scale=0.7]{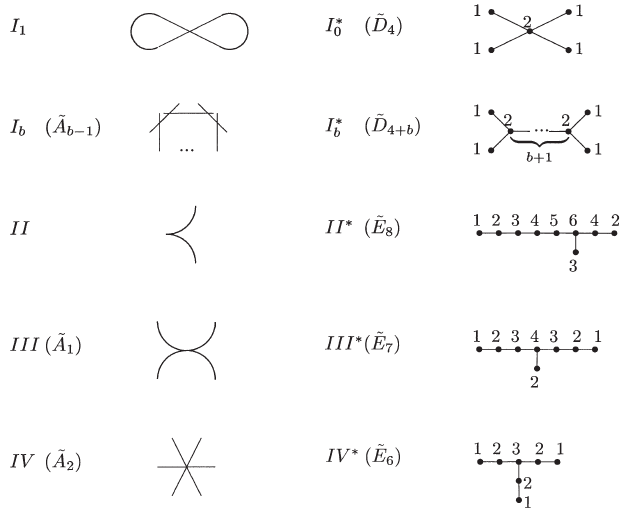}
\end{center}
Moreover, when $F$ is irreducible (namely of type $\text{I}_1$ or $\text{II}$), it is a singular, integral curve of arithmetic genus $1$. When $F$ is reducible, all members in its support are smooth, rational curves with selfintersection $-2$.
\end{teor}

\begin{remark}\label{remark:char_2_and_3}
\normalfont Tate's algorithm consists in identifying singular fibers by analyzing local Weierstrass forms. This local analysis can be pathological in $\text{char}(k)=2,3$, but even in these cases the possibilities for singular fibers are still the ones listed in Theorem~\ref{thm:Kodaira_classification} \cite[Section 6]{Tate75}. In our classification of conic bundles in Chapter~\ref{ch:conic_bundles_on_RES} we do not deal with local Weierstrass forms, but only with the numerical behavior of fibers as divisors, hence we are free to use Theorem~\ref{thm:Kodaira_classification} in any characteristic.
\end{remark}

\newpage

One way to detect singular fibers is by considering a local Weierstrass form (\ref{eq:Weierstrass}) and identifying the zeros of the discriminant, namely \cite{Tate75}
$$\Delta=-b_2^2b_8-8b_4^3-27b_6^2+9b_2b_4b_6,$$
where:
\begin{align*}
b_2&:=a_1^2+4a_2,\\
b_4&:=a_1a_3+2a_4,\\
b_6&:=a_3^2+4a_6,\\
b_8&:=a_1^2a_6-a_1a_3a_4+4a_2a_6+a_2a_3^2-a_4^2.
\end{align*}

\begin{remark}
\normalfont In $\text{char}(k)\neq 2,3$ a suitable change of coordinates yields a simpler Weierstrass form, namely $y^2=x^3+A(t)x+B(t)$ for some $A(t),B(t)\in k(C)$ whose discriminant is given by $\Delta(t)=-16\cdot(4A(t)^3+27B(t)^2)$ \cite[III.1]{SilvermanAdv}.
\end{remark}

Tate's algorithm for identifying fiber types was later refined by the Dokchitser brothers \cite{Dokchitser}, which reduced the identification to a simple inspection of the coefficients $a_i$ of the Weierstrass form (\ref{eq:Weierstrass}). More precisely, given a point $v\in C$ and its corresponding discrete valuation $v:k(C)^\times\to\Z$, the fiber type of $f^{-1}(v)$ is given by Table~\ref{table:Dokchitser} \cite[Thm. 1]{Dokchitser}.

\begin{table}[h]
\centering
\begin{tabular}{|c|c|c|c|c|c|c|c|c|} 
\hline
\multirow{1}{*}{} & \multirow{1}{*}{$\text{II}$} & \multirow{1}{*}{$\text{III}$} & \multirow{1}{*}{$\text{IV}$} & \multirow{1}{*}{$\text{I}^*_0$} & \multirow{1}{*}{$\text{I}^*_{n>0}$} & \multirow{1}{*}{$\text{IV}^*$} & \multirow{1}{*}{$\text{III}^*$} & \multirow{1}{*}{$\text{II}^*$}\\ 
\hline
\multirow{2}{*}{$\min_i\frac{v(a_i)}{i}$} & \multirow{2}{*}{$\frac{1}{6}$} & \multirow{2}{*}{$\frac{1}{4}$} & \multirow{2}{*}{$\frac{1}{3}$} & \multirow{2}{*}{$\frac{1}{2}$} & \multirow{2}{*}{$\frac{1}{2}$} & \multirow{2}{*}{$\frac{2}{3}$} & \multirow{2}{*}{$\frac{3}{4}$} & \multirow{2}{*}{$\frac{5}{6}$}\\ 
& & & & & & & &\\
\multirow{2}{*}{\tiny extra condition} & \multirow{2}{*}{} & \multirow{2}{*}{} & \multirow{2}{*}{\tiny $v(b_6)=2$} & \multirow{2}{*}{\tiny $v(d)=6$} & \multirow{2}{*}{\tiny $\begin{matrix}v(d)>6\\v(a_2^2-3a_4)=2\end{matrix}$ } & \multirow{2}{*}{\tiny $v(b_6)=4$} & \multirow{2}{*}{} & \multirow{2}{*}{}\\ 
& & & & & & & &\\
\hline
\end{tabular}\caption{Dokchitsers' refinement of Tate's algorithm.}\label{table:Dokchitser}
$$(\text{where }b_6:=a_3^2+4a_6=\text{Disc}(y^2+a_3y-a_6)\text{ and } d:=\text{Disc}(x^3+a_2x^2+a_4x+a_6))$$
\end{table}

We introduce some notation for the components of the reducible fibers. If $F_v:=f^{-1}(v)$ is a reducible fiber of the elliptic fibration $f:S\to C$, we write it as:
$$F_v=\sum_{i=0}^{m_v-1}\mu_{v,i}\Theta_{v,i}\,\text{ for each }v\in R,$$
\noindent where
\begin{align*}
m_v&:\text{ number of irreducible components of }F_v.\\
\Theta_{v,i}&:\text{ irreducible components with }0\leq i\leq m_v-1.\\
\mu_{v,i}&:\text{ multiplicity of }\Theta_{v,i}\text{ in }F_v.\\
m^{(1)}_v&:\text{ number of \textit{simple components} of }F_v,\text{ i.e. such that }\mu_{v,i}=1.\\
R&:\text{ set of reducible fibers of }f.
\end{align*}
\newpage
For a fixed $v$, we order the $i$-indices of the components $\Theta_{v,i}$ in the fashion of \cite{KodII} as indicated in Table~\ref{table:indices_reducible_fibers} (we include the multiplicities but omit the index $v$ for simplicity). We warn the reader that the $i$-indices in $\text{I}_n^*$ do not follow the usual order of the respective Dynkin diagram for the lattice $D_{n+4}$ (Figure~\ref{figure:ADE_lattices}).

\begin{table}[h]
\begin{center}
\centering
\begin{tabular}{c c} 
\multirow{4}{*}{\hfil II$^*\,(\widetilde{E}_8)$} & \multirow{4}{*}{\hfil\EEight}\\
& \\
& \\
& \\
\multirow{4}{*}{\hfil III$^*\,(\widetilde{E}_7)$} & \multirow{4}{*}{\hfil \ESeven}\\ 
& \\
& \\
& \\
\multirow{5}{*}{\hfil IV$^*\,(\widetilde{E}_6)$} & \multirow{5}{*}{\hfil \ESix}\\ 
& \\
& \\
& \\
& \\
\multirow{5}{*}{\hfil I$_n^*\,(\widetilde{D}_{n+4})$} & \multirow{5}{*}{\hfil \Dn}\\ 
& \\
& \\
& \\
& \\
\multirow{4}{*}{\hfil I$_n\,(\widetilde{A}_{n-1})$} & \multirow{4}{*}{\hfil \An}\\ 
& \\
& \\
& \\
\end{tabular}
\end{center}
\caption{Reducible fibers according to its Kodaira type and\\ its respective extended Dynkin diagram.}
\label{table:indices_reducible_fibers}
\end{table}	
\newpage

\subsection{Euler number}\
\indent We introduce the formula for the \textit{Euler number} (or \textit{Euler-Poincaré characteristic}) of $S$, denoted by $e(S)$. When working over the complex numbers, this corresponds to the topological Euler number. Over general fields, however, the Euler number is defined by the alternating sum of Betti numbers, i.e. dimensions of the $\ell$-adic étale cohomology \cite[\S 5.12]{MWL}.

For a fiber $F_v$ of $f:S\to C$ with $m_v$ components, we have
$$e(F_v)=\begin{cases}
0, & \text{ if }F_v\text{ is smooth;}\\
m_v, & \text{ if }F_v\text{ is singular of type I}_n;\\
m_v+1, & \text{ if }F_v\text{ is singular and not of type I}_n.\\
\end{cases}$$

It is worth noting that, when there is no wild ramification, the number $e(F_v)$ coincides with the valuation $v(\Delta(t))$ of the discriminant of the Weierstrass form \cite[\S 5.9]{MWL}. In particular, assuming $\text{char}(k)\neq 2,3$ is enough to avoid wild ramification \cite[\S 5.12]{MWL} and we obtain the following formula.

\begin{teor}{\cite[Prop. 5.16]{Dolga}}
Let $f:S\to C$ over $k$ with $\text{char}(k)\neq 2,3$. Then 
$$e(S)=\sum_{v\in C}e(F_v).$$
\end{teor}

\subsection{Base change}\label{subsection:base_change}\
\indent We explain the relationship between base change on elliptic surfaces and changes in the Mordell-Weil rank. We show how this is relevant to the study of rank jumps in Chapter~\ref{ch:rank_jumps}, particularly in the case of base changes of degree $2$.

\subsubsection*{Base change and Mordell-Weil rank}\
\indent In the study of an elliptic curve $E$ over a field $L$ it is a common operation to extend the base field, i.e. to consider a field extension $L'/L$ and analyze the elliptic curve $E$ over $L'$. In this case the set of $L'$-points on $E$ is at least as large as the set of $L$-points, which implies a possibly higher Mordell-Weil rank over $L'$ than over $L$.

In the case of an elliptic fibration $f:S\to C$, whose generic fiber $E$ is an elliptic curve over the function field $K:=k(C)$, this construction has a geometric counterpart. Indeed, a finite field extension $K'/K$ of degree $d$ corresponds to a finite morphism of curves $\varphi:C'\to C$ of degree $d$, and we obtain a commutative diagram from base change
\begin{displaymath}
\begin{tikzcd}
S'\arrow{r}{\varphi'}\arrow[swap]{d}{f'} & S\arrow{d}{f}\\
C'\arrow[swap]{r}{\varphi} & C
\end{tikzcd}
\end{displaymath} 

where $S':=S\times_CC'$ is a new elliptic surface with fibration $f':S'\to C'$. Notice in particular that the Mordell-Weil rank of $f'$ is greater or equal to the Mordell-Weil rank of $f$, since each section $\sigma:C\to S$ of $f$ has a natural corresponding section $(\sigma,\text{id}):C'=C\times_CC'\to S\times_CC'=S'$ of $f'$.

\subsubsection*{Base change and rank jumps}\
\indent If we are able to find a section of $f'$ distinct from those of the form $(\sigma,\text{id}):C'\to S'$, then there is a legitimate reason to ask whether $f'$ has strictly higher Mordell-Weil rank than $f$. It turns out that this new section does exists when $C'$ is a multisection of $f$, i.e. a curve $C'\subset S$ such that $f|_{C'}:C'\to C$ is a finite morphism. Indeed, the inclusion $\iota: C'\hookrightarrow S$ induces a section $(\iota,\text{id}):C'=C\times_CC'\to S'$ distinct from those of the form $(\sigma,\text{id})$.

As we now explain, the observation in the previous paragraph is crucial to our investigation to our investigation of \textit{rank jumps} (Chapter~\ref{ch:rank_jumps}), i.e. of the occurrence of fibers of $f$ whose Mordell-Weil rank is greater than the generic rank. 

Indeed, assume that the new section $(\iota,\text{id})$ is independent from the ones of the form $(\sigma,\text{id})$ in the Mordell-Weil group of $f':S'\to C'$. In particular $r'>r$, where $r,r'$ are the generic ranks of $f,f'$ respectively. Assuming moreover that $C'$ has infinitely many $k$-points, then Silverman's specialization theorem \cite[Thm. C]{Silverman} guarantees that $\text{rank }f'^{-1}(t)\geq r'$ for infinitely many $t\in C'(k)$. By construction, the fiber $f'^{-1}(t)$ corresponds to $f^{-1}(\varphi(t))$, hence $\text{rank }f^{-1}(\varphi(t))\geq r'>r$ for infinitely many $t\in C'(k)$.
In particular, rank jumps occur at infinitely many fibers of $f$.

\textbf{Conclusion:} if $C'$ is a multisection of $f:S\to C$ which contains infinitely many $k$-points and induces a new independent section of $f':S'\to C'$, then there are infinitely many fibers of $f$ for which the rank jumps.

\subsubsection*{Quadratic base change}\
\indent In Chapter~\ref{ch:rank_jumps} we address the problem of rank jumps on a rational elliptic surface $\pi:X\to\P^1$ over a number field $k$. The following result from \cite{SalgadoPhD} tells us of the possibility of having infinitely many curves $D\subset X$ inducing a new independent section of the base-changed fibration $\pi':X'\to D$.

\begin{teor}{\cite{SalgadoPhD}}\label{thm:pencil_general}
Let $\mathscr{L}$ be a pencil of curves on $X$ not all contained in a fiber of $\pi$. Then after base change under $\varphi:=\pi|_D:D\to\P^1$, the elliptic fibration $\pi':X'=X\times_{\P^1}D\to D$ has Mordell-Weil rank strictly greater than that of $\pi$ for all but finitely many $D\in\mathscr{L}$. 
\end{teor}

As mentioned earlier, in order to conclude that rank jumps occur at infinitely many fibers, we still need that there be some curve in $\mathscr{L}$ with infinitely many $k$-points. One way to obtain such $\mathscr{L}$ is by finding a pencil of genus $0$ curves over $k$ (a \textit{conic bundle}, as in Section~\ref{section:conic_bundles}). By Lemma~\ref{lemma:bisection}, this is tantamount to finding a bisection over $k$, i.e. a curve $D\subset X$ such that $\varphi:D\to\P^1$ is finite of degree $2$, in which case we perform a \textit{quadratic} base change.

\subsubsection*{Singular fibers after quadratic base change}\
\indent In Chapter~\ref{ch:rank_jumps}, more specifically in Lemma~\ref{lemma:againrational}, we need to describe how singular fibers of $\pi':X'\to D$ relate to those of $\pi:X\to\P^1$ in a quadratic base change. This information depends on how the degree $2$ morphism $\varphi=\pi|_D:D\to \P^1$ ramifies, as explained in what follows.

If $t\in \P^1$ is not a branch point of $\varphi:D\to\P^1$, then $\varphi^{-1}(t)=\{t_1,t_2\}$ and the fibers $\pi'^{-1}(t_1),\pi'^{-1}(t_2)$ are both isomorphic to $\pi^{-1}(t)$. On the other hand, if $t$ is a branch point of $\varphi$, then $\varphi^{-1}(t)=\{t_1\}$ and the fiber type of $\pi'^{-1}(t_1)$ is determined by the fiber type of $\pi^{-1}(t)$ according to Table~\ref{table:quadratic_base_change} below.
\newpage
\begin{table}[h]
\begin{center}
\begin{tabular}{ |c|c| } 
\hline
$\text{Fibre of } \pi$ & $\text{Fibre of } \pi'$\\
\hline
I$_n$, $n\geq 0$ & I$_{2n}$\\
\hline
I$^*_n$, $n\geq 0$ & I$_{2n}$\\ 
\hline
II & IV\\ 
\hline
II$^*$ & IV$^*$\\ 
\hline
III & I$_0^*$\\
\hline
III$^*$ & I$_0^*$\\
\hline
IV & IV$^*$\\
\hline
IV$^*$ & IV\\
\hline
\end{tabular}
\end{center}
\caption{Singular fibers above branch points in a quadratic base change \cite[VI.4.1]{Miranda}.}
\label{table:quadratic_base_change}
\end{table}

\noindent\textbf{Example.} Consider the rational elliptic surface $X: y^2=x^3+t$. By Dokchitsers' algorithm (Table~\ref{table:Dokchitser}), we identify the fiber types $\text{II},\text{II}^*$ at $t=0$ and $t=\infty$ respectively. Let $\varphi:\P^1\to\P^1$ be given by $(s:t)\mapsto (s^2:t^2)$, whose branch points are $t=0$ and $t=\infty$. The Weierstrass form of the base-changed surface is obtained by pull-back under $\varphi$, namely $X':y^2=x^3+t^2$. Again by Dokchitsers' algorithm, we identify types $\text{IV},\text{IV}^*$ at $t=0$ and $t=\infty$ respectively, which agrees with Table~\ref{table:quadratic_base_change}.

\section{Rational elliptic surfaces}\label{section:RES}\
\indent Among the many examples of nontrivial elliptic surfaces, the most easy to construct, describe and manipulate are \textit{rational} elliptic surfaces, which is the central object of this thesis. In the present section we make the distinction between \textit{rational} and \textit{geometrically rational}, explain how rational elliptic surfaces can be constructed from pencils of cubics on $\P^2$ and list some of their properties. 

\begin{defi}\label{def:RES}
We say that $S$ is a \textit{rational elliptic surface} if $S$ is birational to $\P^2_k$ and admits an elliptic fibration $f:S\to C$.
\end{defi}
\begin{remark}
\normalfont Since $S$ is rational, $C$ is isomorphic to $\P^1_k$ by Lüroth's theorem.
\end{remark}

\subsection{Rational vs. geometrically rational surfaces}\
\indent When dealing with surfaces over a non-algebraically closed field (in our case, a number field in Chapter~\ref{ch:rank_jumps}), we must distinguish between a \textit{rational} surface and a \textit{geometrically rational} surface. 

\begin{defi}\label{def:geometrically_rational}
A surface $S$ over $k$ is called \textit{geometrically rational} if $S\times_k\overline{k}$ is birational to $\P^2_{\overline{k}}$.
\end{defi}

In Chapter~\ref{ch:rank_jumps} we deal with elliptic surfaces when $k$ is a number field, in which case the following criterion can be used to determine whether a surface is geometrically rational.
\newpage
\begin{lemma}\label{lemma:euler_number_12}
Let $f:S\to C$ be a relatively minimal elliptic surface over a number field $k$ and let $e(S)$ be its Euler number. Then $S$ is geometrically rational if and only if $e(S)=12$.
\end{lemma}
\begin{proof}
Regardless of whether $S$ is geometrically rational, $K_S^2=0$ by relative minimality \cite[Prop. IX.3]{Beauville}, hence by Noether's formula $e(S)=12\chi(S)$. If $e(S)=12$ holds, then $\chi(S)=1$, which implies that $S$ is geometrically rational \cite[Lemma III.4.6]{Miranda}. Conversely, let $S$ be geometrically rational, i.e. $\overline{S}:=S\times_k\overline{k}$ is birational to $\P_{\overline{k}}^2$. As $\chi$ is a birational invariant, $\chi(\overline{S})=1$. From the inclusions $H^i(S,\mathcal{O}_S)\subset H^i(\overline{S},\mathcal{O}_{\overline{S}})$ for $i=1,2$ we conclude that $\chi(S)=1$, hence $e(S)=12$.
\end{proof}

\subsection{Construction from pencils of cubics}\label{subsection:RES_from_pencils_of_cubics}\
\indent We exhibit a standard method for constructing rational elliptic surfaces over an arbitrary field. We apply this in Chapter~\ref{ch:conic_bundles_on_RES} to produce examples of rational elliptic surfaces.

Let $L$ be a field and $F,G$ cubics on $\Bbb{P}_L^2$, at least one of them smooth. Assume moreover that $F,G$ only meet at $L$-rational points. The intersection $F\cap G$ has precisely $9$ points counted with multiplicity, and the pencil of cubics $\mathcal{P}:=\{sF+tG=0\mid (s:t)\in\P_L^1\}$ has $F\cap G$ as its base locus. Let $\phi:\P_L^2\dasharrow\P^1_L$ be the rational map associated to $\mathcal{P}$. The blowup $p:X\to \Bbb{P}_L^2$ of the $9$ points of the base locus resolves the indeterminacies of $\phi$ and we get an elliptic fibration $\pi:X\to \P^1_L$.

\begin{displaymath}
\begin{tikzcd}
X\arrow{r}{p}\arrow[bend right, swap]{rr}{\pi} & \P^2_L\arrow[dashed]{r}{\phi} &\P^1_L
\end{tikzcd}
\end{displaymath} 

By construction, $X$ is birational to $\P^2_L$. Moreover, the blowup of each base point induces a $(-1)$-curve on $X$, which is a section of $\pi:X\to\P^1$. If we assume moreover that $L$ is algebraically closed, it turns out that every rational elliptic surface can be obtained by this procedure.

\begin{teor}{\normalfont \cite[Theorem 5.6.1]{Dolga}}\label{thm:RES_from_pencil_of_cubics}
Let $X$ be a rational elliptic surface with elliptic fibration $\pi:X\to \P^1$ over an algebraically closed field. Then $X$ is isomorphic to the blowup of $\P^2$ at the base locus of a pencil of cubics.
\end{teor}

%\begin{remark}
%\normalfont A similar construction is possible when $F,G$ meet at $L'$-rational points for some finite extension $L'/L$ and at least one of these points is $L$-rational. After blowing up the base locus of $L'$-rational points, we obtain a surface $X$ defined over $L'$ (hence also defined over $L$)  In this case we blow up base locus to obtain an elliptic fibration $\pi:X\to\P^1$ which is defined over $L$ and admits a section over $L$, namely the $(-1)$-curve from the blowup of the $L$-rational point. Although $X$ is not necessarily birational to $\P^2_L$, it is birational to $\P^2_{L'}$, hence geometrically rational (Definition~\ref{def:geometrically_rational}).
%\end{remark}

\subsection{Properties of rational elliptic surfaces}\
\indent In addition to the property in Theorem~\ref{thm:RES_from_pencil_of_cubics}, we mention some other distinguished properties of rational elliptic surfaces. As these properties are geometric in nature, we assume $k$ is algebraically closed throughout the rest of this section.

\begin{teor}{\normalfont \cite[Section 8.2]{Schuett-Shioda}}\label{thm:RES_distinguished_properties}
Assume $k$ algebraically closed and let $X$ be a rational elliptic surface with elliptic fibration $\pi:X\to\P^1$. Then
\begin{enumerate}[i)]
\item $\chi(X)=1$, where $\chi(X):=h^0(X,\mathcal{O}_X)-h^1(X,\mathcal{O}_X)+h^2(X,\mathcal{O}_X)$.
\item $-K_X$ is linearly equivalent to any fiber of $\pi$. In particular, $-K_X$ is nef and $X$ admits precisely one elliptic fibration (namely, the one defined by the anticanonical system $|-K_X|$).
\item Every section of $\pi$ is an exceptional curve (smooth, rational curve with selfintersection $-1$).
\end{enumerate}	
\end{teor}

The following property is related to torsion sections of rational elliptic surfaces and plays an important role in the study of intersection numbers in Chapter~\ref{ch:gaps}.

\begin{teor}{\normalfont \cite[Lemma 1.1]{MirandaPersson}}\label{thm:torsion_sections_disjoint}
On a rational elliptic surface, $Q_1\cdot Q_2=0$ for any distinct $Q_1,Q_2\in E(K)_\text{tor}$. In particular, if $O$ is the neutral section, then $Q\cdot O=0$ for all $Q\in E(K)_\text{tor}\setminus\{O\}$.
\end{teor}

\begin{remark}
\normalfont Theorem~\ref{thm:torsion_sections_disjoint} holds for elliptic surfaces over $\Bbb{C}$ even without assuming the surface is rational. Over an arbitrary algebraically closed field, however, the rationality hypothesis is needed \normalfont \cite[Cor. 8.30]{MWL}.
\end{remark}

We include two more results. Lemma~\ref{lemma:detect_fibers} provides a simple test to detect fibers of the elliptic fibration, whereas Lemma~\ref{lemma:negative_curves} describes the negative curves on a rational elliptic surface. 

\begin{lemma}\label{lemma:detect_fibers}
Let $\pi:X\to\P^1$ be a rational elliptic surface and $E$ an integral curve in $X$. If $E\cdot K_X=0$, then $E$ is a component of a fiber of $\pi$. If moreover $E^2=0$, then $E$ is a fiber.
\end{lemma}
\begin{proof}
If $P\in E$, then the fiber $F:=\pi^{-1}(\pi(P))$ intersects $E$ at $P$, i.e. $E\cap F\neq\emptyset$. On the other hand, $-K_X$ is linearly equivalent to $F$ by Theorem~\ref{thm:RES_distinguished_properties}, so $E\cdot F=-E\cdot K_X=0$. Hence $E$ must be a component of $F$. Assuming moreover that $E^2=0$, we prove that $E=F$. In case $F$ is smooth, this is clear. So we assume $F$ is singular and analyze its Kodaira type according to Theorem~\ref{thm:Kodaira_classification}. Notice that $F$ is not reducible, otherwise $E$ would be a $(-2)$-curve, which contradicts $E^2=0$. Hence $F$ is either of type $\text{I}_1$ or $\text{II}$. In both cases $F$ is an integral curve, therefore $E=F$.
\end{proof}

\begin{lemma}\label{lemma:negative_curves}
Let $\pi:X\to\P^1$ be a rational elliptic surface. Every negative curve on $X$ is either a $(-1)$-curve (section of $\pi$) or a $(-2)$-curve (component from a reducible fiber of $\pi$). 
\end{lemma}
\begin{proof}
Let $E$ be any integral curve in $X$ with $E^2<0$. By Theorem~\ref{thm:RES_distinguished_properties}, $-K_X$ is nef and linearly equivalent to any fiber of $\pi$. So $E\cdot(-K_X)\geq 0$ and by adjunction \cite[I.15]{Beauville} $2p_a(E)-2=E^2+E\cdot K_X<0$, which only happens if $p_a(E)=0$. Consequently $E^2=-1$ or $-2$. In case $E^2=-2$ we have $E\cdot K_X=0$, so $E$ is fiber a component by Lemma~\ref{lemma:detect_fibers} and this fiber is reducible by Theorem~\ref{thm:Kodaira_classification}. If $E^2=-1$, by adjunction \cite[I.15]{Beauville} $E\cdot(-K_X)=1$. But $-K_X$ is lineary equivalent to any fiber, so $E$ meets a general fiber at one point, therefore $E$ is a section by Proposition~\ref{prop:correspondence_sections_points}.
\end{proof}

\begin{remark}
\normalfont The adjunction formula in \cite[I.15]{Beauville} is proven over the complex numbers. As the proof only relies on sheaf cohomology and Riemann-Roch \cite[V.1.6]{Hartshorne}, it can be extended to algebraically closed fields (in fact, to arbitrary fields).
\end{remark}

\section{Conic bundles}\label{section:conic_bundles}\
\indent We define conic bundles and explain where they appear the general theory of algebraic surfaces and in this thesis more specifically.

\begin{defi}\label{def:conic_bundles}
A \textit{conic bundle} on a surface $S$ is a surjective morphism onto a smooth curve $\varphi:S\to C$ whose general fiber is a smooth, irreducible curve of genus zero.
\end{defi}

\begin{remark}
\normalfont When $S$ is rational, $C$ is isomorphic to $\P^1$ by Lüroth's theorem. 
\end{remark}

\begin{remark}
\normalfont Definition~\ref{def:conic_bundles} is essentially identical to the ones in \cite[Definition~2.5]{LoughSalgado}, \cite[Definition 3.1]{GarbagnatiSalgado17}. The only difference here is that we allow the base field $k$ to be algebraically closed of any characteristic.
\end{remark}

A prominent case where conic bundles arise is in the classification of minimal models of $L$-rational surfaces, where $L$ is an arbitrary field. As shown by Iskovskikh \cite[Theorem 1]{Isko}, if $S$ is an $L$-minimal rational surface, then $S$ is either (a) $\P^2_L$, (b) a quadric on $\P^3_L$ with $\text{Pic}(S)\simeq \Z$, (c) a del Pezzo surface with $\text{Pic}(S)=\Z\langle-K_S\rangle$, or (d) a surface with $\text{Pic}(S)\simeq \Z^2$ admitting a conic bundle whose singular fibers are isomorphic to a pair of lines meeting at a point. The latter was named a \textit{standard conic bundle} by Manin and Tsfasman \cite[Subsection 2.2]{ManTsfa}. We remark that the notion of standard conic bundle is often extended to higher dimension and plays a role in the classification of threefolds over $\Bbb{C}$ in the Minimal Model Program (see, for example, \cite{Sarkisov}, \cite{IskoRationalityProblem} and the survey \cite{Prokho}).

\begin{remark}
\normalfont Definition~\ref{def:conic_bundles} is more general than that of \textit{standard conic bundle} in \cite[Subsection 2.2]{ManTsfa}, in which every singular fiber is isomorphic to a pair of lines meeting at a point. In our definition, such a singular fiber is of one among four possible types, namely of type $A_2$ as in Theorem~\ref{thm:classification_conic_bundles}.
\end{remark}

As explained in Subsection~\ref{subsection:base_change}, the existence of a pencil of genus $0$ curves (i.e. a conic bundle) on a geometrically rational elliptic surface $\pi:X\to \P^1$ is relevant to the study of rank jumps in Chapter~\ref{ch:rank_jumps}. The following result by \cite{LoughSalgado} tells us that, over a number field $k$, the presence of a conic bundle over $k$ is in fact equivalent the the existence of a bisection over $k$ (Remark~\ref{remark:bisection}).

\begin{lemma}\label{lemma:bisection}{\cite[Lemma 2.6]{LoughSalgado}}
Let $\pi:X\rightarrow \P^1$ be a geometrically rational elliptic surface over a number field $k$. Then $X$ admits a conic bundle over $k$ if and only if
$\pi$ admits a bisection of arithmetic genus $0$ over $k$.
\end{lemma} 

%\begin{proof}
%Let $\varphi:X\to\P^1$ be a conic bundle. If $D$ is a general fiber of $\varphi$ (i.e. a smooth curve of genus zero), then $D^2=0$ and by the adjunction formula \cite[I.15]{Beauville}, $D\cdot K_X=-2$. Since $-K_X$ is linearly equivalent to any fiber $F$ of $\pi$ by Theorem~\ref{thm:RES_distinguished_properties}, then $D$ meets a general $F$ at two points, hence $\pi|_D:D\to\P^1$ has degree $2$, as desired. Conversely, let $D\subset X$ be a bisection of genus zero. As $\pi|_D:D\to \P^1$ has degree 2, then $D$ meets a general fiber $F$ of $\pi$ at two points, hence $D\cdot (-K_X)=D\cdot F=2$. Moreover, $D^2=0$ by adjunction \cite[I.15]{Beauville}. We claim that $|D|$ is a base point free pencil and that the induced morphism $\varphi_{|D|}:X\to\P^1$ is a conic bundle. First, we consider the exact sequence
%\begin{equation}\label{equation:exact_sequence}
%0\to \O_X\to \O_X(D)\to \O_D(D)\to 0.
%\end{equation}

%The fact that $X$ is geometrically rational implies that $h^1(X,\O_X)=0$, so applying cohomology to (\ref{equation:exact_sequence}) gives $h^0(X,D)=2$, hence $|D|$ is a pencil. Moreover, $D^2=0$ implies that $|D|$ is base point free, hence induces a morphism $\varphi_{|D|}:X\to\P^1$. As $D$ is smooth of genus $0$, we use Bertini's theorem \cite[Cor. 10.9]{Hartshorne} to conclude that the general fiber of $\varphi_{|D|}$ is smooth of genus $0$.
%\end{proof}

\section{Thin sets and the Hilbert property}\label{section:thin_sets_Hilbert_property}\
\indent We define thin sets and the Hilbert property following Serre \cite[\S 3.1]{Serre}. We provide a brief explaination of why these concepts are relevant and how they appear in this thesis.

\begin{defi}\label{def:thin_sets}
Let $V$ be a variety over a field $k$. A subset $T\subseteq V(k)$ is called \textit{thin} if it is a finite union of subsets which are either contained in a proper closed subvariety of $V$; or in the image $f(W(k))$ where $f:W\to V$ is a generically finite dominant morphism of degree at least $2$ and $W$ is an integral variety over $k$.
\end{defi}

\begin{defi}
A variety $V$ over a field $k$ is said to \textit{satisfy the Hilbert property} if the set $V(k)$ is not thin.
\end{defi}

A classical example of a variety satisfying the Hilbert property is $\P^n_k$ when $k$ is a numbert field, as a consequence of Hilbert's irreducibility theorem \cite[Thm. 3.4.1]{Serre}. A still unproven conjecture by Colliot-Thélène states that, moreover, every unirational variety over a number field satisfies the Hilbert property \cite[Thm. 3.5.7 and Conjecture 3.5.8]{Serre}. If proven true, this is enough to give an affirmative answer to the inverse Galois problem \cite[Thm. 3.5.9]{Serre}.

In Chapter~\ref{ch:rank_jumps} we are concerned specifically with non-thin subsets of $\P^1$, which is the base of the rational elliptic fibration $\pi:X\to\P^1$. We note that any proper Zariski-closed subset of $\P^1$ is finite and that in order to prove that an infinite subset $T\subset \mathbb{P}^1(k)$ is not thin, it suffices to show that for any finite number of finite covers $\varphi_i:Y_i\to \P^1$ there is a point $P\in T$ such that $P\notin \varphi_i(Y_i)(k)\subset \P^1(k)$ for each $i$. 
\newpage
\section{General facts about the geometry of surfaces}\
\indent We list some elementary results involving divisors, fibers of morphisms $S\to\P^1$ and linear systems, which are needed in the classification of conic bundle fibers in Chapter~\ref{ch:conic_bundles_on_RES}. We remark that these results are valid for smooth surfaces in general, hence do not depend on the theory of elliptic surfaces.

\begin{lemma}\label{lemma:critical_divisors_are_nef}
Let $D$ be an effective divisor on a surface $S$ such that $D\cdot C=0$ for all {\normalfont $C\in\text{Supp }D$}. Then $D$ is nef and $D^2=0$.
\end{lemma}
\begin{proof}
To see that $D$ is nef, just notice that $D\cdot C'\geq 0$ for every $C'\notin\text{Supp }D$. To prove that $D^2=0$ let $D=\sum_in_iC_i$, where each $C_i$ is in $\text{Supp }D$. Then $D^2=D\cdot(\sum_in_iC_i)=\sum_in_iD\cdot C_i=0$.
\end{proof}

\indent The most natural application of Lemma~\ref{lemma:critical_divisors_are_nef} is when $D$ is a fiber of a surjective morphism $S\to\P^1$. In the next lemma, we explore some properties of such morphisms needed in Chapter~\ref{ch:conic_bundles_on_RES}.

\begin{lemma}\label{lemma:fibers}
Let $F$ be a fiber of a surjective morphism $f:S\to\P^1$. Then the following hold:
\begin{enumerate}[a)]
\item $F\cdot C=0$ for all {\normalfont $C\in\text{Supp }F$}.
\item If $F$ is connected and $E$ is a divisor such that {\normalfont $\text{Supp }E\subset\text{Supp }F$}, then $E^2\leq 0$. Moreover $E^2=0$ if and only if $E=rF$ for some $r\in\Q$.
\item If $F_1,...,F_n$ are the connected components of $F$, then each $F_i$ is nef with $F_i^2=0$ and $F_i\cdot K_S\in 2\Z$.
\end{enumerate}
\end{lemma}

\begin{proof}
a) Taking an arbitrary $C\in\text{Supp }F$ and another fiber $F'\neq F$, we have $F\cdot C=F'\cdot C=0$.
\\ \\
\indent b) This is Zariski's lemma \cite[Ch. 6, Lemma 6]{Peters}.
\\ \\
\indent c) Fix $i$. To prove that $F_i$ is nef and $F_i^2=0$ we show that $F_i\cdot C_i=0$ for any $C_i\in\text{Supp }F_i$ then apply Lemma~\ref{lemma:critical_divisors_are_nef}. Indeed, if $j\neq i$, the components of $F_i,F_j$ are disjoint, so $F_j\cdot C_i=0$. Since $F$ is a fiber and $C_i\in\text{Supp }F_i\subset\text{Supp }F$, then $F\cdot C_i=0$ by~a). Hence $F_i\cdot C_i=(F_1+...+F_n)\cdot C_i=F\cdot C_i=0$, as desired. For the last part, by Riemann-Roch $F_i\cdot K_S=2(\chi(S)-\chi(S,F_i))\in 2\Z$.
\end{proof}

\indent We end this section with a simple observation on linear systems without fixed components.
\begin{lemma}\label{fixed_component}
Let $E,E'$ be effective divisors such that $E'\leq E$ and that the linear systems $|E|,|E'|$ have the same dimension. If $|E|$ has no fixed components, then $E'=E$.
\end{lemma}
\begin{proof}
As $E'\leq E$, we have an inclusion of vector spaces $H^0(S,E')\subset H^0(S,E)$. By hypothesis, their dimensions coincide, therefore $H^0(S,E')=H^0(S,E)$. Hence
\begin{align*}
|E|&=\{E+\text{div}(f)\mid f\in H^0(S,E)\}\\
&=\{E'+\text{div}(f)+(E-E')\mid f\in H^0(S,E')\}\\
&=|E'|+(E-E').
\end{align*}

Assuming $|E|$ has no fixed components, we must have $E-E'=0$.
\end{proof}

\section{Lattices}\
\indent We stablish some basic terminology regarding lattices, which we adopt throughout the text, and present ADE lattices, which are the most important for us when dealing with Mordell-Weil lattices in Section~\ref{section:MW_lattice}.

\begin{defi}
A \textit{lattice} is a pair $(L,\langle\cdot,\cdot\rangle)$, where $L$ is a $\Bbb{Z}$-module and $\langle\cdot,\cdot\rangle:L\times L\to\Bbb{R}$ a non-degenerate symmetric bilinear pairing. We say that the lattice is 
\begin{enumerate}[i)]
\item \textit{positive-definite} (resp. \textit{negative-definite}) when $\langle x,x\rangle>0$ (resp. $\langle x,x\rangle<0$) for all $x\in L\setminus\{0\}$.
\item an \textit{integer} lattice when $\langle x,y\rangle\in\Bbb{Z}$ for all $x,y\in L$.
\item an \textit{even} lattice when $\langle x,x\rangle\in 2\Bbb{Z}$ for all $x\in L$.
\end{enumerate}
\end{defi}

\begin{defi}
The \textit{dual} lattice of an integral latice $(L,\langle\cdot,\cdot\rangle)$ is a sublattice of $L\otimes_\Z\Q$ given by
$$L^*:=\{x\otimes r\in L\otimes_\Bbb{Z}\Bbb{Q}\mid r\cdot \langle x,y\rangle\in\Bbb{Z}\text{ for all }y\in L\},$$
with a symmetric non-degenerate bilinear pairing naturally defined by:
$$\langle x\otimes r,x'\otimes r'\rangle:=r\cdot r'\cdot \langle x,x'\rangle\text{ for all }x,x'\in L,\,r,r'\in\Bbb{Q}.$$
In particular, there is a natural embedding $L\hookrightarrow L^*$ via $x\mapsto x\otimes 1$.
\end{defi}

\begin{defi}
We use $L^{-}$ to denote the \textit{opposite} lattice of $(L,\langle\cdot,\cdot\rangle)$, i.e. the lattice defined by the same $\Z$-module $L$ with the opposite pairing $-\langle\cdot,\cdot\rangle$.
\end{defi}

\begin{defi}
Let $(L,\langle\cdot,\cdot\rangle)$ be a lattice generated by $\alpha_1,...,\alpha_r\in L$. We define the \textit{determinant} of $L$ as the determinant of the Gram matrix
$$\det L:=\det(\langle\alpha_i,\alpha_j\rangle)_{i,j}$$
\end{defi}
\begin{remark}
\normalfont The determinant does not depend on the choice of the generators. Indeed, if $\beta_1,...,\beta_r$ generate $L$, then the base-change matrix $U$ is an invertible matrix with integer coefficients. In particular, $\det U=\pm 1$, hence $\det(\langle\beta_i,\beta_j\rangle)_{i,j}=(\det U)^2\det(\langle\alpha_i,\alpha_j\rangle)_{i,j}=\det(\langle\alpha_i,\alpha_j\rangle)_{i,j}$.
\end{remark}

We introduce the notion of \textit{root lattices}, which is central to many apparently unrelated areas of mathematics such as combinatorics, singular theory, Lie algebras and, in our case, Mordell-Weil lattices. 

\begin{defi}
A positive-definite (resp. negative-definite) integer lattice $L$ is called a \textit{root lattice} if it is generated by elements $x\in L$ such that $\langle x,x\rangle=2$ (resp. $\langle x,x\rangle=-2$). Such generators are called the \textit{roots} of $L$. 
\end{defi}

We define the fundamental root lattices, namely the ADE lattices.

\begin{defi}\label{def:ADE_lattices}
A positive-definite root lattice $L$ of rank $r$ is said to be of type $A_r$ $(r\geq 1)$, $D_r$ $(r\geq 4)$ or $E_r$ $(r=6,7,8)$ if it is generated by some set of roots $\alpha_1,...,\alpha_r$ such that $\langle\alpha_i,\alpha_j\rangle=0$ for every $i<j$ except in the following cases:
\begin{align*}
(A_r)&\,\,\,\langle \alpha_i,\alpha_j\rangle=-1\Leftrightarrow i+1=j.\\
(D_r)&\,\,\,\langle \alpha_i,\alpha_j\rangle=-1\Leftrightarrow i+1=j<r,\,\text{ or }i=r-2,\,j=r.\\
(E_r)&\,\,\,\langle \alpha_i,\alpha_j\rangle=-1\Leftrightarrow i+1=j<r,\,\text{ or }i=3,\,j=r.\\
\end{align*}
\end{defi}

\begin{figure}[h]
\begin{center}
\includegraphics[scale=0.65]{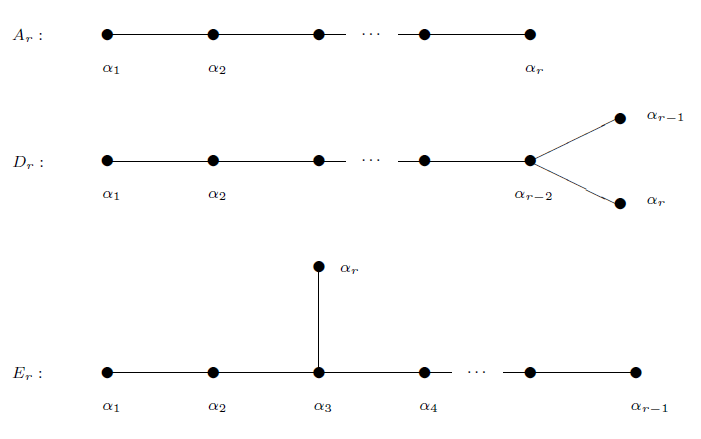}
\end{center}
\caption{Dynkin diagrams associated with each ADE lattice.}\label{figure:ADE_lattices}
\end{figure}

\begin{remark}
\normalfont ADE lattices may also be defined as negative-definite, in which case all signs should be inverted in Definition~\ref{def:ADE_lattices}.
\end{remark}

The importance of ADE lattices is explained by the following result.

\begin{teor}{\cite[Thm 1.2]{Ebeling}}
Let $L$ be a positive-definite (or negative-definite) integer lattice. Then $L$ is a root lattice if and only if it is isometric to a direct sum of ADE lattices.
\end{teor}

\section{Mordell-Weil lattices}\label{section:MW_lattice}\
\indent We introduce the Mordell-Weil lattice, which is our central tool for Chapter~\ref{ch:gaps}. Its notion was first put forward by Elkies and Shioda independently in \cite{Elkies90}, \cite{Shioda90}, \cite{Shioda89} and consists of a lattice structure on the Mordell-Weil group $E(K)$ with an explicit connection with the Néron-Severi lattice, in a sense made precise in this seciton.

Historically, many attempts have been made, in different contexts, to define a bilinear pairing on $E(K)$. This dates back to \cite{Manin}, \cite{BS-D} and \cite{Tate66}, whose objects of interest were, respectively, heights on Abelian varieties, the Birch and Swinnerton-Dyer conjecture and the Tate conjecture. The further step of connecting a lattice structure on $E(K)$ with the Néron-Severi lattice was made by \cite{CoxZ}, whose original goal was to determine whether a given set of sections can generate $E(K)$. The idea of a Mordell-Weil lattice was already implicit in the latter, but a precise definition only appears a few years later in \cite{Elkies90}, \cite{Shioda90}, \cite{Shioda89}. 

Since then, many applications have been found for Mordell-Weil lattices of both arithmetic and geometric interest (see, for instance, \cite[Chapter 10]{MWL}). In this thesis we make a specific use of it, namely, we take advantage of the fact that Mordell-Weil lattices have been explicitly classified for rational elliptic surfaces \cite{OguisoShioda} in order to obtain information about intersection numbers of sections in Chapter~\ref{ch:gaps}.

\newpage

This section is dedicated to introducing the Mordell-Weil lattice, and is organized as follows. First we make some observations about the Néron-Severi lattice $\NS(S)$ and define two sublattices, namely the lattice $T$ (Definition~\ref{def:lattice_T}) and the trivial lattice (Definition~\ref{def:trivial_lattice}). We proceed with the construction of the height pairing, which leads to the definition of the Mordell-Weil lattice. We also present the explicit formula for computing the height pairing and introduce the narrow Mordell-Weil lattice.

We begin with a peculiar property of elliptic surfaces $f:S\to C$ in general, namely that numerical and algebraic equivalences coincide on $S$. By this feature, we may consider the Néron-Severi group $\NS(S)$ equipped with the intersection form, called the \textit{Néron-Severi lattice}.

\begin{teor}{\normalfont \cite[Thm. 3.1]{Shioda90}}\label{thm:algebraic_and_numerical_equivalences}
On an elliptic surface $S$, algebraic and numerical equivalences coincide, i.e. {\normalfont $D_1,D_2\in\text{Div}(S)$} are equivalent in {\normalfont $\NS(S)$} if and only if $D_1\cdot D=D_2\cdot D$ for every {\normalfont $D\in\Div(S)$}. Equivalently stated, {\normalfont $\NS(S)$} is torsion-free.
\end{teor}

\begin{remark}\label{remark:it_makes_sense_to_write_NS_times_Q}
\normalfont The fact that $\NS(S)$ is torsion-free allows us to extend scalars to $\Q$, i.e. to define $\NS(S)_{\Q}:=\NS(S)\otimes_{\Z}\Q$ without annihilating elements. 
\end{remark}

We define two important sublattices of $\NS(S)$, both of which contain information about the reducible fibers of $f:S\to C$. We use the notation from Subsection~\ref{subsection:singular_fibers}. 

\begin{defi}\label{def:lattice_T}
For each $v\in R$ we define the lattices $T_v$ and $T$ as
$$T_v:=\Z\langle \Theta_{v,i}\mid i> 0\rangle,$$
$$T:=\bigoplus_{v\in R}T_v.$$
\end{defi}
\begin{remark}
\normalfont By Kodaira's classification (Theorem~\ref{thm:Kodaira_classification}), each $T_v$ can be represented by a Dynkin diagram. More precisely, the opposite lattice $T_v^{-}$ is isomorphic to an ADE lattice (Figure~\ref{figure:ADE_lattices}). 
\end{remark}

The next lattice we define plays an important role in the construction of the Mordell-Weil lattice. Like the lattice $T$, it contains information about the reducible fibers, only with the addition of the neutral section $O$ and the neutral components $\Theta_{v,0}$ for $v\in R$ in the generator set.

\begin{defi}\label{def:trivial_lattice}
If $O\in E(K)$ be the neutral section, we define the \textit{trivial lattice} as
$$\text{Triv}(S):=\Z\langle O,\Theta_{v,i}\mid v\in R,\,i\geq 0\rangle.$$
\end{defi}

At this point we call the reader's attention to the distinction between the group operations in the Mordell-Weil group and in the Néron-Severi group. As explained in Subsection~\ref{subsection:correspondence_between_sections_and_points}, sections $P_1,P_2\in E(K)$ may also be seen as curves, hence defining classes $\overline{P}_1,\overline{P}_2\in \NS(S)$. However,
$$\overline{P}_1+\overline{P}_2\neq\overline{P_1+P_2}.$$

The following theorem states that, modulo the trivial lattice $\text{Triv}(S)$, these different operations actually induce a group isomorphism.

\begin{teor}{\normalfont \cite[Thm 1.3]{Shioda90}}\label{thm:shioda_tate}
The following map is a group isomorphism:
\begin{align*}
E(K)&\to {\normalfont\NS(S)/\text{Triv}(S)}\\
P&\mapsto \overline{P}\,{\normalfont \text{mod Triv}(S)}
\end{align*}
\end{teor}

We proceed to define a bilinear pairing on $E(K)$. We note that, in order to do it, we cannot use the intersection pairing directly, which only defines a lattice on $\NS(S)$ but not on $E(K)$. This difficulty is overcome by a geometric construction which involves the orthogonal projection with respect to $\text{Triv}(S)$ in the $\Q$-vectors space $\NS(S)_{\Q}:=\NS(S)\otimes_{\Z}\Q$.

\begin{lemma}{\normalfont \cite[Lemma 8.1]{Shioda90}}\label{lemma:orthogonal_projection}
There is a unique function {\normalfont $\varphi:E(K)\to \NS(X)_\Q$} with the following properties:
\begin{enumerate}[(i)]
\item {\normalfont $\varphi(P)\equiv P\text{ mod }\text{Triv}(S)_\Q$} for all $P\in E(K)$.
\item {\normalfont $\varphi(P)\perp\text{Triv}(S)$} for all $P\in E(K)$.
\end{enumerate}
Moreover, the class {\normalfont $\varphi(P)\in\NS(S)_\Q$} is represented by the following $\Q$-divisor:
$$D_P:=\overline{P}-\overline{O}-(P\cdot O+\chi(S))F+\sum_{v\in R}(\Theta_{v,1},...,\Theta_{v,m_v-1})(-A_v^{-1})\left(\begin{matrix}P\cdot \Theta_{v,1}\\\vdots\\P\cdot \Theta_{v,m_v-1}\end{matrix}\right),$$
where $F$ is a fiber of $\pi$ and $A_v$ is the matrix $(\Theta_{v,i}\cdot \Theta_{v,j})_{1\leq i,j\leq m_v-1}$.
\end{lemma}

In fact, the map $\varphi$ is a group morphism and induces an embedding $E(K)/E(K)_\text{tor}\hookrightarrow\NS(S)_\Q$.

\begin{lemma}{\normalfont \cite[Lemma 8.2]{Shioda90}}
The map {\normalfont $\varphi:E(K)\to\NS(X)_\Q$} is a group homomorphism. Moreover, {\normalfont $\ker\varphi=E(K)_\text{tor}$}, hence $\varphi$ induces an injective morphism {\normalfont $E(K)/E(K)_\text{tor}\hookrightarrow\NS(S)$}.
\end{lemma}

%\begin{lemma}
%Let $L$ be the essential lattice of $X$ and $m_v^{(1)}$ the number of components with multiplicity $1$ of the reducible fiber $\pi^{-1}(v)$. Then
%$$\text{Im}(\varphi)\subset\frac{1}{m}L^{-},$$
%where $m$ is the least common multiple of all $m_v^{(1)}$ as $v$ runs through $R$. Moreover, $\varphi$ induces an injection
%$$E(K)/E(K)_\text{tor}\hookrightarrow \frac{1}{m}L^{-}\subset L^{-}\otimes_\Z\Q$$
%\end{lemma}

We are ready to define a bilinear pairing on $E(K)$, which we call the \textit{height pairing}.

\begin{teor}{\normalfont \cite[Thm 6.20]{Shioda90}}
We define the \textit{height pairing} as
\begin{align*}
\langle\cdot,\cdot\rangle:E(K)\times E(K)&\to\Q\\
P,Q&\mapsto-\varphi(P)\cdot\varphi(Q),
\end{align*}
which induces a positive-definite pairing on {\normalfont $E(K)/E(K)_\text{tor}$}. The lattice {\normalfont ($E(K)/E(K)_\text{tor},\langle\cdot,\cdot\rangle)$} is called the Mordell-Weil lattice. 
\end{teor}

Once the height pairing is constructed, we also define the \textit{height} of a section and the \textit{minimal norm} of the Mordell-Weil lattice.
\begin{defi}\label{def:height_minimal_norm}
The \textit{height} of a section and the \textit{minimal norm} of $E(K)$ are defined as
$$h(P):=\langle P,P\rangle\text{ for each }P\in E(K),$$
$$\mu:=\min\{h(P)>0\mid P\in E(K)\}.$$
\end{defi}

A convenient feature of the height pairing is that it can be computed explicitly. Before we introduce the explicit formula, we define one of the terms it involves, namely the \textit{local contribution}. 
\begin{defi}
Let $v\in R$ and $P,Q\in E(K)$. If $P,Q$ meet the fiber $F_v$ at $\Theta_{v,i},\Theta_{v,j}$ respectively, we define the \textit{local contribution} as
\begin{align*}
\text{contr}_v(P,Q)&:=
\begin{cases}
(-A_v^{-1})_{i,j}&\text{if }i,j\geq 1,\\
0&\text{otherwise.}
\end{cases}\\
\text{contr}_v(P)&:=\text{contr}_v(P,P).
\end{align*}
where $(-A_v^{-1})_{i,j}$ is the $(i,j)$-entry of the matrix $-A_v^{-1}$ in Lemma~\ref{lemma:orthogonal_projection}. 
\end{defi}

The explicit values for the local contributions are in Table~\ref{table:local_contributions} \cite[Table 6.1]{MWL}. 

\begin{table}[h]
\begin{center}
\centering
\begin{tabular}{c|c|c|c|c|c|c} 
\hline
\multirow{2}{*}{\hfil $\text{Type of }F_v$} & \multirow{2}{*}{\hfil $\text{III}$} & \multirow{2}{*}{\hfil $\text{III}^*$} & \multirow{2}{*}{\hfil $\text{IV}$} & \multirow{2}{*}{\hfil $\text{IV}^*$} & \multirow{2}{*}{\hfil $\text{I}_n$} & \multirow{2}{*}{\hfil $\text{I}_n^*$}\\ 
& & & & & &\\
\hline
\multirow{2}{*}{\hfil $T_v$} & \multirow{2}{*}{\hfil $A_1$} & \multirow{2}{*}{\hfil $E_7$} & \multirow{2}{*}{\hfil $A_2$} & \multirow{2}{*}{\hfil $E_6$} & \multirow{2}{*}{\hfil $A_{n-1}$} & \multirow{2}{*}{\hfil $D_{n+4}$}\\ 
& & & & & &\\
\hline
\multirow{3}{*}{\hfil $\text{contr}_v(P)$} & \multirow{3}{*}{\hfil $\frac{1}{2}$} & \multirow{3}{*}{\hfil $\frac{3}{2}$} & \multirow{3}{*}{\hfil $\frac{2}{3}$} & \multirow{3}{*}{\hfil $\frac{4}{3}$} & \multirow{3}{*}{\hfil $\frac{i(n-i)}{n}$} & 	\multirow{3}{*}{\hfil $\begin{cases}1 & (i=1)\\1+\frac{n}{4} & (i>1)\end{cases}$}\\
& & & & & &\\
& & & & & &\\
\hline
\multirow{3}{*}{\hfil $\text{contr}_v(P,Q)$} & \multirow{3}{*}{\hfil $\text{-}$} & \multirow{3}{*}{\hfil $\text{-}$} & \multirow{3}{*}{\hfil $\frac{1}{3}$} & \multirow{3}{*}{\hfil $\frac{2}{3}$} & \multirow{3}{*}{\hfil $\frac{i(n-j)}{n}$} & \multirow{3}{*}{\hfil $\begin{cases}\frac{1}{2} & (i=1)\\\frac{1}{2}+\frac{n}{4} & (i>1)\end{cases}$}\\
& & & & & &\\
& & & & & &\\
\hline
\end{tabular}
\caption{Local contributions assuming $F_v$ meets $P$ at $\Theta_{v,i}$ and $Q$ at $\Theta_{v,j}$ with $i<j$.}\label{table:local_contributions}
\end{center}
\end{table}
\newpage
We are ready to present the explicit formula for the height pairing, called the \textit{height formula}.

\begin{teor}{\normalfont \cite[Thm. 8.6]{Shioda90}}\label{thm:height_formula}
Let $P,Q\in E(K)$ and $O\in E(K)$ the neutral section. Then
\begin{equation}\label{equation:height_formula_PQ}
{\normalfont \langle P,Q\rangle=\chi(S)+P\cdot O+Q\cdot O-P\cdot Q-\sum_{v\in R}\text{contr}_v(P,Q),}
\end{equation}
\begin{equation}\label{equation:height_formula_P}
{\normalfont h(P)=2\chi(S)+2(P\cdot O)-\sum_{v\in R}\text{contr}_v(P).}
\end{equation}
\end{teor}

\begin{remark}
\normalfont In this thesis we focus on rational elliptic surfaces, in which case $\chi(S)=1$ in equations~\ref{equation:height_formula_PQ} and \ref{equation:height_formula_P} by Theorem~\ref{thm:RES_distinguished_properties}.
\end{remark}

An important sublattice of $E(K)$ is the \textit{narrow Mordell-Weil lattice} $E(K)^0$, defined as
\begin{align*}
E(K)^0&:= \{P\in E(K)\mid P\text{ intersects }\Theta_{v,0}\text{ for all }v\in R\}\\
&=\{P\in E(K)\mid \text{contr}_v(P)=0\text{ for all } v\in R\}.
\end{align*}

As a subgroup, $E(K)^0$ is torsion-free; as a sublattice, it is a positive-definite even integral lattice with finite index in $E(K)$ \cite[Thm. 6.44]{MWL}. The importance of the narrow lattice can be explained by its considerable size as a sublattice and by the easiness to compute the height pairing on it, since all contribution terms vanish. A complete classification of the lattices $E(K)$ and $E(K)^0$ on rational elliptic surfaces is found in \cite[Main Thm.]{OguisoShioda}.

\newpage

\section{Bounds $c_\text{max}, c_\text{min}$ for the contribution term}\label{section:bounds}\
\indent We define the bounds $c_\text{max},c_\text{min}$ for the contribution term $\sum_v\text{contr}_v(P)$ and state some simple facts about them in the case of rational elliptic surfaces. We also provide an example to illustrate how they are computed.

\indent In our investigation of intersection numbers in Chapter~\ref{ch:gaps}, the need for these bounds arise naturally. Indeed, suppose we are given a section $P\in E(K)$ whose height $h(P)$ is known and we want to determine $P\cdot O$. In case $P\in E(K)^0$ we have a direct answer, namely $P\cdot O=h(P)/2-1$ by the height formula (\ref{equation:height_formula_P}). However, if $P\notin E(K)^0$ the computation of $P\cdot O$ depends on the contribution term $c_P:=\sum_v\text{contr}_v(P)$, which by Table~\ref{table:local_contributions} depends on how $P$ intersects the reducible fibers of $\pi$. Usually we do not have this intersection data at hand, hence an estimate for $c_P$ becomes imperative.

\begin{defi} Assuming $R\neq\emptyset$, we define
\begin{align*}
c_\text{max}&:=\sum_{v\in R}\max\{\text{contr}_v(P)\mid P\in E(K)\},\\
c_\text{min}&:=\min\left\{\text{contr}_v(P)>0\mid P\in E(K), v\in R\right\}.
\end{align*}
\end{defi}

\begin{remark}
\normalfont In a rational elliptic surface, $R=\emptyset$ only occurs when the Mordell-Weil rank is $r=8$ (No. 1 in Table~\ref{table:MWL_data}). In this case $E(K)^0=E(K)$ and $\sum_v\text{contr}_v(P)=0$ $\forall P\in E(K)$, hence we adopt the convention $c_\text{max}=c_\text{min}=0$.
\end{remark}
\begin{remark}
\normalfont We use $c_\text{max},c_\text{min}$ as bounds for $c_P:=\sum_{v}\text{contr}_v(P)$. For our purposes it is not necessary to know whether $c_P$ actually attains one of these bounds for some $P$, therefore $c_\text{max},c_\text{min}$ should be understood as hypothetical values.
\end{remark}

We state some facts about $c_\text{max},c_\text{min}$.

\begin{lemma}\label{lemma:bounds_are_actually_bounds}
Let $X$ be a rational elliptic surface with Mordell-Weil rank $r\geq 1$. Then
\begin{enumerate}[i)]
\item $c_\text{min}>0$ if $R\neq \emptyset$.
\item $c_\text{max}<4$.
\item ${\normalfont c_\text{min}\leq \sum_{v\in R}\text{contr}_v(P)\leq c_\text{max}}$ $\forall P\notin E(K)^0$. For $P\in E(K)^0$, only the second inequality holds.
\item If {\normalfont $\sum_{v\in R}\text{contr}_v(P)=c_\text{min}$}, then {\normalfont $\text{contr}_{v'}(P)=c_\text{min}$} for some $v'$ and {\normalfont $\text{contr}_v(P)=0$} for $v\neq v'$.
\end{enumerate}

\end{lemma}

\begin{proof}
Item i) is immediate from the definition of $c_\text{min}$. For ii) it is enough to check the values of $c_\text{max}$ directly in Table~\ref{table:MWL_data}. For iii), the second inequality follows from the definition of $c_\text{max}$ and clearly holds for any $P\in E(K)$. If $P\notin E(K)^0$, then $c_P:=\sum_v\text{contr}_v(P)>0$, so $\text{contr}_{v_0}(P)>0$ for some $v_0$. Therefore $c_P\geq \text{contr}_{v_0}(P)\geq c_\text{min}$.

At last, we prove iv). If $R=\emptyset$, then $c_\text{min}=0$ and the claim is trivial. Hence let $R\neq\emptyset$ and $\sum_v\text{contr}_v(P)=c_\text{min}$. Assume by contradiction that there are $v_1\neq v_2$ such that $\text{contr}_{v_i}(P)>0$ for $i=1,2$. By definition of $c_\text{min}$, we have $c_\text{min}\leq\text{contr}_{v_i}(P)$ for $i=1,2$, thus
$$c_\text{min}=\sum_v\text{contr}_v(P)\geq \text{contr}_{v_1}(P)+\text{contr}_{v_2}(P)\geq 2c_\text{min},$$
\noindent which is absurd because $c_\text{min}>0$ by i). Hence there is only one $v'$ with $\text{contr}_{v'}(P)>0$, while $\text{contr}_v(P)=0$ for all $v\neq v'$. In particular, $\text{contr}_{v'}(P)=c_\text{min}$.
\end{proof}

\noindent\textbf{Explicit computation.} Once we know the lattice $T$ associated with the reducible fibers of $\pi$, the computation of $c_\text{max},c_\text{min}$ is simple. For a fixed $v\in R$, the extreme values for the local contribution $\text{contr}_v(P)$ are given in Table~\ref{table:extreme_local_contributions}, which is derived from Table~\ref{table:local_contributions}. We provide an example to illustrate this computation. 

\begin{table}[h]
\begin{center}
\centering
\begin{tabular}{c|c|c} 
\hline
\multirow{2}{*}{\hfil $T_v$} & \multirow{2}{*}{\hfil $\max\{\text{contr}_v(P)\mid P\in E(K)\}$} & \multirow{2}{*}{\hfil $\min\{\text{contr}_v(P)>0\mid P\in E(K)\}$}\\ 
&\\
\hline
\multirow{2}{*}{\hfil $A_{n-1}$} & \multirow{2}{*}{\hfil $\frac{\ell(n-\ell)}{n}$, where $\ell:=\left\lfloor\frac{n}{2}\right\rfloor$} & \multirow{2}{*}{\hfil $\frac{n-1}{n}$}\\ 
&\\
\hline
\multirow{2}{*}{\hfil $D_{n+4}$} & \multirow{2}{*}{\hfil $1+\frac{n}{4}$} & \multirow{2}{*}{\hfil $1$}\\ 
&\\
\hline
\multirow{2}{*}{\hfil $E_6$} & \multirow{2}{*}{\hfil $\frac{4}{3}$} & \multirow{2}{*}{\hfil $\frac{4}{3}$}\\ 
&\\
\hline
\multirow{2}{*}{\hfil $E_7$} & \multirow{2}{*}{\hfil $\frac{3}{2}$} & \multirow{2}{*}{\hfil $\frac{3}{2}$}\\ 
&\\
\hline
\end{tabular}
\caption{Extreme values of the local contribution $\text{contr}_v(P)$.}\label{table:extreme_local_contributions}
\end{center}
\end{table}

\noindent\textbf{Example:} Assume $\pi$ has fiber configuration $(\text{I}_4,\text{IV},\text{III},\text{I}_1)$. The reducible fibers are $\text{I}_4,\text{IV},\text{III}$, so $T=A_3\oplus A_2\oplus A_1$. By Table~\ref{table:extreme_local_contributions}, the maximal contributions for $A_3,A_2,A_1$ are $\frac{2\cdot 2}{4}=1$, $\frac{2}{3}$, $\frac{1}{2}$ respectively. The minimal positive contributions are $\frac{1\cdot 3}{4}=\frac{3}{4}$, $\frac{2}{3}$, $\frac{1}{2}$ respectively. Hence
\begin{align*}
c_\text{max}&=1+\frac{2}{3}+\frac{1}{2}=\frac{13}{6},\\
c_\text{min}&=\min\left\{\frac{3}{4},\frac{2}{3},\frac{1}{2}\right\}=\frac{1}{2}.
\end{align*}

\section{The difference $\Delta=c_\text{max}-c_\text{min}$}\
\indent We explain why the value of $\Delta:=c_\text{max}-c_\text{min}$ is relevant to the investigation in Chapter~\ref{ch:gaps}, specially in Section~\ref{section:sufficient_conditions_Delta<=2}. For rational elliptic surfaces, we verify that $\Delta<2$ in most cases and identify the exceptional ones in Table~\ref{table:Delta=2} and Table~\ref{table:Delta>2}.

\indent As noted in Section~\ref{section:bounds}, in case $P\notin E(K)^0$ and $h(P)$ is known, the difficulty of determining $P\cdot O$ lies in the contribution term $c_P:=\sum_v\text{contr}_v(P)$. In particular, the range of possible values for $c_P$ determines the possibilities for $P\cdot O$. This range is measured by the difference
$$\Delta:=c_\text{max}-c_\text{min}.$$

Hence a smaller $\Delta$ means a better control over the intersection number $P\cdot O$, which is why $\Delta$ plays an important role in determining possible intersection numbers. In Section~\ref{section:necessary_and_sufficient_Delta<=2} we assume $\Delta\leq 2$ and state necessary and sufficient conditions for having a pair $P_1,P_2$ such that $P_1\cdot P_2=k$ for a given $k\geq 0$. If however $\Delta>2$, the existence of such a pair is not guaranteed a priori, so a case-by-case treatment is needed. Fortunately by Lemma~\ref{lemma:cases_where_Delta>=2} the case $\Delta>2$ is rare. 

\newpage

\begin{lemma}\label{lemma:cases_where_Delta>=2}
Let $X$ be a rational elliptic surface with Mordell-Weil rank $r\geq 1$. The only cases with $\Delta=2$ and $\Delta>2$ are in Table \ref{table:Delta=2} and \ref{table:Delta>2} respectively. In particular we have $\Delta<2$ whenever $E(K)$ is torsion-free.

\begin{table}[h]
\begin{center}
\centering
\begin{tabular}{ccccc}
No. & $T$ & $E(K)$ & $c_\text{max}$ & $c_\text{min}$\\ 
\hline
\multirow{2}{*}{\hfil 24} & \multirow{2}{*}{\hfil $A_1^{\oplus 5}$} & \multirow{2}{*}{\hfil ${A_1^*}^{\oplus 3}\oplus\Bbb{Z}/2\Bbb{Z}$} & \multirow{2}{*}{\hfil $\frac{5}{2}$} & \multirow{2}{*}{\hfil $\frac{1}{2}$}\\
&\\
\multirow{2}{*}{\hfil 38} & \multirow{2}{*}{\hfil $A_3\oplus A_1^{\oplus 3}$} & \multirow{2}{*}{\hfil $A_1^*\oplus\langle 1/4\rangle\oplus\Bbb{Z}/2\Bbb{Z}$} & \multirow{2}{*}{\hfil $\frac{5}{2}$} & \multirow{2}{*}{\hfil $\frac{1}{2}$}\\
&\\
\multirow{2}{*}{\hfil 53} & \multirow{2}{*}{\hfil $A_5\oplus A_1^{\oplus 2}$} & \multirow{2}{*}{\hfil $\langle 1/6\rangle\oplus \Bbb{Z}/2\Bbb{Z}$} & \multirow{2}{*}{\hfil $\frac{5}{2}$} & \multirow{2}{*}{\hfil $\frac{1}{2}$}\\
&\\
\multirow{2}{*}{\hfil 57} & \multirow{2}{*}{\hfil $D_4\oplus A_1^{\oplus 3}$} & \multirow{2}{*}{\hfil $A_1^*\oplus(\Bbb{Z}/2\Bbb{Z})^{\oplus 2}$} & \multirow{2}{*}{\hfil $\frac{5}{2}$} & \multirow{2}{*}{\hfil $\frac{1}{2}$}\\
&\\
\multirow{2}{*}{\hfil 58} & \multirow{2}{*}{\hfil $A_3^{\oplus 2}\oplus A_1$} & \multirow{2}{*}{\hfil $A_1^*\oplus\Bbb{Z}/4\Bbb{Z}$} & \multirow{2}{*}{\hfil $\frac{5}{2}$} & \multirow{2}{*}{\hfil $\frac{1}{2}$} \\
&\\
\multirow{2}{*}{\hfil 61} & \multirow{2}{*}{\hfil $A_2^{\oplus 3}\oplus A_1$} & \multirow{2}{*}{\hfil  $\langle 1/6\rangle\oplus \Bbb{Z}/3\Bbb{Z}$} & \multirow{2}{*}{\hfil $\frac{5}{2}$} & \multirow{2}{*}{\hfil $\frac{1}{2}$}\\
&\\
\hline
\end{tabular}
\caption{Cases with $\Delta=2$}\label{table:Delta=2}
\end{center}
\end{table}

\begin{table}[h]
\begin{center}
\centering
\begin{tabular}{cccccc}
No. & $T$ & $E(K)$ & $c_\text{max}$ & $c_\text{min}$ & $\Delta$\\ 
\hline
\multirow{3}{*}{\hfil 41} & \multirow{3}{*}{\hfil $A_2\oplus A_1^{\oplus 4}$} & \multirow{3}{*}{\hfil $\frac{1}{6}\left(\begin{matrix}2 & 1\\1 & 2\end{matrix}\right)\oplus\Bbb{Z}/2\Bbb{Z}$} & \multirow{3}{*}{\hfil $\frac{8}{3}$} & \multirow{3}{*}{\hfil $\frac{1}{2}$} & \multirow{3}{*}{\hfil $\frac{13}{6}$}\\
&\\
&\\
\multirow{2}{*}{\hfil 42} & \multirow{2}{*}{\hfil $A_1^{\oplus 6}$} & \multirow{2}{*}{\hfil ${A_1^*}^{\oplus 2}\oplus (\Bbb{Z}/2\Bbb{Z})^{\oplus 2}$} & \multirow{2}{*}{\hfil $3$} & \multirow{2}{*}{\hfil $\frac{1}{2}$} & \multirow{2}{*}{\hfil $\frac{5}{2}$}\\
&\\
\multirow{2}{*}{\hfil 59} & \multirow{2}{*}{\hfil $A_3\oplus A_2\oplus A_1^{\oplus 2}$} & \multirow{2}{*}{\hfil $\langle 1/12\rangle\oplus \Bbb{Z}/2\Bbb{Z}$} & \multirow{2}{*}{\hfil $\frac{8}{3}$} & \multirow{2}{*}{\hfil $\frac{1}{2}$} & \multirow{2}{*}{\hfil $\frac{13}{6}$}\\
&\\
\multirow{2}{*}{\hfil 60} & \multirow{2}{*}{\hfil $A_3\oplus A_1^{\oplus 4}$} & \multirow{2}{*}{\hfil  $\langle 1/4\rangle\oplus (\Bbb{Z}/2\Bbb{Z})^{\oplus 2}$} & \multirow{2}{*}{\hfil $3$} & \multirow{2}{*}{\hfil $\frac{1}{2}$} & \multirow{2}{*}{\hfil $\frac{5}{2}$}\\
&\\
\hline
\end{tabular}
\caption{Cases with $\Delta>2$}\label{table:Delta>2}
\end{center}
\end{table}
\end{lemma}

\begin{proof}
By searching Table~\ref{table:MWL_data} for all cases with $\Delta=2$ and $\Delta>2$, we obtain Table~\ref{table:Delta=2} and Table~\ref{table:Delta>2} respectively. Notice in particular that in both tables the torsion part of $E(K)$ is always nontrivial. Consequently, if $E(K)$ is torsion-free, then $\Delta<2$.
\end{proof}

\section{The quadratic form $Q_X$}\label{subsection:presenting_Q_X}\
\indent We define the positive-definite quadratic form with integer coefficients $Q_X$ derived from the height pairing. The relevance of $Q_X$ is due to the fact that some conditions for having $P_1\cdot P_2=k$ for some $P_1,P_2\in E(K)$ can be stated in terms of what integers can be represented by $Q_X$ (see Corollary~\ref{coro:necessary_conditions_Q_X} and Proposition~\ref{prop:summary_of_sufficient_conditions}).

We define $Q_X$ simply by clearing denominators of the rational quadratic form induced by the height pairing; the only question is how to find a scale factor that works in every case. More precisely, if $E(K)$ has rank $r\geq 1$ and $P_1,...,P_r$ are generators of its free part, then $q(x_1,...,x_r):=h(x_1P_1+...+x_rP_r)$ is a quadratic form with coefficients in $\Q$; we define $Q_X$ by multiplying $q$ by some integer $d>0$ so as to produce coefficients in $\Z$. We show that $d$ may always be chosen as the determinant of the narrow lattice $E(K)^0$.

\begin{defi}
Let $X$ with $r\geq 1$. Let $P_1,...,P_r$ be generators of the free part of $E(K)$. Define
$$Q_X(x_1,...,x_r):=(\det E(K)^0)\cdot h(x_1P_1+...+x_rP_r).$$
\end{defi}
We check that the matrix representing $Q_X$ has entries in $\Z$, therefore $Q_X$ has coefficients in $\Z$.

\begin{lemma}\label{lemma:adjugate}
Let $A$ be the matrix representing the quadratic form $Q_X$, i.e. $Q(x_1,...,x_r)=x^tAx$, where $x:=(x_1,...,x_r)^t$. Then $A$ has integer entries. In particular, $Q_X$ has integer coefficients.
\end{lemma}

\begin{proof}
Let $P_1,...,P_r$ be generators of the free part of $E(K)$ and let $L:=E(K)^0$. The free part of $E(K)$ is isomorphic to the dual lattice $L^*$ \cite[Main Thm.]{OguisoShioda}, so we may find generators $P_1^0,...,P_r^0$ of $L$ such that the Gram matrix $B^0:=(\langle P_i^0,P_j^0\rangle)_{i,j}$ of $L$ is the inverse of the Gram matrix $B:=(\langle P_i,P_j\rangle)_{i,j}$ of $L^*$.

We claim that $Q_X$ is represented by the adjugate matrix of $B^0$, i.e. the matrix $\text{adj}(B^0)$ such that $B^0\cdot\text{adj}(B^0)=(\det B^0)\cdot I_r$, where $I_r$ is the $r\times r$ identity matrix. Indeed, by construction $B$ represents the quadratic form $h(x_1P_1+...+x_rP_r)$, therefore\begin{align*}
Q_X(x_1,...,x_r)&=(\det E(K)^0)\cdot h(x_1P_1+...+x_rP_r)\\
&=(\det B^0)\cdot x^tBx\\
&=(\det B^0)\cdot x^t(B^0)^{-1}x\\
&=x^t\text{adj}(B^0)x,
\end{align*}

as claimed. To prove that $A:=\text{adj}(B^0)$ has integer coefficients, notice that the Gram matrix $B^0$ of $L=E(K)^0$ has integer coefficients (as $E(K)^0$ is an even lattice), then so does $A$.
\end{proof}

\indent We close this section with a simple consequence of the definition of $Q_X$.

\begin{lemma}\label{lemma:Q_X_represents_dm} 
If $h(P)=m$ for some $P\in E(K)$, then $Q_X$ represents $d\cdot m$, where $d:=\det E(K)^0$.
\end{lemma}

\begin{proof}
Let $P_1,...,P_r$ be generators for the free part of $E(K)$. Let $P=a_1P_1+...+a_rP_r+Q$, where $a_i\in\Z$ and $Q$ is a torsion element (possibly zero). Since torsion sections do not contribute to the height pairing, then $h(P-Q)=h(P)=m$. Hence 
\begin{align*}
Q_X(a_1,...,a_r)&=d\cdot h(a_1P_1+...+a_rP_r)\\
&=d\cdot h(P-Q)\\
&=d\cdot m.
\end{align*}
\end{proof}

\newpage

\chapter{Conic bundles on rational elliptic surfaces}\label{ch:conic_bundles_on_RES}\
\indent In Section~\ref{section:conic_bundles} we explain how conic bundles appear in the general theory of algebraic surfaces. In this chapter we focus on  conic bundles on a rational elliptic surface, which is motivated by results and techniques from \cite{Salgado12}, \cite{LoughSalgado}, \cite{GarbagnatiSalgado17}, \cite{GarbagnatiSalgado20} and \cite{ArtebaniGarbagnatiLaface}. In \cite{Salgado12}, \cite{LoughSalgado} the existence of conic bundles on $X$ is used as a central condition for having rank jumps, in a sense made precise in Subsection~\ref{subsection:base_change} and specially in Chapter~\ref{ch:rank_jumps}. The relevance of conic bundles in \cite{GarbagnatiSalgado17}, \cite{GarbagnatiSalgado20} and \cite{ArtebaniGarbagnatiLaface} is due to other reasons, which we now briefly explain.

Over an algebraically closed field of characteristic zero, the existence of conic bundles on $X$ is used in \cite{GarbagnatiSalgado17}, \cite{GarbagnatiSalgado20} in order to classify elliptic fibrations on K3 surfaces that are quadratic covers of $X$. More precisely, given a degree two morphism $f:\P^1\to\P^1$ ramified away from nonreduced fibers of $\pi$, the induced K3 surface is $X':=X\times_f\P^1$. The base change also gives rise to an elliptic fibration $\pi':X'\to\P^1$ and a degree two map $f':X'\to X$. By composition with $f'$, every conic bundle on $X$ induces a genus $1$ fibration on $X'$, in which case we may obtain elliptic fibrations distinct from $\pi$.

In \cite{ArtebaniGarbagnatiLaface} conic bundles are also useful, although in a very different context. In this case the base field is $\C$ and the goal is to find generators for the Cox ring $\mathcal{R}(X):=\bigoplus_{[D]}H^0(X,D)$, where $[D]$ runs through $\text{Pic}(X)$. Given a rational elliptic fibration $\pi:X\to\P^1$, the ring $\mathcal{R}(X)$ is finitely generated if and only if $X$ is a Mori Dream Space \cite[Proposition 2.9]{HuKeel}, which in turn is equivalent to $\pi$ having generic rank zero \cite[Corollary 5.4]{ArtebaniLaface}. Assuming this is the case, the authors show that in many configurations of $\pi$ each minimal set of generators of $\mathcal{R}(X)$ must contain an element $g\in H^0(X,D)$, where $D$ is a fiber of a conic bundle on $X$, whose possibilities are explicitly described.

We take these instances as motivators for a detailed study of conic bundles, which is the theme of this chapter. Here we consider a rational elliptic surface $\pi:X\to\P^1$ over an arbitrary algebraically closed field and proceed by the following plan. In Section~\ref{section:numerical_characterization} we characterize conic bundles in terms of certain Néron-Severi classes and deduce some geometric properties using this point of view. By applying these results in Section~\ref{section:classification_conic_bundles}, we completely describe the possible types of conic bundle fibers. Section~\ref{section:conic_vs_elliptic} is dedicated to the study of how the fiber configuration of $\pi$ interferes with the possible fiber configuration of conic bundles on $X$. Finally in Section~\ref{section:examples_conic_bundles} we present a method to construct conic bundles and produce some examples to illustrate our results.

\section{Numerical characterization of conic bundles}\label{section:numerical_characterization}\
\indent We give a characterization of conic bundles on a rational elliptic surface $X$ from a numerical standpoint. The motivation for this approach comes from the following. Let $\varphi:X\to\P^1$ be a conic bundle and let $C$ be a general fiber of $\varphi$, hence a smooth, irreducible curve of genus zero. Clearly $C$ is a nef divisor with $C^2=0$ and, by adjunction, $C\cdot(-K_X)=2$. These three numerical properties are enough to prove that $|C|$ is a base point free pencil and consequently the induced morphism $\varphi_{|C|}:X\to\P^1$ gives $\varphi$ itself.

Conversely, let $D$ be a nef divisor with $D^2=0$ and $D\cdot(-K_X)=2$. Since numerical and algebraic equivalence coincide by Theorem~\ref{thm:algebraic_and_numerical_equivalences}, it makes sense to consider the class $[D]\in\text{NS}(X)$. The natural question is whether $[D]$ induces a conic bundle on $X$. The answer is yes, moreover there is a natural correspondence between such classes and conic bundles (Theorem~\ref{thm:correspondence}), which is the central result of this section.

In order to prove this correspondence we need a numerical analysis of a given class $[D]\in\text{NS}(X)$ so that we can deduce geometric properties of the induced morphism $\varphi_{|D|}:X\to\P^1$, such as connectivity of fibers (Proposition~\ref{prop:connectedness}) and composition of their support (Proposition~\ref{prop:components_conic_bundle_class}). These properties are also crucial to the classification of fibers in Section~\ref{section:classification_conic_bundles}.

\begin{defi}\label{def:conic_class}
A class $[D]\in\text{NS}(X)$ is called a \textit{conic class} when
\begin{enumerate}[i)]
\item $D$ is nef.
\item $D^2=0$.
\item $D\cdot (-K_X)=2$.
\end{enumerate}
\end{defi}

\begin{lemma}\label{Riemann-Roch}
Let {\normalfont $[D]\in\text{NS}(X)$} be a conic class. Then $|D|$ is a base point free pencil and therefore induces a surjective morphism $\varphi_{|D|}:X\to\P^1$.
\end{lemma}
\begin{proof}
By \cite[Theorem III.1(a)]{Har}, $|D|$ is base point free and $h^1(X,D)=0$. We have $\chi(X)=1$ by Theorem~\ref{thm:RES_distinguished_properties}, and Riemann-Roch gives $h^0(X,D)+h^2(X,D)=2$, so we only need to prove $h^2(X,D)=0$. Assume by contradiction that $h^2(X,D)\geq 1$. By Serre duality $h^0(K_X-D)\geq 1$, so $K_X-D$ is linearly equivalent to an effective divisor. Since $D$ is nef, $(K_X-D)\cdot D\geq 0$, which contradicts $D^2=0$ and $D\cdot(-K_X)=2$.
\end{proof}

\begin{remark}
\normalfont It follows from Lemma~\ref{Riemann-Roch} that a conic class has an effective representative.
\end{remark}

Notice that we do not know a priori that the morphism $\varphi_{|D|}:X\to\P^1$ in Lemma~\ref{Riemann-Roch} is a conic bundle. At this point we can only say that if $C$ is a smooth, irreducible fiber of $\varphi_{|D|}$, then $g(C)=0$ by adjunction. However it is still not clear whether a general fiber of $\varphi_{|D|}$ is irreducible and smooth. We prove that this is the case in Proposition~\ref{prop:finitely_many_sinuglar_fibers}. In order to do that we need information about the components of $D$ from Proposition~\ref{prop:components_conic_bundle_class} and the fact that $D$ is connected from Proposition~\ref{prop:connectedness}.

\newpage

\begin{prop}\label{prop:components_conic_bundle_class}
Let {\normalfont $[D]\in\text{NS}(X)$} be a conic class. If $D$ is an effective representative, then every curve {\normalfont $E\in\text{Supp }D$} is a smooth rational curve with $E^2\leq 0$.
\end{prop}

\begin{proof}
Take an arbitrary $E\in\text{Supp }D$. By Lemma~\ref{Riemann-Roch}, $D$ is a fiber of the morphism $\varphi_{|D|}:X\to\P^1$ induced by $|D|$, hence $D\cdot E=0$ by Lemma~\ref{lemma:fibers}. Assuming $E^2>0$ by contradiction, the fact that $D^2=0$ implies that $D$ is numerically equivalent to zero by the Hodge index theorem \cite[Thm. V.1.9, Rmk. 1.9.1]{Hartshorne}. This is absurd because $D\cdot(-K_X)=2\neq 0$, so indeed $E^2\leq 0$.

To show that $E$ is a smooth rational curve, it suffices to prove that $p_a(E)=0$. By Theorem~\ref{thm:RES_distinguished_properties}, $-K_X$ is linearly equivalent to any fiber of $\pi$, in particular $-K_X$ is nef and $E\cdot K_X\leq 0$. By adjunction \cite[I.15]{Beauville} $2p_a(E)-2=E^2+E\cdot K_X\leq 0$, so $p_a(E)\leq 1$. Assume by contradition that $p_a(E)=1$. This can only happen if $E^2=E\cdot K_X=0$, so $E$ is a fiber of $\pi$ by Lemma~\ref{lemma:detect_fibers}. In this case $E$ is linearly equivalent to $-K_X$, so $D\cdot E=D\cdot(-K_X)=2$, which contradicts $D\cdot E=0$.
\end{proof}

\begin{remark}

\end{remark}

\indent While Proposition~\ref{prop:components_conic_bundle_class} provides information about the support of $D$, the next proposition states that $D$ is connected, which is an important fact about the composition of $D$ as a whole. 

\begin{prop}\label{prop:connectedness}
Let {\normalfont $[D]\in\text{NS}(X)$} be a conic class. If $D$ is an effective representative, then $D$ is connected.
\end{prop}

\begin{proof}
Let $D=D_1+...+D_n$, where $D_1,...,D_n$ are connected components. By Lemma~\ref{lemma:fibers} c) each $D_i$ is nef with $D_i^2=0$ and $D_i\cdot K_X\in 2\Z$. Since $-K_X$ is nef by Theorem~\ref{thm:RES_distinguished_properties} and $D\cdot(-K_X)=2$, then $D_{i_0}\cdot(-K_X)=2$ for some $i_0$ and $D_i\cdot (-K_X)=0$ for $i\neq i_0$. In particular $[D_{i_0}]\in\text{NS}(X)$ is a conic class. By Lemma~\ref{fixed_component}, both $|D|$ and $|D_{i_0}|$ are pencils, so $D=D_{i_0}$ by Lemma~\ref{fixed_component}.
\end{proof}

\indent We use Propositions~\ref{prop:components_conic_bundle_class} and \ref{prop:connectedness} to conclude that $\varphi_{|D|}:X\to\P^1$ is indeed a conic bundle.

\begin{prop}\label{prop:finitely_many_sinuglar_fibers}
Let {\normalfont $[D]\in\text{NS}(X)$} be a conic class. Then all fibers of $\varphi_{|D|}:X\to\P^1$ are smooth, irreducible curves of genus $0$ except for finitely many which are reducible and supported on negative curves. In particular, $\varphi_{|D|}$ is a conic bundle.
\end{prop}

\begin{proof}
Let $F$ a fiber of $\varphi_{|D|}$. Since $F$ is linearly equivalent to $D$, then $[F]=[D]\in\text{NS}(X)$, so $F$ is connected by Proposition~\ref{prop:connectedness}. By Proposition~\ref{prop:components_conic_bundle_class} every $E\in\text{Supp }F$ has $g(E)=0$ and $E^2\leq 0$.

First assume $E^2=0$ for some $E\in\text{Supp }F$. Since $F$ is a connected fiber of $\varphi_{|D|}$, then $E=rF$ for some $r\in\Q$ by Lemma~\ref{lemma:fibers}. Because $F\cdot(-K_X)=2$, we have $E\cdot K_X=-2r$. By adjunction \cite[I.15]{Beauville}, $r=1$, so $F=E$ is a smooth, irreducible curve of genus $0$. 

Now assume $E^2<0$ for every $E\in\text{Supp }F$. Then $F$ must be reducible, otherwise $F^2=E^2<0$, which is absurd since $F$ is a fiber of $\varphi_{|D|}$. Conversely, if $F$ is reducible, then $E^2<0$ for all $E\in\text{Supp }F$, otherwise $E^2=0$ for some $E$ and by the last paragraph $F$ is irreducible, which is a contradiction. 

This shows that either $F$ is smooth, irreducible of genus $0$ or $F$ is reducible, in which case $F$ is supported on negative curves. We are left to show that $\varphi_{|D|}$ has finitely many reducible fibers.

Assume by contradiction that there is an infinite set $\{F_n\}_{n\in\Bbb{N}}$ of reducible fibers of $\varphi_{|D|}$. In particular each $F_n$ is supported on negative curves, which are either $(-1)$-curves (sections of $\pi$) or $(-2)$-curves (components of reducible fibers) by Lemma~\ref{lemma:negative_curves}.
\newpage
Since $\pi$ has finitely many singular fibers, the number of $(-2)$-curves in $X$ is finite, so there are finitely many $F_n$ with $(-2)$-curves in its support. Excluding such $F_n$, we may assume all members in $\{F_n\}_{n\in\Bbb{N}}$ are supported on $(-1)$-curves. For each $n$, take $P_n\in\text{Supp }F_n$. The fibers $F_n,F_m$ are disjoint, so $P_n,P_m$ are disjoint when $n\neq m$. By contracting finitely many exceptional curves $P_1,...,P_\ell$, we are still left with an infinite set $\{P_m\}_{m>\ell}$ of exceptional curves, so we cannot reach a minimal model for $X$, which is absurd.
\end{proof}

\begin{remark}
\normalfont In characteristic zero the proof of Proposition~\ref{prop:finitely_many_sinuglar_fibers} can be made simpler by applying Bertini's theorem \cite[Cor. 10.9]{Hartshorne}, which guarantees that all but finitely many fibers are smooth	from the fact that $X$ is smooth.
\end{remark}

We now prove the numerical characterization of conic bundles.

\begin{teor}\label{thm:correspondence}
Let $X$ be a rational elliptic surface. If {\normalfont $[D]\in\text{NS}(X)$} is a conic class, then $|D|$ is a base point free pencil whose induced morphism $\varphi_{|D|}:X\to\P^1$ is a conic bundle. Moreover, the map $[D]\mapsto \varphi_{|D|}$ has an inverse $\varphi\mapsto [F]$, where $F$ is any fiber of $\varphi$. This gives a natural correspondence between conic classes and conic bundles.
\end{teor}

\begin{proof}
Given a conic class $[D]\in\text{NS}(X)$, by Proposition~\ref{prop:finitely_many_sinuglar_fibers} the general fiber of $\varphi_{|D|}:X\to\P^1$ is a smooth, irreducible curve of genus $0$, so $\varphi_{|D|}$ is a conic bundle.

Conversely, if $\varphi:X\to\P^1$ is a conic bundle and $C$ is a smooth, irreducible fiber of $\varphi$, in particular $C^2=0$, $g(C)=0$ and by adjunction \cite[I.15]{Beauville} $C\cdot(-K_X)=2$. Clearly $C$ is nef, so $[C]\in\text{NS}(X)$ is a conic class. Moreover, any other fiber $F$ of $\varphi$ is linearly equivalent to $C$, therefore $[F]=[C]\in\text{NS}(X)$ and the map $\varphi\mapsto [F]$ is well defined.

We verify that the maps are mutually inverse. Given a class $[D]$ we may assume $D$ is effective since $|D|$ is a pencil, so $D$ is a fiber of $\varphi_{|D|}$, hence $\varphi_{|D|}\mapsto [D]$. Conversely, given a conic bundle $\varphi$ with a fiber $F$, then $\varphi_{|F|}$ coincides with $\varphi$ tautologically, so $[F]\mapsto \varphi$.
\end{proof}

\section{Classification of conic bundle fibers}\label{section:classification_conic_bundles}\
\indent We prove one of the main results the chapter, which is the complete description of fibers of a conic bundle on $X$. We start with a description of fiber components in Lemma~\ref{lemma:description_of_fiber_components}, define some terminology related to intersection graphs and proceed with the proof of Theorem~\ref{thm:classification_conic_bundles}.

\begin{lemma}\label{lemma:description_of_fiber_components}
Let $\varphi:X\to\P^1$ be a conic bundle and $D$ any fiber of $\varphi$. Then $D$ is connected and 
\begin{enumerate}[(i)]
\item $D$ is either a smooth, irreducible curve of genus $0$, or
\item $D=P_1+P_2+D'$, where $P_1,P_2$ are $(-1)$-curves (sections of $\pi$), not necessarily distinct, and $D'$ is either zero or supported on $(-2)$-curves (components of reducible fibers of $\pi$).
\end{enumerate}
\end{lemma}

\begin{proof}
By Proposition~\ref{prop:finitely_many_sinuglar_fibers}, all fibers of $\varphi$ fall into category (i) except for finitely many that are reducible and supported on negative curves. Let $D$ be one of such finitely many.

From Lemma~\ref{lemma:negative_curves}, $\text{Supp }D$ contains only $(-1)$-curves (sections of $\pi$) or $(-2)$-curves (components of reducible fibers of $\pi$). By adjunction \cite[I.15]{Beauville}, if $C\in\text{Supp }D$ is a $(-2)$-curve, then $C\cdot(-K_X)=0$, and if $P\in\text{Supp }D$ is a $(-1)$-curve, then $P\cdot(-K_X)=1$. But $D\cdot(-K_X)=2$, hence $D$ must have a term $P_1+P_2$ where $P_1,P_2$ are possibly equal $(-1)$-curves; and a possibly zero term $D'$ containing $(-2)$-curves, as desired.
\end{proof}

\indent At this point we have enough information about the curves that support a conic bundle fiber. It remains to investigate their multiplicities and how they intersect one another.
\newpage

\begin{teor}\label{thm:classification_conic_bundles}
Let $X$ be a rational elliptic surface with elliptic fibration $\pi:X\to\P^1$ and let $\varphi:X\to\P^1$ be a conic bundle. If $D$ is a fiber of $\varphi$, then the intersection graph of $D$ fits one of the types below. Conversely, if the intersection graph of a divisor $D$ fits any of these types, then $|D|$ induces a conic bundle $\varphi_{|D|}:X\to\P^1$.
\begin{table}[h]
\begin{center}
\centering
\begin{tabular}{|c|c|} 
\hline
Type & Intersection Graph\\ %& $D$\\
\hline
\multirow{3}{*}{\hfil $0$} & \multirow{3}{*}{\hfil \TypeE}\\ %& \multirow{3}{*}{\hfil $C$}\\
&\\
&\\
\hline
\multirow{3}{*}{\hfil $A_2$} & \multirow{3}{*}{\hfil \TypeD}\\%& \multirow{3}{*}{\hfil $P+Q$}\\
& \\
& \\
\hline
\multirow{3}{*}{\hfil $A_n$ ($n\geq 3$)} & \multirow{3}{*}{\hfil \TypeA}\\ %& \multirow{3}{*}{\hfil $P+\sum_i\Theta_i+Q$}\\
& \\ 
& \\
\hline
\multirow{3}{*}{\hfil $D_3$} & \multirow{3}{*}{\hfil \TypeC}\\ %& \multirow{3}{*}{\hfil $\Theta_0^1+2P+\Theta_0^2$}\\ 
& \\
& \\
\hline
\multirow{4}{*}{\hfil $D_m$ ($m\geq 4$)} & \multirow{4}{*}{\hfil \TypeB}\\ %& \multirow{4}{*}{\hfil $2P+\sum_{i=0}^{n-2}2\Theta_i+\Theta_{n-1}+\Theta_n$}\\
& \\ 
& \\
& \\
\hline
\end{tabular}
\end{center}
\begin{align*}
\star&\,\,\text{smooth, irreducible curve of genus }0\\
\circ&\,(-1)\text{-curve (section of }\pi)\\
\bullet&\,(-2)\text{-curve (component of a reducible fiber of }\pi)
\end{align*}
\end{table}
\end{teor}

\noindent\textbf{Terminology.}\label{terminology} Before we prove Theorem~\ref{thm:classification_conic_bundles}, we introduce a natural terminology for dealing with the intersection graph of $D$. When $C,C'\in\text{Supp }D$ are distinct, we say that $C'$ is a \textit{neighbour} of $C$ when $C\cdot C'>0$. If $C$ has exactly one neighbour, we call $C$ an \textit{extremity}. We denote the number of neighbours of $C$ by $n(C)$. A simple consequence of these definitions is the following.

\begin{lemma}\label{neighbours}
If $D=\sum_in_iE_i$ is a fiber of a morphism $X\to\P^1$, then $n(E_i)\leq -n_i^2E_i^2$ for all $i$.
\end{lemma}

\begin{proof}
By definition of $n(E_i)$, clearly $n(E_i)\leq \sum_{j\neq i}E_i\cdot E_j$. Since $D\cdot E_i=0$ by Lemma~\ref{lemma:fibers}, then
\begin{align*}
0=D\cdot E_i&=\sum_jn_jE_j\cdot E_i\\
&=n_iE_i^2+\sum_{j\neq i}n_jE_j\cdot E_i\\
&\geq n_iE_i^2+\sum_{j\neq i}E_j\cdot E_i\\
&\geq n_iE_i^2+n(E_i).
\end{align*}
\end{proof}

\begin{proof}[Proof of Theorem~\ref{thm:classification_conic_bundles}]
We begin by the converse. If $D$ fits one of the types, we must prove $[D]\in\text{NS}(X)$ is a conic class, so that $\varphi_{|D|}:X\to\P^1$ is a conic bundle by Theorem~\ref{thm:correspondence}. A case-by-case verification gives $D\cdot C=0$ for all $C\in\text{Supp }D$. Since $D$ is effective, it is nef with $D^2=0$ by Lemma~\ref{lemma:critical_divisors_are_nef}. The condition $D\cdot(-K_X)=2$ is satisfied in type $0$ by adjunction. We have $D\cdot(-K_X)=2$ in types $A_2,A_n$ for they contain two distinct sections of $\pi$ and also in types $D_3,D_m$ for they contain a section with multiplicity 2. Hence $[D]\in\text{NS}(X)$ is a conic class, as desired.

Now let $D$ be a fiber of $\varphi$. By Lemma~\ref{lemma:description_of_fiber_components}, $D$ is connected and has two possible forms. If $D$ is irreducible, we get type $0$. Otherwise $D=P_1+P_2+D'$, where $P_1,P_2$ are $(-1)$-curves and $D'$ is either zero or supported on $(-2)$-curves. If $D'=\sum_in_iC_i$ then $n(C_i)\leq 2n_i$ by Lemma~\ref{neighbours}. The bounds for $n(P_1), n(P_2)$ depend on whether i) $P_1\neq P_2$ or ii) $P_1=P_2$. In what follows we use Lemma~\ref{lemma:fibers}~a) implicitly several times.

\indent i) $P_1\neq P_2$. In this case $n(P_1)\leq 1$ and $n(P_2)\leq 1$. Since $D$ is connected, both $P_1,P_2$ must have some neighbour, so $n(P_1)=n(P_2)=1$, therefore $P_1,P_2$ are extremities. If the extremities $P_1,P_2$ meet, they form the whole graph, so $D=P_1+P_2$. This is type $A_2$.

If $P_1,P_2$ do not meet, by connectedness there must be a path on the intersection graph joining them, say $P_1,C_1,...,C_k,P_2$. Since $P_1$ is an extremity, it has only $C_1$ as a neighbour, so $0=D\cdot P_1=-1+n_1$ gives $n_1=1$. Moreover $n(C_1)\leq 2$ and by the position of $C_1$ in the path we have $n(C_1)=2$. We prove by induction that $n_i=1$ and $n(C_i)=2$ for all $i=1,...,k$. Assume this is true for $i=1,...,\ell<k$. Then $0=D\cdot C_\ell=1-2+n_{\ell+1}$, so $n_{\ell+1}=1$. Moreover, $n(C_{\ell+1})\leq 2$ and by the position of $C_{\ell+1}$ in the path we have $n(C_{\ell+1})=2$. So the graph is precisely the chain $P_1,C_1,...,C_k,P_2$. This is type $A_n$ ($n\geq 3$).
\\ \\
\indent ii) $P_1=P_2$. In this case $n(P_1)\leq 2$. We cannot have $n(P_1)=0$, otherwise $D^2=(2P_1)^2=-4$, so $n(P_1)=1$ or $2$. If $P_1$ has two neighbours, say $C_1,C_2$, then $0=D\cdot P_1=-2+n_1+n_2$, which only happens if $n_1=n_2=1$. Moreover, $n(C_1)\leq 2$, so $C_1$ can possibly have another neighbour $C_3$ in addition to $P_1$. But then $D\cdot P_1=0$ gives $n_3=0$, which is absurd, so $C_1$ has only $P_1$ as a neighbour. By symmetry $C_2$ also has only $P_1$ as a neighbour. This is type $D_3$.

Finally let $n(P_1)=1$ and $C_1$ be the only neighbour of $P_1$. Then $C_i\cdot P_1=0$ when $i>1$. Notice that $C_1,C_i$ come from the same fiber of $\pi$, say $F$, otherwise $C_i$ would be in a different connected component as $C_1,P_1$, which contradicts $D$ being connected. The possible Dynkin diagrams for $F$ are listed in Theorem~\ref{thm:Kodaira_classification}. Since $P_1$ intersects $F$ in a simple component, the possibilities are

\begin{table}[h]
\begin{center}
\centering
\begin{tabular}{ccc} 
\multirow{1}{*}{(a)} & \multirow{1}{*}{(b)} & \multirow{1}{*}{(c)}\\
\multirow{4}{*}{\DiagramX} & \multirow{4}{*}{\hfil\DiagramY} & \multirow{4}{*}{\hfil\DiagramZ}\\
& &\\
& &
\end{tabular}
\end{center}
\end{table}

In (a), (b) and (c), $D\cdot P_1=0$ implies $n_1=2$. In (c), $D\cdot C_1=0$ gives $n_2+n_m=2$, hence $n_2=n_m=1$. But $D\cdot C_2=0$ gives $n_3=0$, which is absurd, so (c) is ruled out. For (a), (b) we proceed by induction: if $k<m$ and $n_1=...=n_k=2$, then $D\cdot C_k=0$ gives $n_{k+1}=2$, therefore $n_1=...=n_m=2$. But in (a), $D\cdot C_m=0$ implies $n_m=1$, which is absurd, so (a) is also ruled out.

In (b), let $C_{m+1},C_{m+2}$ be the first elements in branches $1$ and $2$ respectively. Then $D\cdot C_m=0$ gives $n_{m+1}=n_{m+2}=1$. Consequently $n(C_{m+1})\leq 2$ and $n(C_{m+2})\leq 2$. If $C_{m+1}$ has another neighbour $C_{m+3}$ in addition to $C_m$, then $D\cdot C_{m+1}=0$ implies $n_{m+3}=0$, which is absurd, so $C_{m+1}$ is an extremity. By symmetry, $C_{m+2}$ is also an extremity. This is type $D_m$ ($m\geq 4$).
\end{proof}

\section{Fibers of conic bundles vs. fibers of the elliptic fibration}\label{section:conic_vs_elliptic}\
\indent Let $X$ be a rational elliptic surface with elliptic fibration $\pi:X\to\P^1$. The existence of a conic bundle $\varphi:X\to\P^1$ with a given fiber type is strongly dependent on the fiber configuration of $\pi$. This relationship is explored in Theorem~\ref{thm:interplay_conic_bundle_elliptic_fibration}, which provides simple criteria to identify when a certain fiber type is possible. Before we prove it, we need the following result about the existence of disjoint sections.

\begin{lemma}\label{lemma:there_is_a_pair_of_disjoint_sections}
If $X$ is a rational elliptic surface with nontrivial Mordell-Weil group, then there exists a pair of disjoint sections.
\end{lemma}

\begin{proof}
Let $E(K)$ be the Mordell-Weil group of $\pi$, whose neutral section we denote by $O$. By \cite[Thm. 2.5]{OguisoShioda}, $E(K)$ is generated by sections which are disjoint from $O$. Hence there must be a generator $P\neq O$ disjoint from $O$, otherwise $E(K)=\{O\}$, which contradicts the hypothesis.
\end{proof}

\indent We now state and prove the main result of this section.

\begin{teor}\label{thm:interplay_conic_bundle_elliptic_fibration}
Let $X$ be a rational elliptic surface with elliptic fibration $\pi:X\to\P^1$. Then the following statements hold:
\begin{enumerate}[a)]
\item $X$ admits a conic bundle with an $A_2$ fiber $\Leftrightarrow$ $\pi$ has positive generic rank and no {\normalfont $\text{III}^*$} fiber.
\item $X$ admits a conic bundle with an $A_{n\geq 3}$ fiber $\Leftrightarrow$ $\pi$ has a reducible fiber distinct from {\normalfont $\text{II}^*$}.
\item $X$ admits a conic bundle with a $D_3$ fiber $\Leftrightarrow$ $\pi$ has at least two reducible fibers.	
\item $X$ admits a conic bundle with a $D_{m\geq 4}$ fiber $\Leftrightarrow$ $\pi$ has a nonreduced fiber or a fiber {\normalfont $\text{I}_{n\geq 4}$}.
\end{enumerate}
\end{teor}

\begin{proof}
a) We leave the proof of this item to Chapter~\ref{ch:gaps}, since the claim is equivalent to Theorem~\ref{thm:surfaces_with_a_1-gap} in Section~\ref{subsection:surfaces_with_a_1-gap}.

b) Assume $X$ admits a conic bundle with an $A_{n\geq 3}$ fiber. Since type $A_n$ contains a $(-2)$-curve, by Lemma~\ref{lemma:negative_curves}, $\pi$ has a reducible fiber $F$. We claim that $F\neq \text{II}^*$, which is equivalent to saying that the lattice $T$ (see Definition~\ref{def:lattice_T}) is not $E_8$. Indeed, if this were so, then Mordell-Weil group $E(K)$ would be trivial \cite[Main Thm.]{OguisoShioda}, which is impossible, since the $A_n$ fiber of the conic bundle contains two distinct sections. Conversely, assume $\pi$ has a reducible fiber $F\neq \text{II}^*$. Then $E(K)$ is not trivial \cite[Main Thm.]{OguisoShioda} and by Lemma~\ref{lemma:there_is_a_pair_of_disjoint_sections} we may find two disjoint sections $P,P'$. Let $C,C'\in\text{Supp }F$ be the components hit by $P,P'$. Since $F$ is connected, there is a path $C,C_1,...,C_\ell,C'$ in the intersection graph of $F$. Let $D:=P+C+C_1+...+C_\ell+C'+P'$. By Theorem~\ref{thm:classification_conic_bundles}, $\varphi_{|D|}:X\to\P^1$ is a conic bundle and $D$ is an $A_n$ fiber of it.
\begin{center}
\includegraphics[scale=0.65]{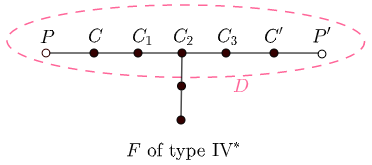}
\end{center}
\newpage
c) Assume $X$ admits a conic bundle with a $D_3$ fiber $D=C_1+2P+C_2$, where $C_1,C_2$ are $(-2)$-curves and $P$ is a section with $C_1\cdot P=C_2\cdot P=1$. Since $P$ hits each fiber of $\pi$ at exactly one point, then $C_1,C_2$ must come from two distinct reducible fibers of $\pi$. Conversely, let $F_1,F_2$ be two reducible fibers of $\pi$. If $P$ is a section, then $P$ hits $F_i$ at some $(-2)$-curve $C_i\in\text{Supp }F_i$. Let $D:=C_1+2P+C_2$. By Theorem~\ref{thm:classification_conic_bundles}, $\varphi_{|D|}:X\to\P^1$ is a conic bundle and $D$ is a $D_3$ fiber of it.

\begin{center}
\includegraphics[scale=0.65]{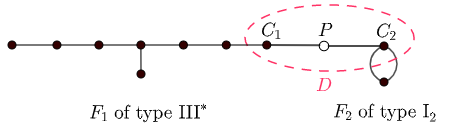}
\end{center}

d) Assume $X$ admits a conic bundle with a $D_m$ fiber $D=2P+2C_1+...+2C_\ell+(C_{\ell+1}+C_{\ell+2})$, where all $C_i$'s come from a reducible fiber $F$ of $\pi$. Notice that if $\ell> 1$ then $C_\ell$ meets three $(-2)$-curves, namely $C_{\ell-1},C_{\ell+1},C_{\ell+2}$ (see picture below). Going through the list in Theorem~\ref{thm:Kodaira_classification}, we see that this intersection behavior only happens if $F$ is $\text{I}_n^*$, $\text{II}^*$, $\text{III}^*$ or $\text{IV}^*$, all of which are nonreduced. If $\ell=1$, then $C_1$ meets the section $P$ and two $(-2)$-curves which do not intersect, namely $C_2,C_3$. Again by examining the list in Theorem~\ref{thm:Kodaira_classification}, $F$ must be $\text{I}_n$ with $n\geq 4$. Conversely, let $F$ be nonreduced or of type $\text{I}_n$ with $n\geq 4$. Take a section $P$ that hits $F$ at $C_1$. Now take a chain $C_2,...,C_\ell$ so that $C_\ell$ meets two other components of $F$. We name these two $C_{\ell+1},C_{\ell+2}$ and define $D:=2P+2C_1+...+2C_\ell+(C_{\ell+1}+C_{\ell+2})$. By Theorem~\ref{thm:classification_conic_bundles}, $\varphi_{|D|}:X\to\P^1$ is a conic bundle and $D$ is a type $D_m$ fiber of it.
\end{proof}

\begin{center}
\includegraphics[scale=0.65]{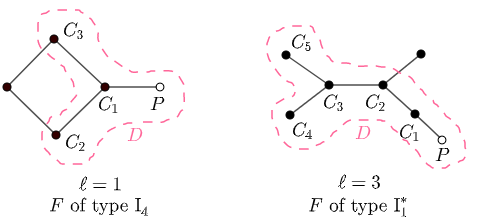}
\end{center}

\newpage

\section{Examples of conic bundles}\label{section:examples_conic_bundles}\
\indent We present examples of conic bundles over rational elliptic surfaces to illustrate how the fiber types in Theorem~\ref{thm:classification_conic_bundles} may appear. For simplicity, in this section we work over $k=\Bbb{C}$, although similar constructions are possible over different fields. In Subsection~\ref{subsection:construction_and_notation} we describe how the examples are constructed and stablish some notation, then present the examples in Subsection~\ref{subsection:examples_conic_bundles}.

\subsection{Construction and notation}\label{subsection:construction_and_notation}\
\indent As explained in Section~\ref{section:RES}, $\pi$ is induced by a pencil of cubics $\mathcal{P}$ from the blowup $p:X\to\P^2$ of the base locus of $\mathcal{P}$. We describe a method for constructing a conic bundle $\varphi:X\to\P^1$ from a pencil of curves with genus zero on $\P^2$. 
\\ \\
\noindent\textbf{Construction.} Let $\mathcal{Q}$ be a pencil of conics (or a pencil of lines) given by a dominant rational map $\psi:\P^2\dasharrow\P^1$ with the following properties:
\begin{enumerate}[(a)]
\item $\psi^{-1}(t)$ is smooth for all but finitely many $t\in\P^1$, i.e. the general member of $\mathcal{Q}$ is smooth. 
\item The indeterminacy locus of $\psi$ (equivalently, the base locus of $\mathcal{Q})$ is contained in the base locus of $\mathcal{P}$ (including infinitely near points).
\end{enumerate}

Now define a surjective morphism $\varphi:X\to\P^1$ by the composition 
\begin{displaymath}
\begin{tikzcd}
X\arrow{r}{p}\arrow[bend right, swap]{rr}{\varphi} & \P^2\arrow[dashed]{r}{\psi} &\P^1
\end{tikzcd}
\end{displaymath}

Notice that $\varphi$ is a well defined conic bundle. Indeed, by property (b) the points of indeterminacy of $\psi$ are blown up under $p$, so $\varphi$ is a morphism. By property (a) the general fiber of $\varphi$ is a smooth, irreducible curve. Since $\mathcal{Q}$ is composed of conics (or lines), the general fiber of $\varphi$ has genus zero.
\\ \\
\noindent\textbf{Example.} Let $C$ be a smooth cubic and let $L_1,L_2,L_3$ be concurrent lines. Define $\mathcal{P}$ as the pencil generated by $C$ and $L_1+L_2+L_3$. Let $P_1,P_2\in L_1\cap C$ and $P_3,P_4\in L_3\cap C$ and let $\mathcal{Q}$ be pencil of conics through $P_1,P_2,P_3,P_4$. In the following picture, $Q\in\mathcal{Q}$ is a general conic, so the strict transform of $Q$ under $p$ is a general fiber of the conic bundle $\varphi:X\to\P^1$.
\begin{center}
\includegraphics[scale=0.5]{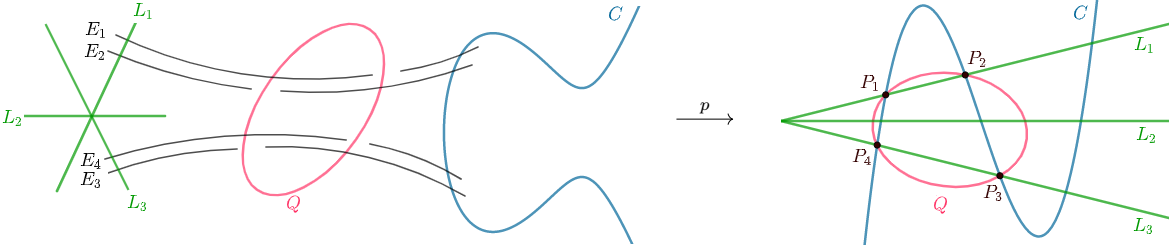}
\end{center}

\begin{remark}
\normalfont Since the base points $P_1,P_2,P_3,P_4$ of $\mathcal{Q}$ are blown up under $p$, then the pullback pencil $p^*\mathcal{Q}$ has four fixed components, namely the exceptional divisors $E_1,E_2,E_3,E_4$. By eliminating these we obtain a base point free pencil $p^*\mathcal{Q}-E_1-E_2-E_3-E_4$, which is precisely the one given by $\varphi:X\to\P^1$.
\end{remark}

\noindent\textbf{Notation.}
\begin{align*}
\mathcal{P}&\,\,\,\,\,\,\,\text{pencil of cubics on }\P^2\text{ inducing }\pi.\\
p&\,\,\,\,\,\,\,\text{blowup }p:X\to\P^2\text{ at the base locus of }\mathcal{P}.\\
\mathcal{Q}&\,\,\,\,\,\,\,\text{pencil of conic (or lines) on }\P^2.\\
Q&\,\,\,\,\,\,\,\text{conic in }\mathcal{Q},\text{ not necessarily smooth}.\\
L&\,\,\,\,\,\,\,\text{line in }\mathcal{Q}.\\
\varphi&\,\,\,\,\,\,\,\text{conic bundle }\varphi:X\to\P^1\text{ induced by }\mathcal{Q}.\\
D&\,\,\,\,\,\,\,\text{singular fiber of }\varphi\text{ such that }\varphi(D)=Q\text{ (or }\varphi(D)=L).\\
D'&\,\,\,\,\,\,\,\text{another singular fiber of }\varphi.
\end{align*}

\begin{remark}
\normalfont For simplicity, the strict transform of a curve $E\subset \P^2$ under $p$ is also denoted by $E$ instead of the usual $\widetilde{E}$.
\end{remark}

\subsection{Examples}\label{subsection:examples_conic_bundles}\
\indent We exhibit two classes of examples: the extreme cases, i.e. where the conic bundle admits only one type of single fiber; and the ones with various fiber types. 

\subsubsection{Extreme cases}\
\indent By Theorem~\ref{thm:interplay_conic_bundle_elliptic_fibration}, there are two cases in which $X$ can only admit conic bundles with exactly one type of singular fiber.
\begin{enumerate}[(1)]
\item When $\pi$ has a $\text{II}^*$ fiber: $X$ only admits conic bundles with singular fibers of type $D_{m\geq 4}$.
\item When $\pi$ has no reducible fibers: $X$ only admits conic bundles with singular fibers of type $A_2$.
\end{enumerate}

In Persson's classification list \cite{Persson}, case (1) corresponds to the first two entries in the list (the ones with trivial Mordell-Weil group) and case (2) corresponds to the last six entries (the ones with maximal Mordell-Weil rank). Examples~\ref{only_type_5} and \ref{only_type_2} illustrate cases (1) and (2) respectively.
\newpage
\begin{example}\label{only_type_5}
Let $\mathcal{P}$ be induced by a smooth cubic {\color{customBlue}$C$} and a triple line {\color{customGreen}$3L$} as defined below. We blow up the base locus $\{9P_1\}$ and obtain a rational elliptic surface whose configuration of singular fibers is $(\text{II}^*,\text{II})$. By Theorem~\ref{thm:interplay_conic_bundle_elliptic_fibration}, $X$ can only admit conic bundles with singular fibers of type $D_{m\geq 4}$. We construct a conic bundle from the pencil $\mathcal{Q}$ of lines through $P_1$. The curve ${\color{customPink}D}:=p^*{\color{customPink}L}-E_1$ is in the base point free pencil $p^*\mathcal{Q}-E_1$, which induces the conic bundle $\varphi_{|{\color{customPink}D}|}:X\to\P^1$. The curve ${\color{customPink}D}$ is a $D_9$ fiber of $\varphi_{|{\color{customPink}D}|}$.
\begin{center}
\includegraphics[scale=0.65]{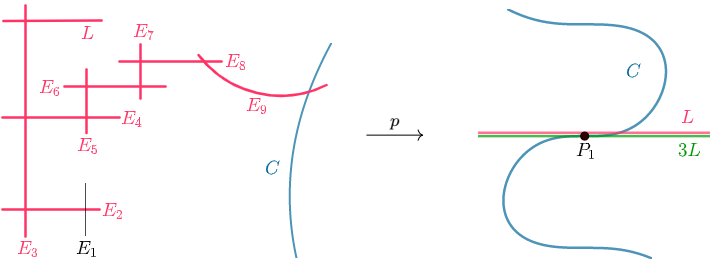}
\end{center}
\begin{align*}
&{\color{customBlue}C}:x^3+y^3-yz^2=0.\\
&{\color{customPink}L}:y=0.\\
&\mathcal{P}:\text{pencil of cubics induced by }{\color{customBlue}C}\text{ and }{\color{customGreen}3L}.\\
&\mathcal{Q}:\text{pencil of lines through }P_1.\\
&{\color{customPink}D}=2E_9+2E_8+2E_7+2E_6+2E_5+2E_4+2E_3+(E_2+{\color{customPink}L}), \text{ type }D_9.\\
&\text{Sequence of contractions:}\,E_9,E_8,E_7,E_6,E_5,E_4,E_3,E_2,E_1.
\end{align*}
\end{example}

\begin{remark}\label{single_singular_fiber}
\normalfont The conic bundle in Example~\ref{only_type_5} is in fact the only conic bundle on $X$. This follows from the fact that $\text{II}^*$ is the only reducible fiber of $\pi$ and that $E_9$ is the only section of $X$, since the Mordell-Weil group is trivial \cite{Persson}. By examining the intersection graph of $\text{II}^*$, we conclude that ${\color{customPink}D}$ is the only divisor that constitutes a $D_{m\geq 4}$ fiber.
\end{remark}

\begin{remark}
\normalfont A similar construction can be made to obtain $\pi:X\to\P^1$ with configuration $(\text{II}^*,2\text{I}_1)$, in which case Remark~\ref{single_singular_fiber} also applies.
\end{remark}

\newpage

\begin{example}\label{only_type_2}
Let $\mathcal{P}$ be induced by a smooth cubic {\color{customBlue} $C$} and a cubic {\color{customGreen}$C'$} with a node, as given below. By blowing up the base locus $\{P_1,...,P_9\}$ we obtain $\pi:X\to\P^1$ with configuration $(\text{II},10\text{I}_1)$. By Theorem~\ref{thm:interplay_conic_bundle_elliptic_fibration}, $X$ can only admit conic bundles with singular fibers of type $A_2$. Let $\mathcal{Q}$ be the pencil of lines through $P_1$. The curve {\color{customPink}$D$}$:=p^*{\color{customPink}L}-E_1$ is in the base point free pencil $p^*\mathcal{Q}-E_1$, which induces the conic bundle $\varphi_{|{\color{customPink}D}|}:X\to\P^1$. The curve ${\color{customPink}D}$ is an $A_2$ fiber of $\varphi_{|{\color{customPink}D}|}$.

Notice moreover that ${\color{customPink}D}$ is not the only singular fiber of the conic bundle. Indeed, each line $L_{1i}$ through $P_1$ and $P_i$ with $i=2,...,9$ corresponds to an $A_2$ fiber of $\varphi_{|{\color{customPink}D}|}$, namely $L_{1i}+E_i$.

\begin{center}
\includegraphics[scale=0.65]{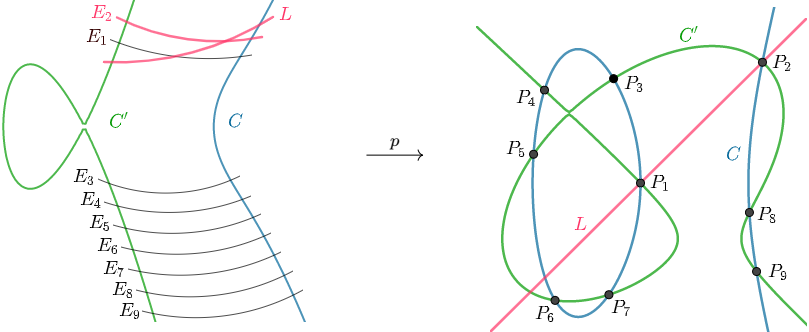}
\end{center}
\begin{align*}
&{\color{customBlue}C}:y^2z-4x^3+4xz^2=0.\\
&{\color{customGreen}C'}:y^2z-4x^3+4xz^2+(127/100)(xz^2-4y^3+4yz^2)=0.\\
&{\color{customPink}L}:\text{ line through }P_1,P_2.\\
&\mathcal{P}:\text{pencil of cubics induced by }{\color{customBlue}C}\text{ and }{\color{customGreen}C'}.\\
&\mathcal{Q}:\text{pencil of lines through }P_1.\\
&{\color{customPink}D}=p^*{\color{customPink}L}-E_1={\color{customPink}L}+E_2, \text{ type }A_2.\\
&\text{Sequence of contractions:}\,E_9,E_8,E_7,E_6,E_5,E_4,E_3,E_2,E_1.
\end{align*}
\end{example}

%\begin{remark}
%\normalfont ${\color{customPink}D}$ is not the only singular fiber of the conic bundle in Example~\ref{only_type_2}. In fact, each line $L_{1i}$ through $P_1$ and $P_i$ with $i=2,...,9$ corresponds to an $A_2$ fiber of $\varphi_{|{\color{customPink}D}|}$, namely $L_{1i}+E_i$.
%\end{remark}

\begin{remark}
\normalfont Since the conic bundle in Example~\ref{only_type_2} only admits singular fibers of type $A_2$, which are isomorphic to a pair of lines meeting at a point, we have a \textit{standard conic bundle} in the sense of Manin and Tsfasman \cite[Subsection 2.2]{ManTsfa}.
\end{remark}

\newpage

\subsubsection{Mixed fiber types}
\begin{example}\label{types_2_and_3}
Let $\mathcal{P}$ be induced by a cubic {\color{customBlue}$C$} with a cusp and the triplet of lines {\color{customGreen}$L_1+L_2+L_3$} as below. By blowing up the base locus $\{P_1,...,P_9\}$ we obtain $\pi:X\to\P^1$ with configuration $(\text{IV}, \text{II}, 6\text{I}_1)$. By Theorem~\ref{thm:interplay_conic_bundle_elliptic_fibration}, $X$ admits only conic bundles with singular fibers of types $A_2$ or $A_{n\geq 3}$. We construct a conic bundle with singular fibers of types $A_2$ and $A_3$. Let $\mathcal{Q}$ be the pencil of lines through $P_1$. We define ${\color{customPink}L}$ as the line through $P_1,P_4$ and ${\color{customPink}L'}:={\color{customGreen}L_1}$. 

Then ${\color{customPink}D}:=p^*{\color{customPink}L}-E_1$ and ${\color{customPink}D'}:=p^*{\color{customPink}L'}-E_1$ are curves in the base point free pencil $p^*\mathcal{Q}-E_1$, which induces the conic bundle $\varphi_{|{\color{customPink}D}|}:X\to\P^1$. The curves $\color{customPink}D$, $\color{customPink}D'$  are fibers of $\varphi_{|\color{customPink}D|}$ of type $A_2$, $A_3$ respectively. In addition to ${\color{customPink}D}$ and ${\color{customPink}D'}$, the conic bundle has five other singular fibers, each of type $A_2$. Indeed, each line $L_{1i}$ through $P_1,P_i$ with $i\in\{2,3,5,7,8\}$ induces the $A_2$ fiber $L_{1i}+E_i$.
\begin{center}
\includegraphics[scale=0.65]{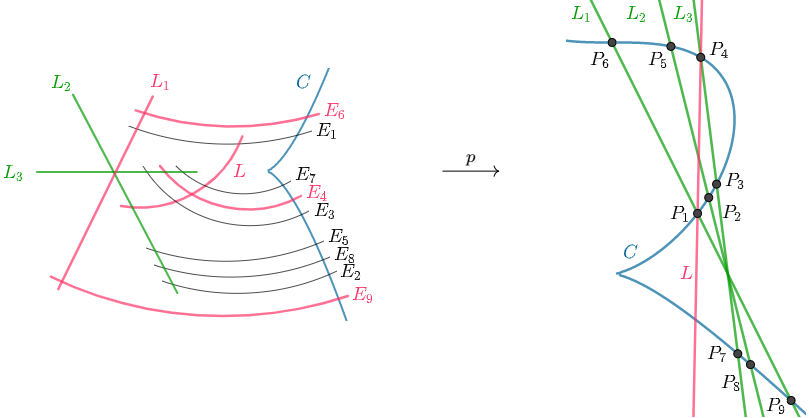}
\end{center}
\begin{align*}
&{\color{customBlue}C}:x^3+y^3-y^2z=0.\\
&{\color{customGreen}L_1}:y+2x-z=0.\\
&{\color{customGreen}L_2}:y+4x-2z=0.\\
&{\color{customGreen}L_3}:y+8x-4z=0.\\
&{\color{customPink}L}:\text{ line through }P_1,P_4.\\
&{\color{customPink}L'}={\color{customGreen}L_1}.\\
&\mathcal{P}:\text{ pencil of cubics induced by }{\color{customBlue}C}\text{ and }{\color{customGreen}L_1+L_2+L_3}.\\
&\mathcal{Q}:\text{ pencil of lines through }P_1.\\
&{\color{customPink}D}={\color{customPink}L}+E_4,\text{ type }A_2.\\
&{\color{customPink}D'}=E_6+{\color{customPink}L'}+E_9,\text{ type }A_3.\\
&\text{Sequence of contractions:}\,E_9,E_8,E_7,E_6,E_5,E_4,E_3,E_2,E_1.
\end{align*}
\end{example}

\newpage

\begin{example}
Let $\mathcal{P}$ be induced by a smooth cubic {\color{customBlue}$C$} and a triplet of lines {\color{customGreen}$L_1+L_2+L_3$} as below. By blowing up the base locus $\{2P_1,3P_2, 2P_3,P_4,P_5\}$ we obtain $\pi:X\to\P^1$ with configuration $(\text{I}_7, \text{II}, 3\text{I}_1)$. By Theorem~\ref{thm:interplay_conic_bundle_elliptic_fibration}, $X$ admits conic bundles with singular fibers of types $A_{n\geq 3}$, $D_{m\geq 4}$ or $A_2$. We construct a conic bundle with types $A_2$, $A_4$, $D_4$. Let $\mathcal{Q}$ be the pencil of lines through $P_1$. We also define ${\color{customPink}L}:={\color{customGreen}L_1}$, ${\color{customPink}L'}:={\color{customGreen}L_2}$ and ${\color{customPink}L''}$ as the line through $P_1,P_5$.

Then ${\color{customPink}D}:=p^*{\color{customPink}L}-E_1$, ${\color{customPink}D'}:=p^*{\color{customPink}L'}-E_1$ and ${\color{customPink}D''}:=p^*{\color{customPink}L''}-E_1$ are curves in the base point free pencil $p^*\mathcal{Q}-E_1$, which induces the conic bundle $\varphi_{|{\color{customPink}D}|}:X\to\P^1$. The curves $\color{customPink}D$, $\color{customPink}D'$, $\color{customPink}D''$ are fibers of $\varphi_{|\color{customPink}D|}$ of types $D_4$, $A_4$, $A_2$ respectively.
\begin{center}
\includegraphics[scale=0.65]{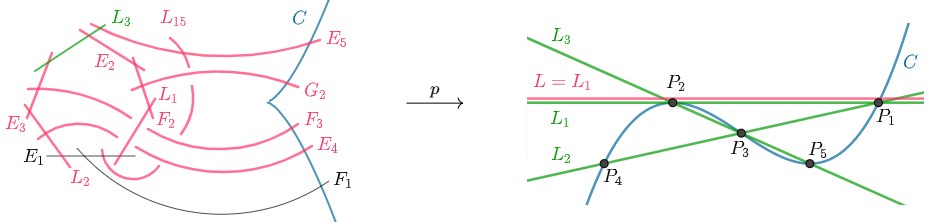}
\end{center}
\begin{align*}
&{\color{customBlue}C}:yz^2-2x^2(x-z)=0.\\
&{\color{customGreen}L_1}:y=0.\\
&{\color{customGreen}L_2}:2x-9y-2z=0.\\
&{\color{customGreen}L_3}:4x+9y=0.\\
&{\color{customPink}L}={\color{customGreen}L_1}.\\
&{\color{customPink}L'}={\color{customGreen}L_2}.\\
&{\color{customPink}L''}:\text{ line through }P_1,P_5.\\
&\mathcal{P}:\text{ pencil of cubics induced by }{\color{customBlue}C}\text{ and }{\color{customGreen}L_1+L_2+L_3}.\\
&\mathcal{Q}:\text{ pencil of lines through }P_1.\\
&{\color{customPink}D}=2G_2+2F_2+(E_2+{\color{customPink}L}),\text{ type }D_4.\\
&{\color{customPink}D'}=E_4+{\color{customPink}L'}+E_3+F_3,\text{ type }A_4.\\
&{\color{customPink}D''}={\color{customPink}L''}+E_5,\text{ type }A_2.\\
&\text{Sequence of contractions:}\,E_5,E_4,F_3,E_3,G_2,F_2,E_2,F_1,E_1.
\end{align*}
\end{example}

\newpage

\begin{example}
Let $\mathcal{P}$ be induced by a cubic with {\color{customBlue}$C$} with a cusp and a triplet of lines {\color{customGreen}$L_1+L_2+L_3$} as given below. By blowing up the base locus $\{P_1,...,P_9\}$ we obtain $\pi:X\to\P^1$ with configuration $(\text{I}^*_2,\text{III},\text{I}_1)$. By Theorem~\ref{thm:interplay_conic_bundle_elliptic_fibration}, $X$ admits conic bundles with singular fibers of types $A_{n\geq 3}$, $D_3$, $D_{m\geq 4}$ or $A_2$. We construct a conic bundle with singular fiber of types $A_3$, $D_3$, $D_5$. Let $\mathcal{Q}$ be the pencil of lines through $P_1$. We also define ${\color{customPink}L}:={\color{customGreen}L_2}$, ${\color{customPink}L'}:={\color{customGreen}L_3}$ and ${\color{customPink}L''}:={\color{customGreen}L_1}$.

Then ${\color{customPink}D}:=p^*{\color{customPink}L}-E_1$, ${\color{customPink}D'}:=p^*{\color{customPink}L'}-E_1$ and ${\color{customPink}D''}:=p^*{\color{customPink}L''}-E_1$ are curves in the base point free pencil $p^*\mathcal{Q}-E_1$, which induces the conic bundle $\varphi_{|{\color{customPink}D}|}:X\to\P^1$. The curves $\color{customPink}D$, $\color{customPink}D'$, $\color{customPink}D''$ are fibers of $\varphi_{|{\color{customPink}D}|}$ of types $A_3$, $D_3$, $D_5$ respectively.

\begin{center}
\includegraphics[scale=0.65]{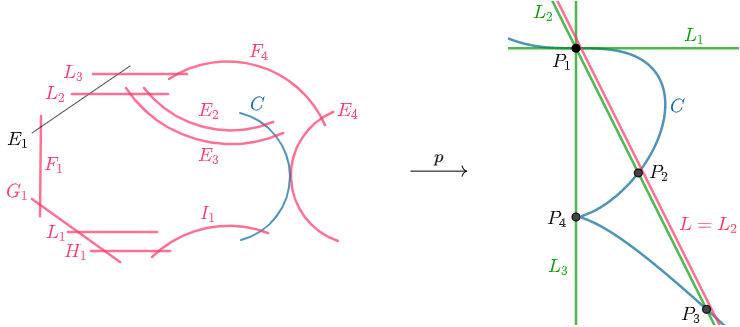}
\end{center}
\begin{align*}
&{\color{customBlue}C}:x^3+y^3-y^2z=0.\\
&{\color{customGreen}L_1}:y-z=0.\\
&{\color{customGreen}L_2}:x=0.\\
&{\color{customGreen}L_3}:y+2x-z	=0.\\
&{\color{customPink}L}:={\color{customGreen}L_2}.\\
&{\color{customPink}L'}:={\color{customGreen}L_3}.\\
&{\color{customPink}L''}:={\color{customGreen}L_1}.\\
&\mathcal{P}:\text{ pencil of cubics induced by }{\color{customBlue}C}\text{ and }{\color{customGreen}L_1+L_2+L_3}.\\
&\mathcal{Q}:\text{ pencil of lines through }P_1.\\
&{\color{customPink}D}=E_2+{\color{customPink}L}+E_3,\text{ type }A_3.\\
&{\color{customPink}D'}=E_4+2F_4+{\color{customPink}L'},\text{ type }D_3.\\
&{\color{customPink}D''}=2I_1+2H_1+2G_1+({\color{customPink}L''}+F_1),\text{ type }D_5.\\
&\text{Sequence of contractions:}\,F_4,E_4,E_3,E_2,I_1,H_1,G_1,F_1,E_1.
\end{align*}
\end{example}

\chapter{Gaps on the intersection numbers of sections}\label{ch:gaps}\
\indent As in Chapter~\ref{ch:conic_bundles_on_RES}, we let $\pi:X\to\P^1$ be a rational elliptic surface over any algebraically closed field $k$. We remind the reader that in the proof of Theorem~\ref{thm:interplay_conic_bundle_elliptic_fibration} we leave item a) to be treated in the present chapter. This involves answering the following question: which rational elliptic surfaces admit a pair of sections $P_1,P_2\in E(K)$ such that $P_1\cdot P_2=1$? The answer is given in Theorem~\ref{thm:surfaces_with_a_1-gap}, which was only a conjecture at the time of our investigation on conic bundles and could only be proven after a more careful study of intersection numbers of sections. This motivating problem lead to a broader broader investigation of the possible values of intersection numbers of sections, whose results are gathered in this chapter.

The main tool we need is the Mordell-Weil lattice (Section~\ref{section:MW_lattice}), which consists of a lattice structure on $E(K)/E(K)_\text{tor}$ having a close connection with the Néron-Severi lattice. We take advantage of the fact that all possibilities for the Mordell-Weil lattices of $X$ have been classified in \cite{OguisoShioda}, so that the well-know information of the height-pairing gives us information about the intersection numbers of sections of $X$.

By adopting this strategy, we often come across a classical problem in number theory, which is the representation of integers by positive-definite quadratic forms. Indeed, if the free part of $E(K)$ is generated by $r$ terms, then the height $h(P):=\langle P,P\rangle$ induces a positive-definite quadratic form on $r$ variables with coefficients in $\Bbb{Q}$. If $O\in E(K)$ is the neutral section and $R$ is the set of reducible fibers of $\pi$, then by the height formula (\ref{equation:height_formula_P})
$$h(P)=2+2(P\cdot O)-\sum_{v\in R}\text{contr}_v(P),$$
where the sum over $v$ is a rational number which can be estimated (Table~\ref{table:local_contributions}). By clearing denominators, we see that the possible values of $P\cdot O$ depend on a certain range of integers represented by a positive-definite quadratic form with coefficients in $\Z$. 
\newpage
We organize this chapter as follows. In Section~\ref{section:gap_numbers} we define some terminology and dedicate Section~\ref{section:intersection_with_a_torsion_section} to explaining the role of torsion sections in our investigation. The technical core of this chapter is in Sections~\ref{section:necessary_conditions}, \ref{section:sufficient_conditions_Delta<=2} and \ref{section:necessary_and_sufficient_Delta<=2}, where we find necessary and sufficient conditions for $k\in\Z_{\geq 0}$ to be the intersection number of two sections. In Section~\ref{section:summary_sufficient_conditions} we list the total of sufficient conditions obtained. The main results are in Section~\ref{section:applications}, namely: the description of gap numbers when $E(K)$ is torsion-free with $r=1$ (Subsection~\ref{subsection:identification_of_gaps_r=1}), the absence of gap numbers for $r\geq 5$ (Subsection~\ref{subsection:gap_free_r>=5}), density of gap numbers when $r\leq 2$ (Subsection~\ref{subsection:gaps_probability_1_r=1,2}) and the classification of surfaces with a $1$-gap (Subsection~\ref{subsection:surfaces_with_a_1-gap}). We use the appendix in Chapter~\ref{ch:appendix} to obtain the information we need about Mordell-Weil lattices.

\section{Gap numbers}\label{section:gap_numbers}\
\indent We introduce some convenient terminology to express the possibility of finding a pair of sections with a given intersection number.

\begin{defi}
If there are no sections $P_1,P_2\in E(K)$ such that $P_1\cdot P_2=k$, we say that $X$ has a $k$-\textit{gap} or that $k$ is a \textit{gap number} of $X$.
\end{defi}

\begin{defi}
We say that $X$ is \textit{gap-free} if for every $k\in\Bbb{Z}_{\geq 0}$ there are sections $P_1,P_2\in E(K)$ such that $P_1\cdot P_2=k$.
\end{defi}

\begin{remark}
\normalfont In case the Mordell-Weil rank is $r=0$, we have $E(K)=E(K)_\text{tor}$. In particular, any two distinct sections are disjoint \cite[Cor. 8.30]{MWL}, hence every $k\geq 1$ is a gap number of $X$. For positive rank, the description of gap numbers is less trivial, hence our focus on $r\geq 1$.
\end{remark}

\section{Intersection with a torsion section}\label{section:intersection_with_a_torsion_section}\
\indent Before dealing with more technical details in Sections~\ref{section:necessary_conditions} and \ref{section:sufficient_conditions_Delta<=2}, we explain how torsion sections can be of help in our investigation.

We first note some general properties of torsion sections. As the height pairing is positive-definite on $E(K)/E(K)_\text{tor}$, torsion sections are inert in the sense that for each $Q\in E(K)_\text{tor}$ we have $\langle Q,P\rangle=0$ for all $P\in E(K)$. Moreover, in the case of rational elliptic surfaces, torsion sections also happen to be mutually disjoint by Theorem~\ref{thm:torsion_sections_disjoint}.

%\begin{teor}{\normalfont \cite[Lemma 1.1]{MirandaPersson}}\label{thm:torsion_sections_disjoint}
%On a rational elliptic surface, $Q_1\cdot Q_2=0$ for any distinct $Q_1,Q_2\in E(K)_\text{tor}$. In particular, if $O$ is the neutral section, then $Q\cdot O=0$ for all $Q\in E(K)_\text{tor}\setminus\{O\}$.
%\end{teor}
%
%\begin{remark}
%\normalfont As stated in \cite[Lemma 1.1]{MirandaPersson}, Theorem~\ref{thm:torsion_sections_disjoint} holds for elliptic surfaces over $\Bbb{C}$ even without assuming $X$ is rational. However, for an arbitrary algebraically closed field the rationality hypothesis is needed, and a proof can be found in \cite[Cor. 8.30]{MWL}.
%\end{remark}

By considering the properties above, we use torsion sections to help us find $P_1,P_2\in E(K)$ such that $P_1\cdot P_2=k$ for a given $k\in\Bbb{Z}_{\geq 0}$. This is particularly useful when $\Delta\geq 2$, in which case $E(K)_\text{tor}$ is not trivial by Lemma~\ref{lemma:cases_where_Delta>=2}. 

The idea works as follows. Given $k\in\Bbb{Z}_{\geq 0}$, suppose we can find $P\in E(K)^0$ with height $h(P)=2k$. By the height formula (\ref{equation:height_formula_P}), $P\cdot O=k-1<k$, which is not yet what we need. In the next lemma we show that replacing $O$ with a torsion section $Q\neq O$ gives $P\cdot Q=k$, as desired.

\newpage

\begin{lemma}\label{lemma:PO_plus_one}
Let $P\in E(K)^0$ such that $h(P)=2k$. Then $P\cdot Q=k$ for all $Q\in E(K)_\text{tor}\setminus\{O\}$.
\end{lemma}
\begin{proof}
Assume there is some $Q\in E(K)_\text{tor}\setminus\{O\}$. By Theorem~\ref{thm:torsion_sections_disjoint}, $Q\cdot O=0$ and by the height formula (\ref{equation:height_formula_P}), $2k=2+2(P\cdot O)-0$, hence $P\cdot O=k-1$. We use the height formula (\ref{equation:height_formula_PQ}) for $\langle P,Q\rangle$ in order to conclude that $P\cdot Q=k$. Since $P\in E(K)^0$, it intersects the neutral component $\Theta_{v,0}$ of every reducible fiber $\pi^{-1}(v)$, so $\text{contr}_v(P,Q)=0$ for all $v\in R$. Hence
\begin{align*}
0&=\langle P,Q\rangle\\
&=1+P\cdot O+Q\cdot O-P\cdot Q-\sum_{v\in R}\text{contr}_v(P,Q)\\
&=1+(k-1)+0-P\cdot Q-0\\
&=k-P\cdot Q.
\end{align*}
\end{proof}

\section{Necessary conditions}\label{section:necessary_conditions}\
\indent If $k\in\Bbb{Z}_{\geq 0}$, we state necessary conditions for having $P_1\cdot P_2=k$ for some sections $P_1,P_2\in E(K)$. We note that the value of $\Delta$ is not relevant in this section, but plays a decisive role for sufficient conditions in Section~\ref{section:sufficient_conditions_Delta<=2}.

\begin{lemma}\label{lemma:necessary_conditions}\
Let $k\in\Bbb{Z}_{\geq 0}$. If $P_1\cdot P_2=k$ for some $P_1,P_2\in E(K)$, then one of the following holds:
\begin{enumerate}[i)]
\item $h(P)=2+2k$ for some $P\in E(K)^0$.
\item $h(P)\in [2+2k-c_\text{max},\,2+2k-c_\text{min}]$ for some $P\notin E(K)^0$.
\end{enumerate}
\end{lemma}

\begin{proof}
Without loss of generality we may assume that $P_2$ is the neutral section, so that $P_1\cdot O=k$. By the height formula~(\ref{equation:height_formula_P}), $h(P_1)=2+2k-c$, where $c:=\sum_v\text{contr}_v(P_1)$. If $P_1\in E(K)^0$, then $c=0$ and $h(P_1)=2+2k$, hence i) holds. If $P_1\notin E(K)^0$, then $c_\text{min}\leq c\leq c_\text{max}$ by Lemma~\ref{lemma:bounds_are_actually_bounds}. But $h(P_1)=2+2k-c$, therefore $2+2k-c_\text{max}\leq h(P_1)\leq 2+2k-c_\text{min}$, i.e. ii) holds.
\end{proof}

\begin{coro}\label{coro:necessary_conditions_Q_X}
Let $k\in\Bbb{Z}_{\geq 0}$. If $P_1\cdot P_2=k$ for some $P_1,P_2\in E(K)$, then $Q_X$ represents some integer in $[d\cdot (2+2k-c_\text{max}),d\cdot(2+2k)]$, where $d:=\det E(K)^0$.
\end{coro}

\begin{proof}
We apply Lemma~\ref{lemma:necessary_conditions} and rephrase it in terms of $Q_X$. If i) holds, then $Q_X$ represents $d\cdot (2+2k)$ by Lemma~\ref{lemma:Q_X_represents_dm}. If ii) holds, then $h(P)\in[2+2k-c_\text{max},2+2k-c_\text{min}]$ and by Lemma~\ref{lemma:Q_X_represents_dm}, $Q_X$ represents $d\cdot h(P)\in[d\cdot (2+2k-c_\text{max}),d\cdot(2+2k-c_\text{min})]$. In both i) and ii), $Q_X$ represents some integer in $[d\cdot(2+2k-c_\text{max}),d\cdot(2+2k)]$.
\end{proof}

\section{Sufficient conditions when $\Delta\leq 2$}\label{section:sufficient_conditions_Delta<=2}\
\indent In this section we state sufficient conditions for having $P_1\cdot P_2=k$ for some $P_1,P_2\in E(K)$ under the assumption that $\Delta\leq 2$. By Lemma~\ref{lemma:cases_where_Delta>=2}, this covers almost all cases (more precisely, all but No. 41, 42, 59, 60 in Table~\ref{table:MWL_data}). We treat $\Delta<2$ and $\Delta=2$ separately, as the latter needs more attention.

\newpage

\subsection{The case $\Delta<2$}\label{subsection:the_case_Delta<2}\
\indent We first prove Lemma~\ref{lemma:sufficient_conditions_Delta<2}, which gives sufficient conditions assuming $\Delta<2$, then Corollary~\ref{coro:sufficient_conditions_Q_X_Delta<2}, which states sufficient conditions in terms of integers represented by $Q_X$. This is followed by Corollary~\ref{coro:sufficient_conditions_mu_Delta<2}, which is a simplified version of Corollary~\ref{coro:sufficient_conditions_Q_X_Delta<2}.

\begin{lemma}\label{lemma:sufficient_conditions_Delta<2}
Assume $\Delta<2$ and let $k\in\Bbb{Z}_{\geq 0}$. If $h(P)\in[2+2k-c_\text{max}, 2+2k-c_\text{min}]$ for some $P\notin E(K)^0$, then $P_1\cdot P_2=k$ for some $P_1,P_2\in E(K)$.
\end{lemma}
\begin{proof}
Let $O\in E(K)$ be the neutral section. By the height formula (\ref{equation:height_formula_P}), $h(P)=2+2(P\cdot O)-c$, where $c:=\sum_v\text{contr}_v(P)$. Since $h(P)\in [2+2k-c_\text{max},2+2k-c_\text{min}]$, then
\begin{align*}
2+2k-c_\text{max}&\leq 2+2(P\cdot O)-c\leq 2+2k-c_\text{min}\\
\Rightarrow \frac{c-c_\text{max}}{2} &\leq P\cdot O-k\leq \frac{c-c_\text{min}}{2}.
\end{align*}

Therefore $P\cdot O-k$ is an integer in $I:=\left[\frac{c-c_\text{max}}{2},\frac{c-c_\text{min}}{2}\right]$. We prove that $0$ is the only integer in $I$, so that $P\cdot O-k=0$, i.e. $P\cdot O=k$. First notice that $c\neq 0$, as $P\notin E(K)^0$. By Lemma~\ref{lemma:bounds_are_actually_bounds} iii), $c_\text{min}\leq c\leq c_\text{max}$, consequently $\frac{c-c_\text{max}}{2}\leq 0\leq \frac{c-c_\text{min}}{2}$, i.e. $0\in I$. Moreover $\Delta<2$ implies that $I$ has length $\frac{c_\text{max}-c_\text{min}}{2}=\frac{\Delta}{2}<1$, so $I$ contains no integer except $0$ as desired.
\end{proof}

\begin{remark}
\normalfont Lemma~\ref{lemma:sufficient_conditions_Delta<2} also applies when $c_\text{max}=c_\text{min}$, in which case the closed interval degenerates into a point.
\end{remark}

\indent The following corollary of Lemma~\ref{lemma:sufficient_conditions_Delta<2} states a sufficient condition in terms of integers represented by the quadratic form $Q_X$ (Subsection~\ref{subsection:presenting_Q_X}).

\begin{coro}\label{coro:sufficient_conditions_Q_X_Delta<2}
Assume $\Delta<2$ and let $d:=\det E(K)^0$. If $Q_X$ represents an integer not divisible by $d$ in the interval $[d\cdot(2+2k-c_\text{max}),d\cdot(2+2k-c_\text{min})]$, then $P_1\cdot P_2=k$ for some $P_1,P_2\in E(K)$.
\end{coro}

\begin{proof}
Let $a_1,...,a_r\in\Bbb{Z}$ such that $Q_X(a_1,...,a_r)\in[d\cdot (2+2k-c_\text{max}),d\cdot(2+2k-c_\text{min})]$ with $d\nmid Q_X(a_1,...,a_r)$. If $P_1,...,P_r$ are generators of the free part of $E(K)$, let $P:=a_1P_1+...+a_rP_r$. Then $d\nmid Q_X(a_1,...,a_r)=d\cdot h(P)$, which implies that $h(P)\notin\Bbb{Z}$. In particular $P\notin E(K)^0$ since $E(K)^0$ is an integer lattice. Moreover $h(P)=\frac{1}{d}Q_X(a_1,...,a_r)\in [2+2k-c_\text{max},2+2k-c_\text{min}]$ and we are done by Lemma~\ref{lemma:sufficient_conditions_Delta<2}.
\end{proof}

\indent The next corollary, although weaker than Corollary~\ref{coro:sufficient_conditions_Q_X_Delta<2}, is more practical for concrete examples and is frequently used in Subsection~\ref{subsection:surfaces_with_a_1-gap}. It does not involve finding integers represented by $Q_X$, but only finding perfect squares in an interval depending on the minimal norm $\mu$ (Definition~\ref{def:height_minimal_norm}).

\begin{coro}\label{coro:sufficient_conditions_mu_Delta<2}
Assume $\Delta<2$. If there is a perfect square $n^2\in\left[\frac{2+2k-c_\text{max}}{\mu},\frac{2+2k-c_\text{min}}{\mu}\right]$ such that $n^2\mu\notin\Bbb{Z}$, then $P_1\cdot P_2=k$ for some $P_1,P_2\in E(K)$.
\end{coro}
\begin{proof}
Take $P\in E(K)$ such that $h(P)=\mu$. Since $h(nP)=n^2\mu\notin\Bbb{Z}$, we must have $nP\notin E(K)^0$ as $E(K)^0$ is an integer lattice. Moreover $h(nP)=n^2\mu\in[2+2k-c_\text{max},2+2k-c_\text{min}]$ and we are done by Lemma~\ref{lemma:sufficient_conditions_Delta<2}.
\end{proof}

\subsection{The case $\Delta=2$}\label{subsection:the_case_Delta=2}\
\indent The statement of sufficient conditions for $\Delta=2$ is almost identical to the one for $\Delta<2$, the only difference being that the closed interval Lemma~\ref{lemma:sufficient_conditions_Delta<2} is substituted by a right half-open interval in Lemma~\ref{lemma:sufficient_conditions_Delta=2}. This small change, however, is associated with a technical difficulty in the case of a section with minimal contribution term, hence the separate treatment for $\Delta=2$.

The results are presented in the following order. First we prove a statement about sections with minimal contribution term (Lemma~\ref{lemma:minimal_contribution_Delta=2}). Next we provide sufficient conditions when $\Delta=2$ in Lemma~\ref{lemma:sufficient_conditions_Delta=2}, then prove Corollaries~\ref{coro:sufficient_conditions_Q_X_Delta=2} and \ref{coro:sufficient_conditions_mu_Delta=2}.

\begin{lemma}\label{lemma:minimal_contribution_Delta=2}
Assume $\Delta=2$. If there is $P\in E(K)$ such that {\normalfont $\sum_{v\in R}\text{contr}_v(P)=c_\text{min}$}, then $P\cdot Q=P\cdot O+1$ for every $Q\in E(K)_\text{tor}\setminus\{O\}$.
\end{lemma}

\begin{proof}
If $Q\in E(K)_\text{tor}\setminus\{O\}$, then $Q\cdot O=0$ by Theorem~\ref{thm:torsion_sections_disjoint}. By the height formula (\ref{equation:height_formula_PQ}), 
$$0=\langle P,Q\rangle=1+P\cdot O+0-P\cdot Q-\sum_v\text{contr}_v(P,Q).\,\,(*)$$

Hence it suffices to show that $\text{contr}_v(P,Q)=0$ $\forall v\in R$. By Lemma~\ref{lemma:bounds_are_actually_bounds}~iv), $\text{contr}_{v'}(P)=c_\text{min}$ for some $v'$ and $\text{contr}_v(P)=0$ for all $v\neq v'$. In particular $P$ meets $\Theta_{v,0}$, hence $\text{contr}_v(P,Q)=0$ for all $v\neq v'$. Thus from $(*)$ we see that $\text{contr}_{v'}(P,Q)$ is an integer, which we prove is $0$.

We claim that $T_{v'}=A_1$, so that $\text{contr}_{v'}(P,Q)=0$ or $\frac{1}{2}$ by Table~\ref{table:local_contributions}. In this case, as $\text{contr}_{v'}(P,Q)$ is an integer, it must be $0$, and we are done. To see that $T_{v'}=A_1$ we analyse $\text{contr}_{v'}(P)$. Since $\Delta=2$, then $c_\text{min}=\frac{1}{2}$ by Table~\ref{table:Delta=2} and $\text{contr}_{v'}(P)=c_\text{min}=\frac{1}{2}$. By Table~\ref{table:local_contributions}, this only happens if $T_{v'}=A_{n-1}$ and $\frac{1}{2}=\frac{i(n-i)}{n}$ for some $0\leq i<n$. The only possibility is $i=1$, $n=2$, $T_{v'}=A_1$.
\end{proof}

\indent With the aid of Lemma~\ref{lemma:minimal_contribution_Delta=2} we are able to state sufficient conditions for having $P_1\cdot P_2=k$ for some $P_1,P_2\in E(K)$ when $\Delta=2$.

\begin{lemma}\label{lemma:sufficient_conditions_Delta=2}
Assume that $\Delta=2$ and let $k\in\Bbb{Z}_{\geq 0}$. If $h(P)\in[2+2k-c_\text{max}, 2+2k-c_\text{min})$ for some $P\notin E(K)^0$, then $P_1\cdot P_2=k$ for some $P_1,P_2\in E(K)$.
\end{lemma}

\begin{proof}
By the height formula (\ref{equation:height_formula_P}), $h(P)=2+2(P\cdot O)-c$, where $c:=\sum_v\text{contr}_v(P)$. We repeat the arguments from Lemma~\ref{lemma:sufficient_conditions_Delta<2}, in this case with the right half-open interval, so that the hypothesis that $h(P)\in [2+2k-c_\text{max},2+2k-c_\text{min})$, implies that $P\cdot O-k$ is an integer in $I':=\left[\frac{c-c_\text{max}}{2},\frac{c-c_\text{min}}{2}\right)$. 

Since $I'$ is half-open with length $\frac{c_\text{max}-c_\text{min}}{2}=\frac{\Delta}{2}=1$, then $I'$ contains exactly one integer. If $0\in I'$, then $P\cdot O-k=0$, i.e. $P\cdot O=k$ and we are done. Hence we assume $0\notin I'$. 

We claim that $P\cdot O=k-1$. First, notice that if $c>c_\text{min}$, then the inequalities $c_\text{min}<c\leq c_\text{max}$ give $\frac{c-c_\text{max}}{2}\leq 0<\frac{c-c_\text{min}}{2}$, i.e. $0\in I'$, which is a contradiction. Hence $c=c_\text{min}$. Since $\Delta=2$, then $I'=[-1,0)$. Thus $P\cdot O-k=-1$, i.e. $P\cdot O=k-1$, as claimed. 

Finally, let $Q\in E(K)_\text{tor}\setminus\{O\}$, so that $P\cdot Q=P\cdot O+1=k$ by Lemma~\ref{lemma:minimal_contribution_Delta=2} and we are done. We remark that $E(K)_\text{tor}$ is not trivial by Table~\ref{table:Delta=2}, therefore such $Q$ exists.
\end{proof}

\indent The following corollaries are analogues to Corollary~\ref{coro:sufficient_conditions_Q_X_Delta<2} and Corollary~\ref{coro:sufficient_conditions_mu_Delta<2} adapted to $\Delta=2$. Similarly to the case $\Delta<2$, Corollary~\ref{coro:sufficient_conditions_Q_X_Delta=2} is stronger than Corollary~\ref{coro:sufficient_conditions_mu_Delta=2}, although the latter is more practical for concrete examples. We remind the reader that $\mu$ denotes the minimal norm (Definition~\ref{def:height_minimal_norm}).

\begin{coro}\label{coro:sufficient_conditions_Q_X_Delta=2}
Assume $\Delta=2$ and let $d:=\det E(K)^0$. If $Q_X$ represents an integer not divisible by $d$ in the interval $[d\cdot(2+2k-c_\text{max}),d\cdot(2+2k-c_\text{min}))$, then $P_1\cdot P_2=k$ for some $P_1,P_2\in E(K)$.
\end{coro}

\begin{proof}
We repeat the arguments in Corollary~\ref{coro:sufficient_conditions_Q_X_Delta<2}, in this case with the half-open interval.
\end{proof}

\begin{coro}\label{coro:sufficient_conditions_mu_Delta=2}
Assume $\Delta=2$. If there is a perfect square $n^2\in\left[\frac{2+2k-c_\text{max}}{\mu},\frac{2+2k-c_\text{min}}{\mu}\right)$ such that $n^2\mu\notin\Bbb{Z}$, then $P_1\cdot P_2=k$ for some $P_1,P_2\in E(K)$.
\end{coro}

\begin{proof}
We repeat the arguments in Corollary~\ref{coro:sufficient_conditions_mu_Delta<2}, in this case with the half-open interval.
\end{proof}

\section{Necessary and sufficient conditions when $\Delta\leq 2$}\label{section:necessary_and_sufficient_Delta<=2}\
\indent For completeness, we present a unified statement of necessary and sufficient conditions for having $P_1\cdot P_2=k$ for some $P_1,P_2\in E(K)$ assuming $\Delta\leq 2$, which follows naturally from results in Sections~\ref{section:necessary_conditions} and \ref{section:sufficient_conditions_Delta<=2}.

\begin{lemma}\label{lemma:necessary_sufficient_conditions_Delta<=2}
Assume $\Delta\leq 2$ and let $k\in\Bbb{Z}_{\geq 0}$. Then $P_1\cdot P_2=k$ for some $P_1,P_2\in E(K)$ if and only if one of the following holds:
\begin{enumerate}[i)]
\item $h(P)=2+2k$ for some $P\in E(K)^0$.
\item {\normalfont $h(P)\in[2+2k-c_\text{max},2+2k-c_\text{min})$} for some $P\notin E(K)^0$.
\item {\normalfont $h(P)=2+2k-c_\text{min}$} and {\normalfont $\sum_{v\in R}\text{contr}_v(P)=c_\text{min}$} for some $P\in E(K)$.
\end{enumerate}
\end{lemma}

\begin{proof}
If i) or iii) holds, then $P\cdot O=k$ directly by the height formula (\ref{equation:height_formula_P}). But if ii) holds, it suffices to to apply Lemma~\ref{lemma:sufficient_conditions_Delta<2} when $\Delta<2$ and by Lemma~\ref{lemma:sufficient_conditions_Delta=2} when $\Delta=2$.

Conversely, let $P_1\cdot P_2=k$. Without loss of generality, we may assume that $P_2=O$, so that $P_1\cdot O=k$. By the height formula (\ref{equation:height_formula_P}), $h(P_1)=2+2k-c$, where $c:=\sum_v\text{contr}_v(P_1)$. 

If $c=0$, then $P_1\in E(K)^0$ and $h(P_1)=2+2k$, so i) holds. Hence we let $c\neq 0$, i.e. $P_1\notin E(K)^0$, so that $c_\text{min}\leq c\leq c_\text{max}$ by Lemma~\ref{lemma:bounds_are_actually_bounds}. In case $c=c_\text{min}$, then $h(P_1)=2+2k-c_\text{min}$ and iii) holds. Otherwise $c_\text{min}<c\leq c_\text{max}$, which implies $2+2k-c_\text{max}\leq h(P_1)<2+2k-c_\text{min}$, so ii) holds.
\end{proof}

\section{Summary of sufficient conditions}\label{section:summary_sufficient_conditions}\
\indent For the sake of clarity, we summarize in a single proposition all sufficient conditions for having $P_1\cdot P_2=k$ for some $P_1,P_2\in E(K)$ proven in this chapter.

\begin{prop}\label{prop:summary_of_sufficient_conditions}
Let $k\in\Bbb{Z}_{\geq 0}$. If one of the following holds, then $P_1\cdot P_2=k$ for some $P_1,P_2\in E(K)$.
\begin{enumerate}[1)]
\item $h(P)=2+2k$ for some $P\in E(K)^0$.
\item $h(P)=2k$ for some $P\in E(K)^0$ and $E(K)_\text{tor}$ is not trivial.
\item $\Delta<2$ and there is a perfect square $n^2\in \left[\frac{2+2k-c_\text{max}}{\mu},\frac{2+2k-c_\text{min}}{\mu}\right]$ with $n^2\mu\notin\Bbb{Z}$, where $\mu$ is the minimal norm (Definition~\ref{def:height_minimal_norm}). In case $\Delta=2$, consider the right half-open interval.
\item $\Delta<2$ and the quadratic form $Q_X$ represents an integer not divisible by $d:=\det E(K)^0$ in the interval $[d\cdot(2+2k-c_\text{max}),d\cdot(2+2k-c_\text{min})]$. In case $\Delta=2$, consider the right half-open interval.
\end{enumerate}
\end{prop}

\begin{proof}
In 1) a height calculation gives $2+2k=h(P)=2+2(P\cdot O)-0$, so $P\cdot O=k$. For 2), we apply Lemma~\ref{lemma:minimal_contribution_Delta=2} to conclude that $P\cdot Q=k$ for any $Q\in E(K)_\text{tor}\setminus\{O\}$. In 3) we use Corollary~\ref{coro:sufficient_conditions_mu_Delta<2} when $\Delta<2$ and Corollary~\ref{coro:sufficient_conditions_mu_Delta=2} when $\Delta=2$. In 4), we apply Corollary~\ref{coro:sufficient_conditions_Q_X_Delta<2} if $\Delta<2$ and Corollary~\ref{coro:sufficient_conditions_Q_X_Delta=2} if $\Delta=2$.
\end{proof}

\section{Applications}\label{section:applications}\
\indent We prove the four main theorems of this chapter, which are independent applications of the results from Sections~\ref{section:necessary_conditions} and \ref{section:sufficient_conditions_Delta<=2}. The first two are general attempts to describe when and how gap numbers occur: Theorem~\ref{thm:gap_free_r>=5} tells us that large Mordell-Weil groups prevent the existence of gap numbers, more precisely for Mordell-Weil rank $r\geq 5$; in Theorem~\ref{thm:gaps_probability_1_r=1,2} we show that for a smaller Mordell-Weil group, more precisely when $r\leq 2$, gap numbers occur with probability $1$. The last two theorems, on the other hand, deal with explicit values of gap numbers: Theorem~\ref{thm:identification_of_gaps_r=1} provides a complete description of gap numbers in certain cases, whereas Theorem~\ref{thm:surfaces_with_a_1-gap} is a classification of all cases with a $1$-gap.

\subsection{No gap numbers in $r\geq 5$}\label{subsection:gap_free_r>=5}\
\indent We show that if $E(K)$ has rank $r\geq 5$, then $X$ is gap-free. Our strategy is to prove that for every $k\in\Bbb{Z}_{\geq 0}$ there is some $P\in E(K)^0$ such that $h(P)=2+2k$, and by Proposition~\ref{prop:summary_of_sufficient_conditions} 1) we are done. We accomplish this in two steps. First we show that this holds when there is an embedding of $A_1^{\oplus}$ or of $A_4$ in $E(K)^0$ (Lemma~\ref{lemma:sublattice_A1_times_4_or_A4}). Second, we show that if $r\geq 5$, then such embedding exists, hence $X$ is gap-free (Theorem~\ref{thm:gap_free_r>=5}).

\begin{lemma}\label{lemma:sublattice_A1_times_4_or_A4}
Assume $E(K)^0$ has a sublattice isomorphic to $A_1^{\oplus 4}$ or $A_4$. Then for every $\ell\in\Bbb{Z}_{\geq 0}$ there is $P\in E(K)^0$ such that $h(P)=2\ell$.
\end{lemma}

\begin{proof}
First assume $A_1^{\oplus 4}\subset E(K)^0$ and let $P_1,P_2,P_3,P_4$ be generators for each factor $A_1$ in $A_1^{\oplus 4}$. Then $h(P_i)=2$ and $\langle P_i,P_j\rangle=0$ for distinct $i,j=1,2,3,4$. By Lagrange's four-square theorem \cite[\S 20.5]{HardyWright} there are integers $a_1,a_2,a_3,a_4$ such that $a_1^2+a_2^2+a_3^2+a_4^2=\ell$. Defining $P:=a_1P_1+a_2P_2+a_3P_3+a_4P_4\in A_1^{\oplus 4}\subset E(K)^0$, we have
$$h(P)=2a_1^2+2a_2^2+2a_3^2+2a_4^2=2\ell.$$

Now let $A_4\subset E(K)^0$ with generators $P_1,P_2,P_3,P_4$. Then $h(P_i)=2$ for $i=1,2,3,4$ and $\langle P_i,P_{i+1}\rangle=-1$ for $i=1,2,3$. We need to find integers $x_1,...,x_4$ such that $h(P)=2\ell$, where $P:=x_1P_1+...+x_4P_4\in A_4\subset E(K)^0$. Equivalently, we need that
$$\ell=\frac{1}{2}\langle P,P\rangle=x_1^2+x_2^2+x_3^2+x_4^2-x_1x_2-x_2x_3-x_3x_4.$$

Therefore $\ell$ must be represented by $q(x_1,...,x_4):=x_1^2+x_2^2+x_3^2+x_4^2-x_1x_2-x_2x_3-x_3x_4$. We prove that $q$ represents all positive integers. Notice that $q$ is positive-definite, since it is induced by $\langle\cdot,\cdot\rangle$. By Bhargava-Hanke's 290-theorem \cite{BhargavaHanke}[Thm. 1], $q$ represents all positive integers if and only if it represents the following integers:
$$2, 3, 5, 6, 7, 10, 13, 14, 15, 17, 19, 21, 22, 23, 26,$$
$$29, 30, 31, 34, 35, 37, 42, 58, 93, 110, 145, 203, 290.$$

The representation for each of the above is found in Table~\ref{table:representation_critical_integers}.
\end{proof}

\indent We now prove the main theorem of this section.

\begin{teor}\label{thm:gap_free_r>=5} 
If $r\geq 5$, then $X$ is gap-free. 
\end{teor}

\begin{proof}
We show that for every $k\geq 0$ there is $P\in E(K)^0$ such that $h(P)=2+2k$, so that by Proposition~\ref{prop:summary_of_sufficient_conditions} 1) we are done. By Lemma~\ref{lemma:sublattice_A1_times_4_or_A4} it suffices to prove that $E(K)^0$ has a sublattice isomorphic to $A_1^{\oplus 4}$ or $A_4$.

The cases with $r\geq 5$ are No. 1-7 (Table~\ref{table:MWL_data}). In No. 1-6, $E(K)^0= E_8,E_7,E_6,D_6,D_5,A_5$ respectively. Each of these admit an $A_4$ sublattice \cite[Lemmas 4.2, 4.3]{Nishiyama}. In No. 7 we claim that $E(K)^0=D_4\oplus A_1$ has an $A_1^{\oplus 4}$ sublattice. This is the case because $D_4$ admits an $A_1^{\oplus 4}$ sublattice \cite[Lemma 4.5 (iii)]{Nishiyama}.

% Indeed, let $P_1,P_2,P_3,P_4$ be generators for $D_4$ and $P_5$ a generator for $A_1$. Then $L:=\Bbb{Z}\langle P_1,P_3,P_4,P_5\rangle$ is clearly isomorphic to $A_1^{\oplus 4}$ (Figure \ref{D_4_A_1}).
\end{proof}

%\begin{figure}[h!]
%\begin{center}
%\includegraphics[scale=0.7]{D_4_A_1}
%\caption{Generators $P_1,P_3,P_4,P_5$ for the sublattice $L\simeq A_1^{\oplus 4}\subset D_4\oplus A_1$.}\label{D_4_A_1}
%\end{center}
%\end{figure}

\begin{table}[h]
\begin{center}
\centering
\begin{tabular}{c|c} 
$n$ & $x_1,x_2,x_3,x_4$ with $x_1^2+x_2^2+x_3^2+x_4^2-x_1x_2-x_2x_3-x_3x_4=n$\\
\hline
$1$ & $1,0,0,0$\\
$2$ & $1,0,1,0$\\
$3$ & $1,1,2,0$\\
$5$ & $1,0,2,0$\\
$6$ & $1,1,-2,-1$\\
$7$ & $1,1,-2,0$\\
$10$ & $1,0,3,0$\\
$13$ & $2,0,3,0$\\
$14$ & $1,2,5,1$\\
$15$ & $1,5,5,2$\\
$17$ & $1,0,4,0$\\
$19$ & $1,5,3,-1$\\
$21$ & $1,5,0,0$\\
$22$ & $1,5,0,-1$\\
$23$ & $1,6,6,2$\\
$26$ & $1,0,5,0$\\
$29$ & $2,0,5,0$\\
$30$ & $1,5,0,-3$\\
$31$ & $1,3,-4,-2$\\
$34$ & $3,0,5,0$\\
$35$ & $1,2,-2,4$\\
$37$ & $1,0,6,0$\\
$42$ & $1,1,-4,3$\\
$58$ & $3,0,7,0$\\
$93$ & $1,1,-10,0$\\
$110$ & $1,-2,3,-8$\\
$145$ & $1,0,12,0$\\
$203$ & $1,-5,-9,8$\\
$290$ & $1,0,17,0$\\
\end{tabular}\caption{Representation of the critical integers in Bhargava-Hanke's 290-theorem.}\label{table:representation_critical_integers}
\end{center}
\end{table}

\subsection{Gaps with probability $1$ in $r\leq 2$}\label{subsection:gaps_probability_1_r=1,2}\
\indent Fix a rational elliptic surface $\pi:X\to\P^1$ with Mordell-Weil rank $r\leq 2$. We prove that if $k$ is a uniformly random natural number, then $k$ is a gap number with probability $1$. More precisely, if $G:=\{k\in\Bbb{N}\mid k\text{ is a gap number of }X\}$ is the set of gap numbers, then $G\subset\Bbb{N}$ has density $1$, i.e. 
$$d(G):=\lim_{n\to\infty}\frac{\#G\cap\{1,...,n\}}{n}=1.$$

\indent We adopt the following strategy. If $k\in\Bbb{N}\setminus G$, then $P_1\cdot P_2=k$ for some $P_1,P_2\in E(K)$ and by Corollary~\ref{coro:necessary_conditions_Q_X} the quadratic form $Q_X$ represents some integer $t$ depending on $k$. This defines a function $f:\Bbb{N}\setminus G\to T$, where $T$ is the set of integers represented by $Q_X$. Since $Q_X$ is a quadratic form on $r\leq 2$ variables, $T$ has density $0$ in $\Bbb{N}$ by Lemma~\ref{lemma:representable_integers_density_0}. By analyzing the pre-images of $f$, in Theorem~\ref{thm:gaps_probability_1_r=1,2} we conclude that $d(\Bbb{N}\setminus G)=d(T)=0$, hence $d(G)=1$ as desired.

\begin{lemma}\label{lemma:representable_integers_density_0}
Let $Q$ be a positive-definite quadratic form on $r=1,2$ variables with integer coefficients. Then the set of integers represented by $Q$ has density $0$ in $\Bbb{N}$.
\end{lemma}

\begin{proof}
Let $S$ be the set of integers represented by $Q$. If $d$ is the greatest common divisor of the coefficients of $Q$, let $S'$ be the set of integers representable by the primitive form $Q':=\frac{1}{d}\cdot Q$. By construction $S'$ is a rescaling of $S$, so $d(S)=0$ if and only if $d(S')=0$.

If $r=1$, then $Q'(x_1)=x_1^2$ and $S'$ is the set of perfect squares, so clearly $d(S')=0$. If $r=2$, then $Q'$ is a binary quadratic form and the number  of elements in $S'$ bounded from above by $x>0$ is given by $C\cdot \frac{x}{\sqrt{\log x}}+o(x)$ with $C>0$ a constant and $\lim_{x\to\infty}\frac{o(x)}{x}=0$ \cite[p. 91]{Bernays}. Thus 
$$d(S')=\lim_{x\to\infty}\frac{C}{\sqrt{\log x}}+\frac{o(x)}{x}=0.$$
\end{proof}

We now prove the main result of this section.

\begin{teor}\label{thm:gaps_probability_1_r=1,2}
Let $\pi:X\to\P^1$ be a rational elliptic surface with Mordell-Weil rank $r\leq 2$. Then the set $G:=\{k\in\Bbb{N}\mid k\text{ is a gap number of }X\}$ of gap numbers of $X$ has density $1$ in $\Bbb{N}$.
\end{teor}

\begin{proof}
If $r=0$, then all sections are torsion sections and the claim is obvious by Theorem~\ref{thm:torsion_sections_disjoint}. Assume $r=1,2$. We prove that $S:=\Bbb{N}\setminus G$ has density $0$. If $S$ is finite, there is nothing to prove. Otherwise, let $k_1<k_2<...$ be the increasing sequence of all elements of $S$. By Corollary~\ref{coro:necessary_conditions_Q_X}, for each $n$ there is some $t_n\in J_{k_n}:=[d\cdot (2+2k_n-c_\text{max}),d\cdot(2+2k_n)]$ represented by the quadratic form $Q_X$. Let $T$ be the set of integers represented by $Q_X$ and define the function $f\colon S\to T$ by $k_n\mapsto t_n$. Since $Q_X$ has $r=1,2$ variables, $T$ has density $0$ by Lemma~\ref{lemma:representable_integers_density_0}.

For $N>0$, let $S_N:=S\cap\{1,...,N\}$ and $T_N:=T\cap\{1,...,N\}$. Since $T$ has density zero, $\#T_N=o(N)$, i.e. $\frac{\# T_N}{N}\to 0$ when $N\to\infty$ and we need to prove that $\#S_N=o(N)$. We analyze the function $f$ restricted to $S_N$. Notice that as $t_n\in J_{k_n}$, then $k_n\leq N$ implies $t_n\leq d\cdot (2+2k_n)\leq d\cdot(2+2N)$. Hence the restriction $g:=f|_{S_N}$ can be regarded as a function $g:S_N\to T_{d\cdot(2+2k)}$. 

We claim that $\#g^{-1}(t)\leq 2$ for all $t\in T_{d\cdot(2+2N)}$, in which case $\#S_N\leq 2\cdot\#T_{d\cdot(2+2N)}=o(N)$ and we are done. Assume by contradiction that $g^{-1}(t)$ contains three distinct elements, say $k_{\ell_1}<k_{\ell_2}<k_{\ell_3}$ with $t=t_{\ell_1}=t_{\ell_2}=t_{\ell_3}$. Since $t_{\ell_i}\in J_{k_{\ell_i}}$ for each $i=1,2,3$, then $t\in J_{k_{\ell_1}}\cap J_{k_{\ell_2}}\cap J_{k_{\ell_3}}$. We prove that $J_{k_{\ell_1}}$ and $J_{k_{\ell_3}}$ are disjoint, which yields a contradiction. Indeed, since $k_{\ell_1}<k_{\ell_2}<k_{\ell_3}$, in particular $k_{\ell_3}-k_{\ell_1}\geq 2$, therefore $d\cdot(2+2k_{\ell_1})\leq d\cdot(2+2k_{\ell_3}-4)$. But $c_\text{max}<4$ by Lemma~\ref{lemma:bounds_are_actually_bounds}, so $d\cdot(2+2k_{\ell_1})<d\cdot(2+2k_{\ell_3}-c_\text{max})$, i.e. $\max J_{k_{\ell_1}}<\min J_{k_{\ell_3}}$. Thus $J_{k_{\ell_1}}\cap J_{k_{\ell_3}}=\emptyset$, as desired.
\end{proof}

\newpage

\subsection{Identification of gaps when $E(K)$ is torsion-free with rank $r=1$}\label{subsection:identification_of_gaps_r=1}\
\indent The results in Subsections~\ref{subsection:gap_free_r>=5} and \ref{subsection:gaps_probability_1_r=1,2} concern the existence and the distribution of gap numbers. In the following subsections we turn our attention to finding gap numbers explicitly. In this subsection we give a complete description of gap numbers assuming $E(K)$ is torsion-free with rank $r=1$. Such descriptions are difficult in the general case, but our assumption guarantees that each $E(K),E(K)^0$ is generated by a single element and that $\Delta<2$ by Lemma~\ref{lemma:cases_where_Delta>=2}, which makes the problem more accessible. 

We organize this subsection as follows. First we point out some trivial facts about generators of $E(K),E(K)^0$ when $r=1$ in Lemma~\ref{lemma:generators_r=1}. Next we state necessary and sufficient conditions for having $P_1\cdot P_2=k$ when $E(K)$ is torsion-free with $r=1$ in Lemma~\ref{lemma:necessary_sufficient_conditions_r=1}. As an application of the latter, we prove Theorem~\ref{thm:identification_of_gaps_r=1}, which is the main result of the subsection.

\begin{lemma}\label{lemma:generators_r=1}
Let $X$ be a rational elliptic surface with Mordell-Weil rank $r=1$. If $P$ generates the free part of $E(K)$, then 
\begin{enumerate}[a)]
\item $h(P)=\mu$.
\item $1/\mu$ is an even integer.
\item $E(K)^0$ is generated by $P_0:=(1/\mu)P$ and $h(P_0)=1/\mu$.
\end{enumerate}
\end{lemma}

\begin{proof}
Item  a) is clear. Items b), c) follow from the fact that $E(K)^0$ is an even lattice and that $E(K)\simeq L^*\oplus E(K)_\text{tor}$, where $L:=E(K)^0$ \cite[Main Thm.]{OguisoShioda}.
\end{proof}

\indent In what follows we use Lemma~\ref{lemma:generators_r=1} and results from Section~\ref{section:sufficient_conditions_Delta<=2} to state necessary and sufficient conditions for having $P_1\cdot P_2=k$ for some $P_1,P_2\in E(K)$ in case $E(K)$ is torsion-free with $r=1$. 

\begin{lemma}\label{lemma:necessary_sufficient_conditions_r=1}
Assume $E(K)$ is torsion-free with rank $r=1$. Then $P_1\cdot P_2=k$ for some $P_1,P_2\in E(K)$ if and only if one of the following holds:
\begin{enumerate}[i)]
\item $\mu\cdot(2+2k)$ is a perfect square.
\item There is a perfect square $n^2\in\left[\frac{2+2k-c_\text{max}}{\mu},\frac{2+2k-c_\text{min}}{\mu}\right]$ such that $\mu\cdot n\notin\Bbb{Z}$.
\end{enumerate}
\end{lemma}

\begin{proof}
By Lemma~\ref{lemma:generators_r=1}, $E(K)$ is generated by some $P$ with $h(P)=\mu$ and $E(K)^0$ is generated by $P_0:=n_0P$, where $n_0:=\frac{1}{\mu}\in 2\Bbb{Z}$.

First assume that $P_1\cdot P_2=k$ for some $P_1,P_2$. Without loss of generality we may assume $P_2=O$. Let $P_1=nP$ for some $n\in\Bbb{Z}$. We show that $P_1\in E(K)^0$ implies i) while $P_1\notin E(K)^0$ implies ii). 

If $P_1\in E(K)^0$, then $n_0\mid n$, hence $P_1=nP=mP_0$, where $m:=\frac{n}{n_0}$. By the height formula (\ref{equation:height_formula_P}), $2+2k=h(P_1)=h(mP_0)=m^2\cdot \frac{1}{\mu}$. Hence $\mu\cdot (2+2k)=m^2$, i.e. i) holds. 

If $P_1\notin E(K)^0$, then $n_0\nmid n$, hence $\mu\cdot n=\frac{n}{n_0}\notin\Bbb{Z}$. Moreover, $h(P_1)=n^2h(P)=n^2\mu$ and by the height formula (\ref{equation:height_formula_P}), $n^2\mu=h(P)=2+2k-c$, where $c:=\sum_v\text{contr}_v(P_1)\neq 0$. The inequalities $c_\text{min}\leq c\leq c_\text{max}$ then give $\frac{2+2k-c_\text{max}}{\mu}\leq n^2\leq \frac{2+2k-c_\text{min}}{\mu}$. Hence ii) holds.

Conversely, assume i) or ii) holds. Since $E(K)$ is torsion-free, $\Delta<2$ by Lemma~\ref{lemma:cases_where_Delta>=2}, so we may apply Lemma~\ref{lemma:sufficient_conditions_Delta<2}. If i) holds, then $\mu\cdot(2+2k)=m^2$ for some $m\in\Bbb{Z}$. Since $mP_0\in E(K)^0$ and $h(mP_0)=\frac{m^2}{\mu}=2+2k$, we are done by Lemma~\ref{lemma:sufficient_conditions_Delta<2} i). If ii) holds, the condition $\mu\cdot n\notin\Bbb{Z}$ is equivalent to $n_0\nmid n$, hence $nP\notin E(K)^0$. Moreover $n^2\in\left[\frac{2+2k-c_\text{max}}{\mu},\frac{2+2k-c_\text{min}}{\mu}\right]$, implies $h(nP)=n^2\mu\in[2+2k-c_\text{max},2+2k-c_\text{min}]$. By Lemma~\ref{lemma:sufficient_conditions_Delta<2} ii), we are done.
\end{proof}

\newpage

\indent An application of Lemma~\ref{lemma:necessary_sufficient_conditions_r=1} to all possible cases where $E(K)$ is torsion-free with rank $r=1$ yields the main result of this subsection.

\begin{teor}\label{thm:identification_of_gaps_r=1}
If $E(K)$ is torsion-free with rank $r=1$, then all the gap numbers of $X$ are described in Table~\ref{table:description_gaps_r=1}.
\begin{table}[h]
\begin{center}
\centering
\begin{tabular}{cccc} 
\hline
\multirow{2}{*}{No.} & \multirow{2}{*}{$T$} & \multirow{2}{*}{$\begin{matrix}k\text{ is a gap number}\Leftrightarrow \text{none of}\\ \text{the following are perfect squares}\end{matrix}$} & \multirow{2}{*}{first gap numbers}\\ 
& \\
\hline
\multirow{2}{*}{43} & \multirow{2}{*}{$E_7$} & \multirow{2}{*}{$k+1$, $4k+1$} & \multirow{2}{*}{$1,4$}\\ 
& \\
\hline
\multirow{2}{*}{45} & \multirow{2}{*}{$A_7$} & \multirow{2}{*}{$\frac{k+1}{4}$, $16k,...,16k+9$} & \multirow{2}{*}{$8,11$}\\ 
& \\
\hline
\multirow{2}{*}{46} & \multirow{2}{*}{$D_7$} & \multirow{2}{*}{$\frac{k+1}{2}$, $8k+1,...,8k+4$} & \multirow{2}{*}{$2,5$}\\ 
& \\
\hline
\multirow{2}{*}{47} & \multirow{2}{*}{$A_6\oplus A_1$} & \multirow{2}{*}{$\frac{k+1}{7}$, $28k-3,...,28k+21$} & \multirow{2}{*}{$12,16$}\\ 
& \\
\hline
\multirow{2}{*}{49} & \multirow{2}{*}{$E_6\oplus A_1$} & \multirow{2}{*}{$\frac{k+1}{3}$, $12k+1,...,12k+9$} & \multirow{2}{*}{$3,7$}\\ 
& \\
\hline
\multirow{2}{*}{50} & \multirow{2}{*}{$D_5\oplus A_2$} & \multirow{2}{*}{$\frac{k+1}{6}$, $24k+1,...,24k+16$} & \multirow{2}{*}{$6,11$}\\ 
& \\
\hline
\multirow{2}{*}{55} & \multirow{2}{*}{$A_4\oplus A_3$} & \multirow{2}{*}{$\frac{k+1}{10}$, $40k-4,...,40k+25$} & \multirow{2}{*}{$16,20$}\\ 
& \\
\hline
\multirow{2}{*}{56} & \multirow{2}{*}{$A_4\oplus A_2\oplus A_1$} & \multirow{2}{*}{$\frac{k+1}{15}$, $60k-11,...,60k+45$} & \multirow{2}{*}{$22,27$}\\ 
& \\
\hline
\end{tabular}\caption{Description of gap numbers when $E(K)$ is torsion-free with $r=1$.}\label{table:description_gaps_r=1}
\end{center}
\end{table}
\end{teor}

\begin{proof}
For the sake of brevity we restrict ourselves to No. 55. The other cases are treated similarly. Here $c_\text{max}=\frac{2\cdot 3}{5}+\frac{2\cdot 2}{4}=\frac{11}{5}$, $c_\text{min}=\min\left\{\frac{4}{5},\frac{3}{4}\right\}=\frac{3}{4}$ and $\mu=1/20$. 

By Lemma~\ref{lemma:necessary_sufficient_conditions_r=1}, $k$ is a gap number if and only if neither i) nor ii) occurs. Condition i) is that $\frac{2+2k}{20}=\frac{k+1}{10}$ is a perfect square. Condition ii) is that $\left[\frac{2+2k-c_\text{max}}{\mu},\frac{2+2k-c_\text{min}}{\mu}\right]=[40k-4,40k+25]$ contains some $n^2$ with $20\nmid n$. We check that $20\nmid n$ for every $n$ such that $n^2=40k+\ell$, with $\ell=-4,...,25$. Indeed, if $20\mid n$, then $400\mid n^2$ and in particular $40\mid n^2$. Then $40\mid (n^2-40k)=\ell$, which is absurd.
\end{proof}

\subsection{Surfaces with a $1$-gap}\label{subsection:surfaces_with_a_1-gap}\
\indent In Subsection~\ref{subsection:identification_of_gaps_r=1} we take each case in Table~\ref{table:description_gaps_r=1} and describe all its gap numbers. In this subsection we do the opposite, namely, we fix a number and describe all cases having it as a gap number. We remind the reader that our motivating problem was to determine when there are sections $P_1,P_2$ such that $P_1\cdot P_2=1$, which induce a conic bundle having $P_1+P_2$ as a reducible fiber. The answer for this question is the main theorem of this subsection:

\begin{teor}\label{thm:surfaces_with_a_1-gap}
Let $X$ be a rational elliptic surface. Then $X$ has a $1$-gap if and only if $r=0$ or $r=1$ and $\pi$ has a {\normalfont $\text{III}^*$} fiber.
\end{teor}
\newpage
Our strategy for the proof is the following. We already know that a $1$-gap exists whenever $r=0$ (Theorem~\ref{thm:torsion_sections_disjoint}) or when $r=1$ and $\pi$ has a $\text{III}^*$ fiber (Theorem~\ref{thm:identification_of_gaps_r=1}, No. 43). Conversely, we need to find $P_1,P_2$ with $P_1\cdot P_2=1$ in all cases with $r\geq 1$ and $T\neq E_7$. 

First we introduce two lemmas, which solve most cases with little computation, and leave the remaining ones for the proof of Theorem~\ref{thm:surfaces_with_a_1-gap}. In both Lemma~\ref{lemma:E0_height_4} and Lemma~\ref{lemma:E0=An} our goal is to analyze the narrow lattice $E(K)^0$ and apply Proposition~\ref{prop:summary_of_sufficient_conditions} to detect cases without a $1$-gap.

\begin{lemma}\label{lemma:E0_height_4}
If one of the following holds, then $h(P)=4$ for some $P\in E(K)^0$.
\begin{enumerate}[a)]
\item The Gram matrix of $E(K)^0$ has a $4$ in its main diagonal.
\item There is an embedding of $A_n\oplus A_m$ in $E(K)^0$ for some $n,m\geq 1$.
\item There is an embedding of $A_n,D_n$ or $E_n$ in $E(K)^0$ for some $n\geq 3$.
\end{enumerate}

\end{lemma}

\begin{proof}
Case a) is trivial. Assuming b), we take generators $P_1,P_2$ from $A_n,A_m$ respectively with $h(P_1)=h(P_2)=2$. Since $A_n,A_m$ are in direct sum, $\langle P_1,P_2\rangle=0$, hence $h(P_1+P_2)=4$, as desired. If c) holds, then the fact that $n\geq 3$ allows us to choose two elements $P_1,P_2$ among the generators of $L_1=A_n,D_n$ or $E_n$ such that $h(P_1)=h(P_2)=2$ and $\langle P_1,P_2\rangle=0$. Thus $h(P_1+P_2)=4$ as claimed.
\end{proof}

\begin{coro}\label{coro:E0_height_4}
In the following cases, $X$ does not have a $1$-gap.
\begin{itemize}
\item $r\geq 3:$ all cases except possibly {\normalfont No. 20}.
\item $r=1,2:$ cases {\normalfont No. 25, 26, 30, 32-36, 38, 41, 42, 46, 52, 54, 60}.
\end{itemize}
\end{coro}

\begin{proof}
We look at column $E(K)^0$ in Table~\ref{table:MWL_data} to find which cases satisfy one of the conditions a), b), c) from Lemma~\ref{lemma:E0_height_4}. 
\begin{enumerate}[a)]
\item Applies to No. 12, 17, 19, 22, 23, 25, 30, 32, 33, 36, 38, 41, 46, 52, 54, 60.
\item Applies to No. 10, 11, 14, 15, 18, 24, 26, 34, 35, 42.
\item Applies to No. 1-10, 13, 16, 21.
\end{enumerate}

In particular, this covers all cases with $r\geq 3$ (No. 1-24) except No. 20. By Lemma~\ref{lemma:E0_height_4} in each of these cases there is $P\in E(K)^0$ with $h(P)=4$ and we are done by Proposition~\ref{prop:summary_of_sufficient_conditions} 1).
\end{proof}

\indent In the next lemma we also analyze $E(K)^0$ to detect surfaces without a $1$-gap.

\begin{lemma}\label{lemma:E0=An}
Assume $E(K)^0\simeq A_n$ for some $n\geq 1$ and that $E(K)$ has nontrivial torsion part. Then $X$ does not have a $1$-gap. This applies to cases {\normalfont No. 28, 39, 44, 48, 51, 57, 58} in Table~\ref{table:MWL_data}.
\end{lemma}

\begin{proof}
Take a generator $P$ of $E(K)^0$ with $h(P)=2$ and apply Proposition~\ref{prop:summary_of_sufficient_conditions} 2).
\end{proof}

\newpage
\indent We are ready to prove the main result of this subsection.

\begin{proof}[Proof of Theorem~\ref{thm:surfaces_with_a_1-gap}]
We need to show that in all cases where $r\geq 1$ and $T\neq E_7$ there are $P_1,P_2\in E(K)$ such that $P_1\cdot P_2=1$. This corresponds to cases No. 1-61 except 43 in Table~\ref{table:MWL_data}. 

The cases where $r=1$ and $E(K)$ is torsion-free can be solved by Theorem~\ref{thm:identification_of_gaps_r=1}, namely No. 45-47, 49, 50, 55, 56. Adding these cases to the ones treated in Corollary~\ref{coro:E0_height_4} and Lemma~\ref{lemma:E0=An}, we have therefore solved the following:
$$\text{No. }1\text{-}19,\,21\text{-}26,\,28,\,30,\,32\text{-}36,\,38,\,39,\,41\text{-}52,\,54\text{-}58,\,60.$$

For the remaining cases, we apply Proposition~\ref{prop:summary_of_sufficient_conditions} 3), which involves finding perfect squares in the interval $\left[\frac{4-c_\text{max}}{\mu},\frac{4-c_\text{min}}{\mu}\right]$ (see Table~\ref{table:perfect_squares_in_the_interval}), considering the half-open interval in the cases with $\Delta=2$ (No. 53, 61).

\begin{table}[h]
\begin{center}
\centering
\begin{tabular}{cccccc} 

\multirow{2}{*}{No.} & \multirow{2}{*}{$T$} & \multirow{2}{*}{$E(K)$} & \multirow{2}{*}{$\mu$} & \multirow{2}{*}{$I$} & \multirow{2}{*}{$n^2\in I$}\\ 
& \\
\hline
\multirow{2}{*}{\hfil 20} & \multirow{2}{*}{\hfil $A_2^{\oplus 2}\oplus A_1$} & \multirow{2}{*}{\hfil $A_2^*\oplus\langle1/6\rangle$} & \multirow{2}{*}{$\frac{1}{6}$} & \multirow{2}{*}{$[13,23]$} & \multirow{2}{*}{$4^2$}\\ 
&\\
\multirow{2}{*}{\hfil 27} & \multirow{2}{*}{\hfil $E_6$} & \multirow{2}{*}{\hfil $A_2^*$} & \multirow{2}{*}{$\frac{2}{3}$} & \multirow{2}{*}{$[4,4]$} & \multirow{2}{*}{$2^2$}\\ 
&\\
\multirow{2}{*}{\hfil 29} & \multirow{2}{*}{\hfil $A_5\oplus A_1$} & \multirow{2}{*}{\hfil $A_1^*\oplus\langle 1/6\rangle$} & \multirow{2}{*}{$\frac{1}{6}$} & \multirow{2}{*}{$[12,21]$} & \multirow{2}{*}{$4^2$}\\
& \\ 
\multirow{3}{*}{\hfil 31} & \multirow{3}{*}{\hfil $A_4\oplus A_2$} & \multirow{3}{*}{\hfil $\frac{1}{15}\left(\begin{matrix} 2 & 1\\1 & 8\end{matrix}\right)$} & \multirow{3}{*}{$\frac{2}{15}$} & \multirow{3}{*}{$[16,21]$} & \multirow{3}{*}{$4^2$}\\
& \\
& \\
\multirow{2}{*}{\hfil 37} & \multirow{2}{*}{\hfil $A_3\oplus A_2\oplus A_1$} & \multirow{2}{*}{\hfil $A_1^*\oplus\langle 1/12\rangle$} & \multirow{2}{*}{$\frac{1}{12}$} & \multirow{2}{*}{$[22,28]$} & \multirow{2}{*}{$5^2$}\\
& \\ 
\multirow{2}{*}{\hfil 40} & \multirow{2}{*}{\hfil $A_2^{\oplus 2}\oplus A_1^{\oplus 2}$} & \multirow{2}{*}{\hfil $\langle 1/6\rangle^{\oplus 2}$}  & \multirow{2}{*}{$\frac{1}{6}$} & \multirow{2}{*}{$\left[10,21\right]$} & \multirow{2}{*}{$4^2$}\\
& \\ 
\multirow{2}{*}{\hfil 53} & \multirow{2}{*}{\hfil $A_5\oplus A_1^{\oplus 2}$} & \multirow{2}{*}{\hfil $\langle 1/6\rangle\oplus \Bbb{Z}/2\Bbb{Z}$} &  \multirow{2}{*}{$\frac{1}{6}$} & \multirow{2}{*}{$[9,12]$} & \multirow{2}{*}{$3^2$}\\
& \\ 
\multirow{2}{*}{\hfil 59} & \multirow{2}{*}{\hfil $A_3\oplus A_2\oplus A_1^{\oplus 2}$} & \multirow{2}{*}{\hfil $\langle 1/12\rangle\oplus \Bbb{Z}/2\Bbb{Z}$} & \multirow{2}{*}{$\frac{1}{12}$} & \multirow{2}{*}{$[16,42]$} & \multirow{2}{*}{$4^2,5^2,6^2$}\\
& \\ 
\multirow{3}{*}{\hfil 61} & \multirow{3}{*}{\hfil $A_2^{\oplus 3}\oplus A_1$} & \multirow{3}{*}{\hfil $\langle 1/6\rangle\oplus\Bbb{Z}/3\Bbb{Z}$} & \multirow{3}{*}{$\frac{1}{6}$} & \multirow{3}{*}{$[9,12]$} & \multirow{3}{*}{$3^2$}\\
& \\ 
\hline
\end{tabular}\caption{Perfect squares in the interval $I:=\left[\frac{4-c_\text{max}}{\mu},\frac{4-c_\text{min}}{\mu}\right]$.}\label{table:perfect_squares_in_the_interval}
\end{center}
\end{table}

\indent In No. 59 we have $\Delta>2$, so a particular treatment is needed. Let $T=T_{v_1}\oplus T_{v_2}\oplus T_{v_3}\oplus T_{v_4}=A_3\oplus A_2\oplus A_1\oplus A_1$. If $P$ generates the free part of $E(K)$ and $Q$ generates its torsion part, then $h(P)=\frac{1}{12}$ and $4P+Q$ meets the reducible fibers at $\Theta_{v_1,2},\Theta_{v_2,1},\Theta_{v_3,1},\Theta_{v_4,1}$ \cite{Kurumadani}[Example 1.7]. By Table~\ref{table:local_contributions} and the height formula (\ref{equation:height_formula_P}), 
$$\frac{4^2}{12}=h(4P+Q)=2+2(4P+Q)\cdot O-\frac{2\cdot 2}{4}-\frac{1\cdot 2}{3}-\frac{1}{2}-\frac{1}{2},$$

hence $(4P+Q)\cdot O=1$, as desired.
\end{proof}

\newpage

\chapter{Large rank jumps and the Hilbert property}\label{ch:rank_jumps}\
\indent Throughout this chapter $k$ denotes a number field and $X$ a geometrically rational elliptic surface over $k$ with elliptic fibration $\pi:X\to\P^1$ and Mordell-Weil rank $r$ over $k$. We study the variation of Mordell-Weil rank $r_t$ (over $k$) of the fiber $\pi^{-1}(t)$ as $t$ runs through $t\in\P^1_k$ in comparison to the generic rank $r$ and say that there is a \textit{rank jump} if $r_t>r$. We cover some of the approaches used in the literature, explain their limitations and how we adapt them in order to prove the main theorem of this chapter, namely 

\begin{teor}\label{thm:rank_jump_3_times}
Let $\pi:X\to \P^1$ be a geometrically rational elliptic surface over a number field $k$ with generic rank $r$. Assume that $\pi$ admits a \textit{RNRF}-conic bundle (Definition~\ref{def:RNRF}). Then the subset of the base of the elliptic fibration
$$\{t\in \P^1_k\mid r_t\geq r+3\}\subset \P^1_k$$
is not thin (Definition~\ref{def:thin_sets}).
\end{teor}

The starting point for our investigation is a specialization theorem by Silverman \cite[Thm. C]{Silverman}, which states that $r_t\geq r$ for all but finitely many $t\in\P^1_k$. We remark that this result is a development from a more general theorem by Néron \cite[Thm. 6]{Neron52} over higher-dimensional bases which says that the rank is at least the generic rank outside of a thin subset (see Definition~\ref{def:thin_sets}) of the base curve. In light of this, our natural guiding question is: what is the nature of the set of fibers where rank jumps occur?
	
Techniques to deal with this question include the study of the behavior of the root numbers in families first carried out by Rohrlich \cite{Rohrlich}; height theory estimations as done by Billard \cite{Billard} and geometric techniques, more precisely base change, introduced by Néron \cite{Neron52}. 

A considerable limitation of the root number approach is that it relies on the Birch and Swinnerton-Dyer conjecture and, even with that assumption, one is restricted to $\Q$ since the functional equation is only conjectural over number fields in general. As for the use of height theory, oftentimes one is also restricted to $\Q$, as many properties of the height machinery are proven only over the rationals. The geometric techniques, on the other hand, tend to allow more flexibility on the base field and have been used by several autors, such as \cite{Shioda91, Salgado12, Salgado15, HindrySalgado, ColliotThelene20}

\newpage

We explain Néron's geometric method. His goal in \cite[\S 5]{Neron56} is to start with $X$ with generic rank $8$ and build an elliptic surface $f:Y\to C$ such that $Y$ is a degree $6$ cover of $X$ with generic rank at least $11=8+3$. More precisely, the elliptic fibrations $\pi:X\to \P^1$ and $f:Y\to C$ are related by the following commutative diagram.

\begin{displaymath}
\begin{tikzcd}
Y\arrow[swap]{d}{6:1} \arrow{r}{f} & C \arrow{d}\\
X\arrow[swap]{r}{\pi}   & \P^1
\end{tikzcd}
\end{displaymath}

By applying Néron's specialization theorem, the fact that $f:Y\to C$ has rank at least $11$ implies that there are infinitely many fibers of $\pi$ whose ranks are at least $11=8+3$, hence a rank jump of $3$. The construction of $Y$ is performed by taking three rational bisections $L_1, L_2,L_3$ in $X$ (Definition~\ref{def:multisection}) such that the base change of $\pi:X\to \P^1$ by $C:=L_1\times_{\P^1}L_2 \times_{\P^1} L_3$ gives the new fibration $f:Y\to C$, which he shows admits $3$ new and independent sections. A key feature of this construction is the fact that the morphisms of the base $L_i\to \P^1$ share a common branch point, hence the curve $L_1\times_{\P^1}L_2 \times_{\P^1} L_3$ has genus $1$. 

A generalization of this technique was made by Salgado in \cite{Salgado12} when $X$ is any rational elliptic surface admitting a conic bundle over $k$. One key difference in comparison to Néron's construction is that the curves $L_i$ need not to be given explicitly. Instead, she shows a more general result for elliptic surfaces considering base change with respect to curves on a linear system, which produce a strictly greater rank:

\begin{teor}{\cite[Cor. 4.3]{Salgado12}}\label{thm:rank_jump_pencil}
Let $f:S\to C$ be an elliptic surface over a number field and $\mathscr{L}$ a pencil of curves on $S$ which does not contain a fiber of $f$. Then for all but finitely many $D\in\mathscr{L}$, the base-changed fibration $S\times_CD\to D$ has generic rank strictly greater than that of $f:S\to C$.
\end{teor}

A drawback in this construction, however, is that the morphisms to the base $D\to\P^1$ which play role of $L_i\to\P^1$ in Néron's construction do not necessarily admit a common branch point. In this case $D_1\times_{\P^1}D_2$ may already have genus $1$, which limits the number of base changes to $2$. As a consequence, one is only allowed to conclude about rank jumps of at least $2$, but not $3$, which is precisely the case in \cite{Salgado12} and \cite{LoughSalgado}. One way to overcome this difficulty is to consider conic bundles with certain ramification properties, which we call \textit{RNRF-conic bundles} (for \textit{ramified at a nonreduced fiber}, Definition~\ref{def:RNRF}), which allows us to detect rank jumps of at least $3$.

We want, moreover, to study the structure of set of fibers with rank jump of at least $3$. Given Néron's specialization theorem, it is natural to ask whether this set is thin or not. In the case of Néron's construction of rank jump $3$, the fact that the curves $L_i$ are rigid imply that the set is thin. On the other hand, in Salgado's construction \cite{Salgado12}, the pencil of curves implies a larger set of rank jumps which turns out to be \textit{not} thin, as proven later in collaboration with Loughran \cite{LoughSalgado}. By combining the technique from \cite{Salgado12} with the later developments from \cite{LoughSalgado}, we arrive at Theorem~\ref{thm:rank_jump_3_times}, which we prove in this chapter.

We organize the chapter as follows. In Section~\ref{section:nonreduced_fibers} we explore some properties of rational elliptic surfaces which admit a nonreduced fiber. In Section~\ref{section:RNRF} we introduce a key tool, which are the \textit{RNRF} conic bundles and present their constructions. Section~\ref{section:RNRF_and_multiple_base_changes} is a commentary on the need to consider nonreduced fibers, which introduces the proof of the main theorem in Section~\ref{section:rank_jump_3_times}. At last, a series of examples is provided in Section~\ref{section:examples_rank_jumps}.

\section{Nonreduced fibers}\label{section:nonreduced_fibers}\
\indent The central hypothesis of Theorem~\ref{thm:rank_jump_3_times} is the presence of an \textit{RNRF} conic bundle, i.e. a conic bundle ramified at a nonreduced fiber (Definition~\ref{def:RNRF}). By \textit{nonreduced} we mean a fiber with one or more components with multiplicity $\geq 2$, which by Kodaira's classification (Theorem~\ref{thm:Kodaira_classification}) correspond to starred types, i.e. $\text{I}^*_{n\geq 0}, \text{II}^*, \text{III}^*,\text{IV}^*$. In this section we present some results for the case when $\pi:X\to\P^1$ admits a nonreduced fiber (Proposition~\ref{prop:conic_bundles_from_Weierstrass_form}, Corollaries~\ref{coro:unirational} and \ref{coro:rank_jump_2}) and prove Lemma~\ref{lemma:support_RNRF}, which is helpful in the construction of conic bundles in Section~\ref{section:RNRF}.

We begin by considering a local Weierstrass form for our rational elliptic surface, namely
$$y^2+a_1xy+a_3y=x^3+a_2x^2+a_4x+a_6,\,\text{ where } a_i\in k(\Bbb{P}^1)\text{ for all }i.$$

Dokchitsers' Table~\ref{table:Dokchitser} allows us to readily see that rational elliptic fibrations with a nonreduced fiber at $t=\infty$, except possibly I$_0^*$, admit a conic bundle over the $x$-line. More precisely, we have the following result.

\begin{prop}\label{prop:conic_bundles_from_Weierstrass_form}
Let $\pi: X\rightarrow \P^1$ be a rational elliptic surface defined over $k$ with a nonreduced fiber $F$ at $t=\infty$ which is not of type $I_0^*$. Then $X$ admits a Weierstrass equation of the form
$$y^2+a_1(t)xy+a_3(t)y=x^3+a_2(t)x^2+a_4(t)x+a_6(t),\,\text{ with } \deg a_i\leq 2 \text{ for all }i.$$
In particular, $X$ admits a conic bundle by projection over the $x$-line.
\end{prop}

\begin{proof}
To analyze the behavior of the fiber at infinity, i.e. $t=(1:0) \in \mathbb{P}^1$, we write the (long) Weierstrass equation of $X$ with homogeneous coefficients:
$$y^2+a_1(t,u)xy+a_3(t,u)y=x^3+a_2(t,u)x^2+a_4(t,u)x+a_6(t,u).$$

Since $X$ is rational, $\deg a_i(t,u)=i$. Dokchitsers' Table~\ref{table:Dokchitser} tells us that $v_{\infty}(\frac{a_i}{i})\geq \frac{2}{3}$, if $F$ is of type IV$^*$, III$^*$ or II$^*$. Hence $\deg_u a_i\geq\frac{2i}{3}$ or $a_i=0$, which immediately bounds the degrees of $a_1, a_2, a_3, a_4$ and $a_6$ in the variable $t$ from above by 2, as claimed.

It remains to check that this also holds for fibers of type $\text{I}_n^*$ with $n\neq 0$. In this case, Dokschitsers' table tells us that $v_{\infty}(\frac{a_i}{i})\geq \frac{1}{2}$ and $v_{\infty}(d)\geq 6$. The former implies that $\deg_t a_i(u,t)\leq 2$ or $a_i=0$, for $i=1,2,3$ and $4$, while the latter gives us $\deg_t a_6(u,t)\leq 2$, or $a_6=0$.
\end{proof}

\begin{coro}\label{coro:unirational}
Let $\pi:X\to \P^1$ be a geometrically rational elliptic surface over $k$. Suppose that $\pi$ admits a unique nonreduced fiber. Then $X$ is $k$-unirational.
\end{coro}

\begin{proof}
We apply \cite[Thm. 7]{KollarMella} to the surface obtained by contracting the zero section, which is a conic bundle of degree $1$. 
\end{proof}

\begin{coro}\label{coro:rank_jump_2}
Let $\pi: X\rightarrow \P^1$  be a rational elliptic fibration defined over $k$ with a nonreduced fiber that is not of type {\normalfont I$_0^*$}. Then the subset 
$$\{t\in \P^1_k\mid r_t\geq r+2\}\subset\P^1_k$$

is not thin.
\end{coro}

\begin{proof}
Proposition~\ref{prop:conic_bundles_from_Weierstrass_form} puts us in a position to apply \cite[Theorem 1.1]{LoughSalgado}.
\end{proof}

\begin{remark} 
\textbf{ }
\begin{enumerate}[i)]
\item {\normalfont Rational elliptic surfaces with 2 fibers of type $\text{I}_0^*$ have been treated separately in \cite{LoughSalgado}}.
\item {\normalfont Unfortunately one might not be able to use the conic bundles from Proposition~\ref{prop:conic_bundles_from_Weierstrass_form} to show that there is a rank jump of 3, which is our main goal. Indeed, the hypothesis that the restriction of $\pi$ to the conics has a common ramification is crucial. By \cite[Lemma 2.10]{LoughSalgado}, this can only happen over nonreduced fibers. Example~\ref{example2} presents a surface with a fiber of type $\text{IV}^*$ whose conic bundle over the $x$-line is not ramified over the nonreduced fiber}.
\end{enumerate}
\end{remark}

The following lemma is relevant for the construction of conic bundles in Section~\ref{section:RNRF}. It provides a supply of divisors that can be used to form the support of a genus $0$ bisection over $k$.

\begin{lemma}\label{lemma:support_RNRF}
Let $\pi:X\rightarrow \P^1$ be a rational elliptic surface defined over $k$ with a nonreduced fiber $F$ with components $\Theta_i$'s as in Table~\ref{table:indices_reducible_fibers}. Then one of the following holds:
\begin{enumerate}[i)]
\item If {\normalfont $F=\text{II}^*$} or {\normalfont $\text{III}^*$}, then all its components are defined over $k$.
\item If {\normalfont $F=\text{IV}^*$}, then $\Theta_0,\Theta_1,\Theta_2$ are defined over $k$. The other components are defined over some extension of $k$ of degree at most $2$.
\item If {\normalfont $F=\text{I}_n^*$ ($n\geq 1$)}, then $\Theta_0,\Theta_1$ and all the nonreduced components are defined over $k$. The far components $\Theta_2,\Theta_3$ are defined over some extension of $k$ of degree at most $2$.
\item If {\normalfont $F=\text{I}_0^*$} and $F$ is the only reducible fiber of $\pi$, then $\Theta_0$ and $\Theta_4$ are defined over $k$. The other components are defined over an extension of $k$ of degree at most $3$.
\end{enumerate} 
\end{lemma}

\begin{proof}
The fibration $\pi$ is defined over $k$, and, by assumption, so is its zero section $O$. The Galois group $G:=\text{Gal}(\overline{k}/k)$ acts on the Néron-Severi group preserving intersection multiplicities.  Since $F$ is assumed to be the unique fiber of its type, the elements of $G$ permute the components of $F$. 

In particular, the presence of conjugate components implies symmetries in the intersection graph of $F$. For the sake of clarity and to avoid repetition in what follows, we state as a fact one of the immediate consequences of the intersection-preserving action of $G$ in the Néron-Severi group. We refer to it as \textbf{\emph{I.F.}} in what follows.
\\ \\
\noindent\textbf{Intersection Fact (I.F.):} Let $C,D,E\subset X$ be integral curves such that $C$ meets $D$ but not $E$. If $C$ is stable under $G$, then $D,E$ cannot be conjugate.
\\ \\
\indent We also note that since $O$ is defined over $k$ (hence stable under $G$) and meets $\Theta_0$, then $\Theta_0$ is necessarily stable under $G$. We now analyze i), ii), iii), iv) with the notation from Table~\ref{table:nonreduced_fibers}.

\begin{table}[h]
\begin{center}
\centering
\begin{tabular}{c c} 
\multirow{4}{*}{\hfil II$^*\,(\widetilde{E}_8)$} & \multirow{4}{*}{\hfil\EEight}\\
& \\
& \\
& \\
\multirow{4}{*}{\hfil III$^*\,(\widetilde{E}_7)$} & \multirow{4}{*}{\hfil \ESeven}\\ 
& \\
& \\
& \\
\multirow{4}{*}{\hfil IV$^*\,(\widetilde{E}_6)$} & \multirow{4}{*}{\hfil \ESix}\\ 
& \\
& \\
& \\
\multirow{4}{*}{\hfil I$_n^*\,(\widetilde{D}_{n+4})$} & \multirow{4}{*}{\hfil \Dn}\\ 
& \\
& \\
& \\
\end{tabular}
\end{center}
\caption{Nonreduced fibers of $\pi$}
\label{table:nonreduced_fibers}
\end{table}	
\newpage

\noindent i) Let $F=\text{II}^*$. For lack of symmetry, each component of $F$ is stable under $G$, as desired. Now let $F=\text{III}^*$. By symmetry $\Theta_3, \Theta_7$ are stable and the possible conjugate pairs are $(\Theta_0,\Theta_6)$, $(\Theta_1,\Theta_5)$, $(\Theta_2,\Theta_4)$. Since $\Theta_0$ is stable, then so is $\Theta_6$. We use \textbf{I.F.} with $(C,D,E)=(\Theta_0,\Theta_1,\Theta_5)$ so that $\Theta_1,\Theta_5$ are stable. Then for $(C,D,E)=(\Theta_1,\Theta_2,\Theta_4)$ we conclude $\Theta_2,\Theta_4$ are stable.
\\ \\
\noindent ii) Let $F=\text{IV}^*$. By symmetry, $\Theta_2$ is stable and the possible $G$-orbits are $(\Theta_0,\Theta_4,\Theta_6)$ and $(\Theta_1,\Theta_3,\Theta_5)$. Since $\Theta_0$ is stable, then $\Theta_4,\Theta_6$ are possibly conjugates, in which case both are defined over some $L/k$ with $[L:k]\leq 2$. Applying \textbf{I.F.} with $(C,D,E)=(\Theta_0,\Theta_1,\Theta_3)$ and then with $(C,D,E)=(\Theta_0,\Theta_1,\Theta_5)$ we conclude that $\Theta_1$ is stable and $(\Theta_3,\Theta_5)$ is possibly an orbit. In this case, $\Theta_3,\Theta_5$ are defined over some $L/k$ with $[L:k]\leq 2$.
\\ \\
\noindent iii) Let $F=\text{I}_{n\geq 1}^*$. By symmetry, the following pairs is a possible $G$-orbits: $(\Theta_i,\Theta_{n+j})$ where both $\Theta_i,\Theta_{n+j}$ are components with multiplicity $2$ in $F$ and $i+j=8$. Since $\Theta_0$ is stable, we apply \textbf{I.F.} with $(C,D,E)=(\Theta_0,\Theta_4,\Theta_{n+4})$ to conclude that $\Theta_4,\Theta_{n+4}$ are stable. We now use induction on $i$: assuming $\Theta_{i}$ is stable and applying \textbf{I.F.} to $(C,D,E)=(\Theta_i,\Theta_{i+1},\Theta_{n+j+1})$ with $j+i=8$, we conclude that $\Theta_{i+1},\Theta_{n+j+1}$ are stable, hence all components with multiplicity $2$ are stable. We are left with $\Theta_2,\Theta_3$. By symmetry, $(\Theta_0,\Theta_1,\Theta_2,\Theta_3)$ is a possible orbit. But $\Theta_0,\Theta_4$ are stable, so $\Theta_2,\Theta_3$ are possibly conjugates, hence both are defined over some $L/k$ with $[L:k]\leq 2$.
\\ \\
\noindent iv) Let $F=\text{I}_0^*$. By symmetry, $\Theta_4$ is stable and $(\Theta_0,\Theta_1,\Theta_2,\Theta_3)$ is a possible $G$-orbit. But $\Theta_0$ is stable, so $\Theta_1,\Theta_2,\Theta_3$ are possibly conjugates, all three defined over some $L/k$ with $[L:k]\leq 3$.
\end{proof}

%\begin{remark}
%\normalfont The hypothesis that $F$ is the only nonreduced fiber is unnecessary when $X$ is geometrically rational unless $F=\text{I}_0^*$. Indeed, rational elliptic surfaces have Euler number 12 and each nonreduced fiber has Euler number at least 6. It is 6 precisely when $F=\text{I}_0^*$.
%\end{remark}

\section{Conic bundles ramified at a nonreduced fiber (RNRF)}\label{section:RNRF}\
\indent We introduce conic bundles with the property that the restriction of the elliptic fibration to all conics share a ramification point, which are the main tool in the proof of Theorem~\ref{thm:rank_jump_3_times}.

%and can be thought of as a replacement for Néron's tangent lines in his construction of elliptic curves with rank $11$ (\cite{Neron56}).

We begin by defining a multisection \textit{ramified} at a fiber of $\pi:X\to\P^1$, which leads to the central object of this chapter, namely the \textit{RNRF-conic bundle}. We remind the reader that the presence of a conic bundle over $k$ is equivalent to the existence of a bisection of genus $0$ over $k$ by Lemma~\ref{lemma:bisection}.

\begin{defi}\label{def:bisection_ramifies}
Let $C\subset X$ be a multisection $\pi:X\to\P^1$. We say that $C$ is \textit{ramified} at the fiber $\pi^{-1}(t)$ if the finite morphism $\pi|_C:C\to\P^1$ branches at $t$.
\end{defi}

\begin{defi}\label{def:RNRF}
Let $\varphi_{|D|}:X\to\P^1$ be a conic bundle induced by a bisection $D\subset X$. If $\pi$ has a unique nonreduced fiber $F$, we say that this is an \textit{RNRF-conic bundle} if $D$ is ramified at $F$. Equivalently, if each $C\in |D|$ intersects (transversally) a nonreduced component of $F$.  The acronym \textit{RNRF} stands for \textit{Ramified at a Nonreduced Fiber}.
\end{defi}

RNRF-conic bundles arise naturally on rational elliptic surfaces with a nonreduced fiber. The following result tells us that they exist over $k$ without any further assumption depending on the type of nonreduced fiber, and gives conditions for it to happen in the remaining configurations.

\begin{prop}\label{prop:RNRF}
Let $\pi:X\to\P^1$ be a rational elliptic surface over $k$. Assume that one of the following holds:
\begin{enumerate}
\item It admits a fiber of type {\normalfont II$^*$, III$^*$} or {\normalfont I$_n^*$} for $n\in \{2,3,4\}$.
\item It admits a fiber of type {\normalfont IV$^*$} or {\normalfont I$_m^*$}, for $m\in\{0,1\}$ and a reducible, reduced fiber.
\item It admits a fiber of type {\normalfont IV$^*$} and a non-trivial section defined over $k$.
\item It admits a fiber of type {\normalfont I$_1^*$} and a non-trivial section defined over $k$ that intersects the near component.
\item It admits a fiber of type {\normalfont I$_1^*$} and two non-intersecting sections that are conjugate under $\text{Gal}(\overline{k}/k)$.
\item It admits 2 fibers of type {\normalfont I$_0^*$} and a non-trivial 2-torsion section defined over $k$.
\end{enumerate}

Then $X$ admits a RNRF-conic bundle over $k$.
\end{prop}

\begin{proof}
For each possible configuration listed in the hypothesis of the theorem, we provide an effective divisor $D$ on $X$ (see Table~\ref{table:class-of-conics}) such that $D^2=0$, $D\cdot F=2$ and such that every $D'\in|D|$ intersects a nonreduced component of $F$. Moreover, by Lemma~\ref{lemma:support_RNRF}, the divisor $D$ is defined over $k$, i.e. it remains invariant by the action of $\text{Gal}(\overline{k}/k)$.

To prove that $\pi|_D:D\to\P^1$ is ramified above $F$, we only need to observe that each $D$ constructed meets $F$ at a nonreduced component. 
\end{proof}

\newpage

\begin{table}[h]
\begin{center}
\centering
\begin{tabular}{|c|c|c|c|} 
\hline
$\text{MW rank} $ & \text{nonreduced fiber} & \text{class of conics} & \text{extra information}\\
\hline
\multirow{4}{*}{\hfil 0} & \multirow{4}{*}{\hfil II$^*$} & \multirow{4}{*}{\hfil \DiagramOne} &  \multirow{4}{*}{\hfil }\\
& & &\\
& & &\\
& & &\\
\hline
\multirow{4}{*}{\hfil 0,1} & \multirow{4}{*}{\hfil III$^*$} & \multirow{4}{*}{\hfil \DiagramTwo} & \multirow{4}{*}{\hfil }\\ 
& & &\\
& & &\\
& & &\\
\hline
\multirow{3}{*}{\hfil 0} & \multirow{3}{*}{\hfil IV$^*$} & \multirow{3}{*}{\hfil \DiagramThree} & \multirow{3}{*}{\hfil }\\ 
& & &\\
& & &\\
\hline
\multirow{3}{*}{0} & \multirow{3}{*}{\hfil I$_4^*$} & \multirow{3}{*}{\hfil \DiagramFour} & \multirow{3}{*}{\hfil }\\ 
& & &\\
& & &\\
\hline
\multirow{3}{*}{\hfil 0,1} & \multirow{3}{*}{\hfil I$_3^*$} & \multirow{3}{*}{\hfil \DiagramFive} & \multirow{3}{*}{\hfil }\\
& & &\\
& & &\\
\hline
\multirow{3}{*}{\hfil 0,1,2} & \multirow{3}{*}{\hfil I$_2^*$} & \multirow{3}{*}{\hfil \DiagramSix} & \multirow{3}{*}{\hfil }\\
& & &\\
& & &\\
\hline
\multirow{2}{*}{\hfil 0} & \multirow{2}{*}{\hfil I$_1^*$} & \multirow{2}{*}{\hfil \DiagramSeven} & \multirow{2}{*}{\hfil }\\
& & &\\
\hline
\multirow{3}{*}{\hfil 1,2} & \multirow{3}{*}{\hfil I$_1^*$} & \multirow{3}{*}{\hfil \DiagramEight} & \multirow{3}{*}{\hfil }\\
& & &\\
& & &\\
\hline
\multirow{3}{*}{\hfil 3} & \multirow{3}{*}{\hfil I$_1^*$} & \multirow{3}{*}{\hfil \DiagramNine} & \multirow{3}{*}{\thead{$P_1,P_2$ conjugate sections\\ intersecting near\\ components}}\\
& & &\\
& & &\\
\hline
\multirow{3}{*}{\hfil 1} & \multirow{3}{*}{\hfil IV$^*$} & \multirow{3}{*}{\hfil \DiagramTen} & \multirow{3}{*}{\hfil }\\
& & &\\
& & &\\
\hline
\multirow{2}{*}{\hfil 2} & \multirow{2}{*}{\hfil IV$^*$} & \multirow{2}{*}{\hfil \DiagramEleven} & \multirow{2}{*}{\thead{$P_1,P_2$ non-intersecting\\ conjugate sections}}\\
& & &\\
\hline
\end{tabular}
\end{center}
\caption{Class of conics for the proof of Proposition~\ref{prop:RNRF}}
\label{table:class-of-conics}
\end{table}

\newpage

\begin{remark}\label{rmk:stillrational}
\indent\par
\begin{enumerate}[i)]
\item {\normalfont Conics in a RNRF-conic bundle are distinguished in the following sense. Generally, a degree 2 base change of a rational elliptic surface produces a K3 surface. This is true if the base change is ramified in smooth fibers, and remais true even if it ramifies at reduced singular fibers if one considers the desingularization of the base changed surface. On the other hand, if a degree 2 base change ramifies at a nonreduced fiber, then by an Euler number calculation one can show that the base changed surface is still rational (see Lemma \ref{lemma:againrational}).} 

\item {\normalfont It is natural to wonder about the extra conditions on surfaces with fibers of type $\text{IV}^*$, $\text{I}_1^*$, $\text{I}_0^*$. One can show that surfaces with such fiber configuration always admit a conic bundle over $k$. Nevertheless, such a conic bundle is not necessarily a RNRF-conic bundle. It is worth noticing that, on the other hand, the isotrivial rational surfaces with $2$I$_0^*$ admit a RNRF-conic bundle. In this particular setting it is called a Châtelet bundle (see \cite[Lemma 3.3]{LoughSalgado}) since they occur as conic bundles on a Châtelet surface obtained after blowing down the sections of $\pi$. Still, they are useless for our purposes as all fibers ramify above the same $2$ nonreduced fibers, not leaving the degree of freedom needed to avoid certain covers when verifying that the rank jump occurs on a not thin subset.}
\end{enumerate}
\end{remark}

The following result justifies the study of RNRF-conic bundles on rational elliptic surfaces. It tells us that the base change of a rational elliptic fibration by a bisection in an RNRF-conic bundle is again a rational elliptic fibration. This allows us to apply \cite[Thm. 1.1]{LoughSalgado} to the base-changed rational elliptic surface and achieve a higher rank jump. 

\begin{lemma}\label{lemma:againrational}
Let $\pi:X\to\Bbb{P}^1$ be a rational elliptic surface with only one nonreduced fiber $F$. Let $D\subset X$ be a genus $0$ bisection ramified at $F$, and let $X_D$ be the normalization of the base change surface $X\times_{\P^1} D$. Then $X_D$ is a rational elliptic surface and $|D|$ induces a conic bundle on $X_D$.
\end{lemma}

\begin{proof}
Let $\varphi:=\pi|_D:D\to \Bbb{P}^1$ be the base change map. The curve $D$ is rational, so by Hurwitz formula $\varphi$ has two branch points. By hypothesis, these two points correspond to $F$ and some other reduced fiber. By inspection of possible singular fibers in a rational elliptic surface \cite{Persson}, $F$ is one of the following:
$$\text{I}_n^*\,\text{with }0\leq n\leq 4,\,\text{II}^*,\,\text{III}^*\text{ or IV}^*.$$

By the ramification of $\varphi$, there is precisely one fiber $F'$ in $X_D$ above $F$. An inspection of Table~\ref{table:quadratic_base_change}, confirms a pattern for the Euler number, namely, $e(F')=2\,e(F)-12$ ($*$).

Moreover, if $G$ is any singular fiber of $\pi$ distinct from $F$, we claim that $\sum_{G'}e(G')=2e(G)$ ($**$), where $G'$ runs through the fibers of $X_D$ above $G$ (since $\varphi$ has degree $2$, there are at most $2$ possibilities for $G'$). Indeed, if there is no ramification associated to $G$, then there are two fibers $G_1',G_2'$ isomorphic to $G$, hence $e(G_1')+e(G_2')=2\,e(G)$. In case there is ramification, there is only one possibility for $G'$ and $e(G')=2\,e(G)$ for all cases by Table~\ref{table:quadratic_base_change}. This proves $(**)$.
\newpage
We prove that $X_D$ is rational by checking that $e(X_D)=12$ and applying Lemma~\ref{lemma:euler_number_12}. Since $X$ is rational, then $e(X)=12$ and combining ($*$), ($**$) we get
\begin{align*}
e(X_D)&=e(F')+\sum_{G'\neq F'}e(G')\\
&=e(F')+\sum_{G\neq F}\sum_{G'}e(G')\\
&=2\,e(F)-12+2\sum_{G\neq F} 2\,e(G)\\
&=2\,e(X)-12\\
&=12,
\end{align*}

as desired. In order to prove that $|D|$ induces a conic bundle on $X_D$, take an arbitrary $C\in |D|$ and let $\psi:=\pi|_C:C\to\P^1$. As in $\varphi$, the map $\psi$ has $2$ branch points. Let $\nu:X_D\to X\times_{\P^1} D$ be the normalization map and $E\subset X_D$ the strict transform of $C\times_{\P^1} D$ under $\nu$. We prove that $E$ is smooth of genus $0$, so that $|E|$ induces a conic bundle on $X_D$. We have the following diagram:

\begin{displaymath}
\begin{tikzcd}
& \arrow[bend right]{ddl} E\arrow{d}{\nu}&\\
   & \arrow[swap]{dl}{\varphi'}C\times_{\P^1} D\arrow{dr}{\psi'} &\\
C\arrow[swap]{dr}{\psi}  &  & D\arrow{dl}{\varphi}\\
  & \P^1 &
\end{tikzcd}
\end{displaymath}

The singularities of $C\times_{\P^1} D$ are the points $(c,d)\in C\times D$ such that $\psi,\varphi$ ramify at $c,d$ respectively. These are also singular points of $X\times_{\P^1} D$, which are eliminated by $\nu$, hence $E$ is smooth. Since $\varphi$ has degree $2$, then $E\to C$ has degree $2$ and, moreover, has $2$ branch points, namely the ones where $\psi$ ramifies. By Hurwitz formula, $g(E)=0$ as desired. 
\end{proof}

\newpage

\section{RNRF and multiple base changes}\label{section:RNRF_and_multiple_base_changes}\
\indent In \cite{Salgado12} and \cite{LoughSalgado}, the authors use $2$ base changes by curves in a conic bundle to show that the rank jumps by at least $2$. The final base for the base changed fibration is an elliptic curve with positive Mordell-Weil rank. One is tempted to consider a third base change to increase the rank jump further. Unfortunately, this cannot be done in such generality. Indeed, the genus of a new base after a third base change would be at least $2$. In particular, the base curve would have a finite set of $k$-points thanks to Faltings' theorem \cite{Faltings}.

Rational elliptic surfaces with a nonreduced fiber on the other hand allow for a sequence of 3 base changes with final base curve of genus 1. Further hypotheses on the surface allow us to show that the genus 1 curve is an elliptic curve with positive Mordell--Weil rank. Indeed, Lemma \ref{lemma:againrational} assures that the first base change by an RNRF-conic bundle produces a surface that is again rational and moreover admits a conic bundle. In other words, admits a bisection of arithmetic genus $0$ as in the hypothesis of \cite[Theorem 2]{LoughSalgado}. One can then take the latter as the starting point and apply \cite[Theorem 2]{LoughSalgado} to conclude. This is explained in detail in the following section.

\section{Rank Jump three times}\label{section:rank_jump_3_times}\
\indent In this section we make use of an RNRF-conic bundle on a rational elliptic surface with a unique nonreduced fiber to show that the collection of fibers for which the Mordell--Weil rank is at least the generic rank plus 3 is not thin as a subset of the base of the fibration.

Throughout this section,  $\pi: X\rightarrow \mathbb{P}^1$ is a geometrically rational elliptic surface with a unique nonreduced fiber $F$ and $D$ is a bisection of $\pi$ such that $|D|$ is a RNRF-conic bundle.

We let $\psi_i:Y_i\rightarrow \P^1$, for $i\in I$, be an arbitrary finite collection of finite morphisms of degree $\geq 2$, as in the discussion at the end of Section~\ref{section:thin_sets_Hilbert_property}. Our goal is to show that there is a curve $X \supset C {\overset{\varphi} \rightarrow} \P^1$ such that:
\begin{enumerate}
\item $C(k)$ is infinite;
\item $\textrm{rank}(X\times_{\P^1} C)(k(C)) \geq r+3$;
\item $\exists P\in\varphi(C(k))\subset \P^1(k)$ such that $P\notin\bigcup_i \psi(Y_i(k))$.
\end{enumerate}

\begin{lemma}\label{lemma:genus1}
Assume $\pi:X\to\P^1$ has a unique non-reduced fiber. Then for all $D_1\in |D|$, there are infinitely many pairs $(D_2,D_3)\in |D|\times |D|$ such that $D_1 \times_{\P^1} D_2\times_{\P^1} D_3$ is a curve of genus $1$.
\end{lemma}
\begin{proof}
For a given a $D_1$, there are only finitely many $D_2,D_3\in |D|$ such that $D_2$ or $D_3$ share 2 common ramification points with $D_1$. We exclude these. The 3 quadratic extensions $k(D_i)$ are non-isomorphic as they have precisely one ramification point in common. Hence they are linearly disjoint and the curve  $C=D_1\times_{\P^1} D_2 \times_{\P^1} D_3$ is geometrically integral. Our hypothesis on the common ramification implies moreover that the curve $D_1\times_{\P^1} D_2$ has genus $0$ (see \cite[Lemma 5.2]{LoughSalgado}). Let $t, t_i\in \P^1(k)$ be the two branch points of $D_i \to \P^1$ where $t$ is the common branch point corresponding to the nonreduced fiber, and $t_i\neq t_j$, for $i\neq j$. Consider the degree $2$ morphism $\phi: D_1\times_{\P^1} D_2 \times_{\P^1} D_3 \to D_1\times_{\P^1} D_2$. Then $\phi$ is ramified at the four points above $t_3 \in \P^1(k)$. A direct application of the Hurwitz formula gives that $C$ is a curve of genus $1$.
\end{proof}

\newpage

\begin{prop}\label{prop:proof}
Assume $\pi:X\to\P^1$ has a unique nonreduced fiber and let $\mathcal{P}=\{P_1,\cdots, P_m\}$ be a finite subset of $\P^1_k$. Then there are infinitely many $D_1, D_2, D_3\in |D|$ such that:
\begin{itemize}
\item[a)] $C:=D_1\times_{\P^1} D_2\times_{\P^1} D_3$ is an elliptic curve with positive Mordell-Weil rank over $k$;
\item[b)] $k(D_1)\otimes k(D_2)\otimes k(D_3)$ is linearly disjoint with every $k(Y_i)$;
\item[c)] The rank of the generic fiber of $X_C\rightarrow C$ is at least $r+3$;
\item[d)] $C$ ramifies above exactly one of $P_j$.
\end{itemize}
\end{prop}

\begin{proof}
This proof is analogous to the proof of \cite[Prop. 4.1]{LoughSalgado}. We may assume that $\mathcal{P}$ contains all branch points of $Y_i$ and the singular locus of $\pi$. In particular, it contains a ramification point of all conics in $|D|$. We call this point $P_1$. After choosing $D_1$ and $D_2$ in an infinite set such that the rank of the generic fiber of the elliptic fibration $\pi_{12}: X_{D_1\times_B D_2}\rightarrow D_1\times_{\P^1} D_2$ is at least $r+2$ and $D_1\times_{\P^1} D_2(k)$ is infinite and $X_{D_1\times_{\P^1} D_2}(k)$ is Zariski dense as in \cite{LoughSalgado}, allowing $D_3$ to vary in $|D|$ gives by Lemma~\ref{lemma:genus1} an infinite family of elliptic curves with positive Mordell-Weil rank that are bisections of $\pi_{12}$. By Theorem~\ref{thm:rank_jump_pencil}, all but finitely many of such curves can be used to base change and obtain an elliptic surface $X_C\rightarrow C$ with $C=D_1\times_{\P^1} D_2\times_{\P^1} D_3$ and generic fiber of rank at least $r+3$.  

The fibration $\pi$ has a unique nonreduced fiber so after excluding finitely many curves when picking $D_1,D_2$ and $D_3$ we may assume that $D_i$'s are not ramified over other singular fibers of $\pi$, nor do they share ramification points with $Y_i$ other than possibly $P_1$. This proves d).
\end{proof}

The following result is parallel to \cite[Lemma 5.5]{LoughSalgado}. Since we construct conic bundles whose members are always ramified at a nonreduced singular fiber, we need to reprove the result. Fortunately, that comes with no cost as elliptic fibrations defined over global fields have at least two singular fibers.

\begin{lemma}
The elliptic surface $X\times_{\P^1} C\rightarrow C$ is nonconstant, i.e. not isomorphic to $E\times C\to C$, where $E$ is an elliptic curve.
\end{lemma}
\begin{proof}
The surface $X\to \P^1$ is a relatively minimal geometrically rational elliptic surface defined over a global field. In particular, it has at least $2$ singular fibers. Hence there is at least $1$ reduced singular fiber $F$. Since we chose $C$ such that $C\to \P^1$ is not ramified at $F$, the pull-back of $F$ to $X\times_{\P^1} C\rightarrow C$ is a singular fiber.
\end{proof}

We have at hand all the tools needed to prove our main result.

\begin{proof}[Proof of Theorem~\ref{thm:rank_jump_3_times}]
We choose $C$ as in Proposition~\ref{prop:proof}. By construction, the curves $D_1, D_2, D_3$, and $Y_i$ are smooth and the respective maps to $\P^1$ share at most one branch point. On the other hand, the map $C\times_{\P^1} Y_i \to C$ is branched on the ramification points of $Y_i\to \P^1$. A direct application of the Riemann-Hurwitz formula gives $g(C\times_{\P^1} Y_i)\geq 2$. In particular, by Faltings' theorem \cite{Faltings}, $C\times_{\P^1} Y_i(k)$ is finite. To conclude, we invoke part c) of Proposition~\ref{prop:proof} and apply Theorem~\ref{thm:rank_jump_pencil} to the nonconstant elliptic surface $X\times_{\P^1} C\to C$.  
\end{proof}

\newpage

\section{Examples}\label{section:examples_rank_jumps}\
\begin{example}
Let $X$ be an elliptic surface with Weierstrass equation \[y^2=x^3+a(t)x+b(t),\] with $\deg a(t), b(t)\leq 1$ and $a(t)$ and $b(t)$ not simultaneously constant. Then $X$ admits a nonreduced fiber at infinity. More precisely: if $\deg a(t)=1$ then $X$ admits a fiber of type $\text{III}^*$; and if $\deg a(t)=0$ then $X$ admits a fiber of type $\text{II}^*$. 

The surface $X$ admits a RNRF-conic bundle over the $x$-line. In case $ii)$ this is the unique conic bundle on the surface. A nice geometric description for this case is as follows. Let $C$ be a plane cubic with an inflection point $P$ defined over $k$. Let $L$ be the line tangent to $C$ at $P$. We consider the following pencil of plane cubics 
$$uC+t(3L)=0,\,\, (t:u)\in\P^1,$$

whose base locus is the nonreduced scheme $\{9P\}$. We consider its 9-fold blow up and obtain a rational elliptic surface with a fiber of type II$^*$ at infinity. The unique conic bundle on it is given by the strict transforms of the pencil of lines through $P$. By following the blow ups one sees readily that all conics intersect the unique double component of the fiber of type II$^*$ (Figure~\ref{figure:II*}).

\begin{figure}[!h]\label{figure:II*}
\begin{center}
\includegraphics[scale=0.65]{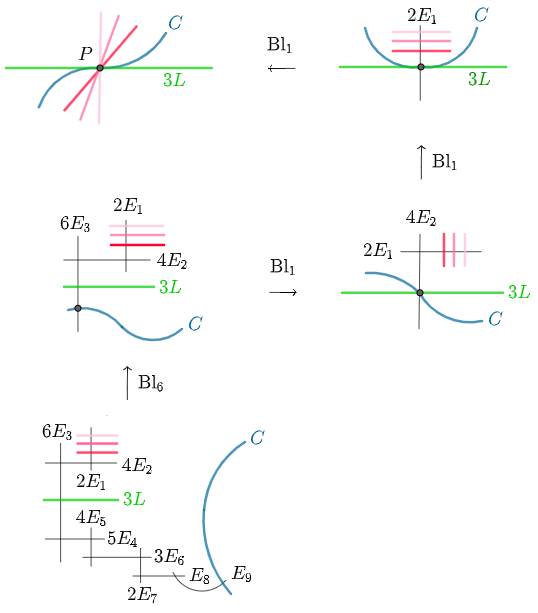}
\end{center}
\caption{RNRF-conic bundle (pink) on $X$ with a $II^*$ fiber.}\label{figure:II*}
\end{figure}
\end{example}

\newpage

\begin{example}\label{example2}
Let $F=y^2z-x^3+x^2z+xz^2-z^3$ and $G=z^2(y-z)$ be two plane cubics.  Let $P_1=(0:1:0)$ and $P_i=(x_i:1:1)$ with $x_i$, for $i=2,3,4$, the 3 roots of the polynomial $x^3-x^2-x$. Hence $F,G$ meet at the nonreduced scheme $F\cap G=\{6P_1,P_2,P_3,P_4\}$. Let $X$ be the blow up of $\mathbb{P}^2$ in $F\cap  G$. Then its (affine) Weierstrass equation is 
$$y^2-ty=x^3-x^2-x+(t-1).$$ 

In particular, $X$ has a fiber of type IV$^*$ at $t=\infty$ and, as expected by Proposition~\ref{prop:conic_bundles_from_Weierstrass_form}, it admits a conic bundle over the $x$-line. Geometrically, the fiber of type IV$^*$ is given by $G'-E_1-E_2-E_3-E_4$, where $G'$ is the proper transform of $G$ and $E_i$ is the exceptional divisor above $P_i$. If $D$ is a fiber of the conic bundle over the $x$-line then $D=\ell_1-E_1$, where $\ell_1$ is the proper transform of a line through $P_1$. In particular, for all lines through $P_1$, except the $3$ lines that pass through $P_2, P_3$ or $P_4$, the curve $D$ intersects the fiber IV$^*$ transversally in the simple component given by $m-E_2-E_3-E_4$ where $m$ is the proper transform of the line $y=z$, and in a simple component above the blow up of $P_1$. Hence the restriction of the elliptic fibration to all but $3$ conics is not ramified at IV$^*$ and, in particular, cannot all share a common ramification. In other words, the conic bundle over the $x$-line is not a RNRF-conic bundle.

Nevertheless, $X$ admits a RNRF-conic bundle, namely the one given by $|\ell_2-E_2|$, where $\ell_2$ is the proper transform of a line through $P_2$. Indeed, $\ell_2$ intersects the double component of IV$^*$ above the strict transform of $z^2=0$. Since $P_2=(0:1:1)$, the conic bundle is defined over $\Q$.
Thus there are infinitely many $t\in \mathbb{Q}$ such that $r_t\geq 1+3=4$ (Figure~\ref{figure:IV*}).

\begin{figure}[h]
\begin{center}
\includegraphics[scale=0.6]{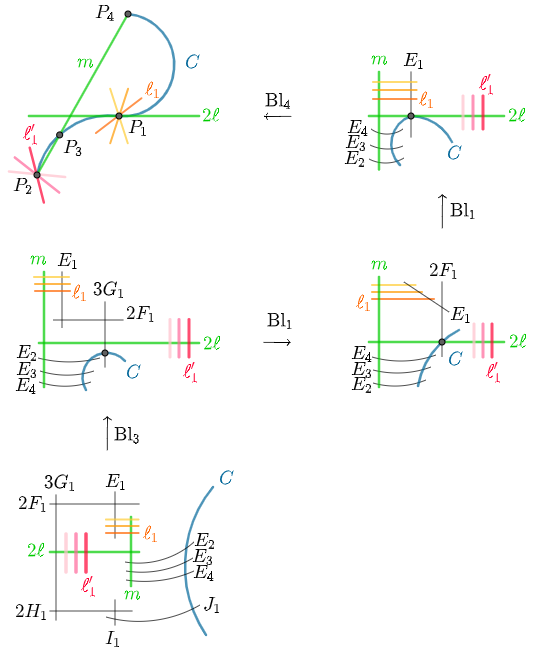}
\end{center}
\caption{RNRF-conic bundle (pink) and a non-RNRF conic bundle (orange) on $X$ with a IV$^*$ fiber.}\label{figure:IV*}
\end{figure}
\end{example}

\begin{example}
Let $X$ be the rational elliptic surface studied by Washington in \cite{Washington} with Weierstrass equation 
$$y^2=x^3+tx^2-(t+3)x+1.$$

The generic Mordell-Weil rank of this surface over $\mathbb{Q}$ is 1. Its singular fibers are of type $(\text{I}_2^*,2\text{II})$. In particular, by Proposition~\ref{prop:RNRF}, it admits a RNRF-conic bundle defined over $\mathbb{Q}$. We can apply Theorem~\ref{thm:rank_jump_3_times} to conclude that the subset of fibers of rank at least 4 is not thin.

For this surface, we can expect an even higher rank jump on a non-thin set. Indeed, Rizzo proved in \cite[Thm. 1]{Rizzo} that the root number of each fiber is $-1$. Hence, under the Parity conjecture \cite{DokchitserParity}, the rank of all fibers is odd. This together with Theorem~\ref{thm:rank_jump_3_times} would imply that the set of fibers with rank at least 5 is not thin. In other words, under the Parity conjecture, there is a rank jump of at least 4 for a non-thin set of fibers.

\end{example}

\newpage
\chapter{Appendix}\label{ch:appendix}\
\indent We reproduce part of the table in \cite[Main Th.]{OguisoShioda} with data on Mordell-Weil lattices of rational elliptic surfaces with Mordell-Weil rank $r\geq 1$. We add columns $c_\text{max},c_\text{min},\Delta$.

\begin{center}
\begin{longtable}{cccccccc}\label{table:MWL_data}
\multirow{1}{*}{No.} & \multirow{1}{*}{$r$} & \multirow{1}{*}{$T$} & \multirow{1}{*}{$E(K)^0$} & \multirow{1}{*}{$E(K)$} & \multirow{1}{*}{$c_\text{max}$} & \multirow{1}{*}{$c_\text{min}$} & \multirow{1}{*}{$\Delta$}\\
\hline
\multirow{1}{*}{1} & \multirow{1}{*}{$8$} & \multirow{1}{*}{$0$} & \multirow{1}{*}{$E_8$} & \multirow{1}{*}{$E_8$} & \multirow{1}{*}{0} & \multirow{1}{*}{0} & \multirow{1}{*}{0}\\
\hline
\multirow{2}{*}{2} & \multirow{2}{*}{$7$} & \multirow{2}{*}{$A_1$} & \multirow{2}{*}{$E_7$} & \multirow{2}{*}{$E_8^*$} & \multirow{2}{*}{$\frac{1}{2}$} & \multirow{2}{*}{$\frac{1}{2}$} & \multirow{2}{*}{$0$}\\
&\\
\hline
\multirow{2}{*}{3} & \multirow{2}{*}{$6$} & \multirow{2}{*}{$A_2$} & \multirow{2}{*}{$E_6$} & \multirow{2}{*}{$E_6^*$} & \multirow{2}{*}{$\frac{2}{3}$} & \multirow{2}{*}{$\frac{2}{3}$} & \multirow{2}{*}{$0$}\\
&\\
\multirow{2}{*}{4} & \multirow{2}{*}{} & \multirow{2}{*}{$A_1^{\oplus 2}$} & \multirow{2}{*}{$D_6$} & \multirow{2}{*}{$D_6^*$} & \multirow{2}{*}{$\frac{3}{2}$} & \multirow{2}{*}{$1$} & \multirow{2}{*}{$\frac{1}{2}$}\\
&\\
\hline
\multirow{2}{*}{5} & \multirow{2}{*}{$5$} & \multirow{2}{*}{$A_3$} & \multirow{2}{*}{$D_5$} & \multirow{2}{*}{$D_5^*$} & \multirow{2}{*}{$1$} & \multirow{2}{*}{$\frac{3}{4}$} & \multirow{2}{*}{$\frac{1}{4}$}\\
&\\
\multirow{2}{*}{6} & \multirow{2}{*}{} & \multirow{2}{*}{$A_2\oplus A_1$} & \multirow{2}{*}{$A_5$} & \multirow{2}{*}{$A_5^*$} & \multirow{2}{*}{$\frac{7}{6}$} & \multirow{2}{*}{$\frac{1}{2}$} & \multirow{2}{*}{$\frac{2}{3}$}\\
&\\
\multirow{2}{*}{7} & \multirow{2}{*}{} & \multirow{2}{*}{$A_1^{\oplus 3}$} & \multirow{2}{*}{$D_4\oplus A_1$} & \multirow{2}{*}{$D_4^*\oplus A_1^*$} & \multirow{2}{*}{$\frac{3}{2}$} & \multirow{2}{*}{$\frac{1}{2}$} & \multirow{2}{*}{$1$}\\
&\\
\hline
\multirow{2}{*}{8} & \multirow{2}{*}{$4$} & \multirow{2}{*}{$A_4$} & \multirow{2}{*}{$A_4$} & \multirow{2}{*}{$A_4^*$} & \multirow{2}{*}{$\frac{6}{5}$} & \multirow{2}{*}{$\frac{4}{5}$} & \multirow{2}{*}{$\frac{2}{5}$}\\
&\\
\multirow{2}{*}{9} & \multirow{2}{*}{} & \multirow{2}{*}{$D_4$} & \multirow{2}{*}{$D_4$} & \multirow{2}{*}{$D_4^*$} & \multirow{2}{*}{$1$} & \multirow{2}{*}{$1$} & \multirow{2}{*}{$0$}\\
&\\
\multirow{2}{*}{10} & \multirow{2}{*}{} & \multirow{2}{*}{$A_3\oplus A_1$} & \multirow{2}{*}{$A_3\oplus A_1$} & \multirow{2}{*}{$A_3^*\oplus A_1^*$} & \multirow{2}{*}{$\frac{3}{2}$} & \multirow{2}{*}{$\frac{1}{2}$} & \multirow{2}{*}{$1$}\\
&\\
\multirow{2}{*}{\hfil 11} & \multirow{2}{*}{\hfil } & \multirow{2}{*}{\hfil $A_2^{\oplus 2}$} & \multirow{2}{*}{$A_2^{\oplus 2}$} & \multirow{2}{*}{${A_2^*}^{\oplus 2}$} & \multirow{2}{*}{\hfil $\frac{4}{3}$} & \multirow{2}{*}{\hfil $\frac{2}{3}$} & \multirow{2}{*}{\hfil $\frac{2}{3}$}\\
&\\
\multirow{4}{*}{\hfil 12} & \multirow{4}{*}{} & \multirow{4}{*}{\hfil $A_2\oplus A_1^{\oplus 2}$} & \multirow{4}{*}{$\left(\begin{matrix}4 & -1 & 0 & 1\\-1 & 2 & -1 & 0\\0 & -1 & 2 & -1\\1 & 0 & -1 & 2\end{matrix}\right)$} & \multirow{4}{*}{$\frac{1}{6}\left(\begin{matrix}2 & 1 & 0 & -1\\1 & 5 & 3 & 1\\0 & 3 & 6 & 3\\-1 & 1 & 3 & 5\end{matrix}\right)$} & \multirow{4}{*}{\hfil $\frac{5}{3}$} & \multirow{4}{*}{\hfil $\frac{1}{2}$} & \multirow{4}{*}{\hfil $\frac{7}{6}$}\\
&\\
&\\
&\\
\multirow{2}{*}{\hfil 13} & \multirow{2}{*}{} & \multirow{2}{*}{\hfil $A_1^{\oplus 4}$} & \multirow{2}{*}{$D_4$} & \multirow{2}{*}{$D_4^*\oplus\Bbb{Z}/2\Bbb{Z}$} & \multirow{2}{*}{\hfil $2$} & \multirow{2}{*}{\hfil $\frac{1}{2}$} & \multirow{2}{*}{\hfil $\frac{3}{2}$}\\
&\\
\multirow{2}{*}{\hfil 14} & \multirow{2}{*}{} & \multirow{2}{*}{\hfil $A_1^{\oplus 4}$} & \multirow{2}{*}{$A_1^{\oplus 4}$} & \multirow{2}{*}{${A_1^*}^{\oplus 4}$} & \multirow{2}{*}{\hfil $2$} & \multirow{2}{*}{\hfil $\frac{1}{2}$} & \multirow{2}{*}{\hfil $\frac{3}{2}$}\\
&\\
\hline
\multirow{2}{*}{\hfil 15} & \multirow{2}{*}{$3$} & \multirow{2}{*}{\hfil $A_5$} & \multirow{2}{*}{$A_2\oplus A_1$} & \multirow{2}{*}{$A_2^*\oplus A_1^*$} & \multirow{2}{*}{\hfil $\frac{3}{2}$} & \multirow{2}{*}{\hfil $\frac{5}{6}$} & \multirow{2}{*}{\hfil $\frac{2}{3}$}\\
&\\
\multirow{2}{*}{\hfil 16} & \multirow{2}{*}{} & \multirow{2}{*}{$D_5$} & \multirow{2}{*}{$A_3$} & \multirow{2}{*}{$A_3^*$} & \multirow{2}{*}{\hfil $\frac{5}{4}$} & \multirow{2}{*}{\hfil $1$} & \multirow{2}{*}{\hfil $\frac{1}{4}$}\\
&\\
\multirow{3}{*}{17} & \multirow{3}{*}{} & \multirow{3}{*}{$A_4\oplus A_1$} & \multirow{3}{*}{$\left(\begin{matrix}4 & -1 & 1\\-1 & 2 & -1\\1 & -1 & 2\end{matrix}\right)$} & \multirow{3}{*}{$\frac{1}{10}\left(\begin{matrix}3 & 1 & -1\\1 & 7 & 3\\-1 & 3 & 7\end{matrix}\right)$} & \multirow{3}{*}{\hfil $\frac{17}{10}$} & \multirow{3}{*}{\hfil $\frac{1}{2}$} & \multirow{3}{*}{\hfil $\frac{6}{5}$}\\
&\\
&\\
\multirow{2}{*}{18} & \multirow{2}{*}{} & \multirow{2}{*}{$D_4\oplus A_1$} & \multirow{2}{*}{$A_1^{\oplus 3}$} & \multirow{2}{*}{${A_1^*}^{\oplus 3}$} & \multirow{2}{*}{\hfil $\frac{3}{2}$} & \multirow{2}{*}{\hfil $\frac{1}{2}$} & \multirow{2}{*}{\hfil $1$}\\
&\\
\multirow{3}{*}{19} & \multirow{3}{*}{} & \multirow{3}{*}{$A_3\oplus A_2$} & \multirow{3}{*}{$\left(\begin{matrix}2 & 0 & -1\\0 & 2 & -1\\-1 & -1 & 4\end{matrix}\right)$} & \multirow{3}{*}{$\frac{1}{12}\left(\begin{matrix}7 & 1 & 2\\1 & 7 & 2\\2 & 2 & 4\end{matrix}\right)$} & \multirow{3}{*}{\hfil $\frac{5}{3}$} & \multirow{3}{*}{\hfil $\frac{2}{3}$} & \multirow{3}{*}{\hfil $1$}\\
&\\
&\\
\multirow{2}{*}{20} & \multirow{2}{*}{} & \multirow{2}{*}{$A_2^{\oplus 2}\oplus A_1$} & \multirow{2}{*}{$A_2\oplus\langle 6\rangle$} & \multirow{2}{*}{$A_2^*\oplus\langle 1/6\rangle$} & \multirow{2}{*}{\hfil $\frac{11}{6}$} & \multirow{2}{*}{\hfil $\frac{1}{2}$} & \multirow{2}{*}{\hfil $\frac{4}{3}$}\\
&\\
\multirow{2}{*}{21} & \multirow{2}{*}{} & \multirow{2}{*}{$A_3\oplus A_1^{\oplus 2}$} & \multirow{2}{*}{$A_3$} & \multirow{2}{*}{$A_3^*\oplus\Z/2\Z$} & \multirow{2}{*}{$2$} & \multirow{2}{*}{$\frac{1}{2}$} & \multirow{2}{*}{$\frac{3}{2}$}\\
&\\
\multirow{2}{*}{22} & \multirow{2}{*}{} & \multirow{2}{*}{$A_3\oplus A_1^{\oplus 2}$} & \multirow{2}{*}{$A_1\oplus\langle 4\rangle$} & \multirow{2}{*}{$A_1^*\oplus\langle 1/4\rangle$} & \multirow{2}{*}{$2$} & \multirow{2}{*}{$\frac{1}{2}$} & \multirow{2}{*}{$\frac{3}{2}$}\\
&\\
\multirow{2}{*}{23} & \multirow{2}{*}{} & \multirow{2}{*}{$A_2\oplus A_1^{\oplus 3}$} & \multirow{2}{*}{$A_1\oplus\left(\begin{matrix}4 & -2\\-2 & 4\end{matrix}\right)$} & \multirow{2}{*}{$A_1^*\oplus\frac{1}{6}\left(\begin{matrix}2 & 1\\1 & 2\end{matrix}\right)$} & \multirow{2}{*}{$\frac{13}{6}$} & \multirow{2}{*}{$\frac{1}{2}$} & \multirow{2}{*}{$\frac{5}{3}$}\\
&\\
\multirow{2}{*}{24} & \multirow{2}{*}{} & \multirow{2}{*}{$A_1^{\oplus 5}$} & \multirow{2}{*}{$A_1^{\oplus 3}$} & \multirow{2}{*}{${A_1^*}^{\oplus 3}\oplus\Z/2\Z$} & \multirow{2}{*}{$\frac{5}{2}$} & \multirow{2}{*}{$\frac{1}{2}$} & \multirow{2}{*}{$2$}\\
&\\
\hline
\multirow{2}{*}{25} & \multirow{2}{*}{$2$} & \multirow{2}{*}{$A_6$} & \multirow{2}{*}{$\left(\begin{matrix}4 & -1\\-1& 2\end{matrix}\right)$} & \multirow{2}{*}{$\frac{1}{7}\left(\begin{matrix}2 & 1\\1 & 4\end{matrix}\right)$} & \multirow{2}{*}{$\frac{12}{7}$} & \multirow{2}{*}{$\frac{6}{7}$} & \multirow{2}{*}{$\frac{6}{7}$}\\
&\\
\multirow{2}{*}{26} & \multirow{2}{*}{} & \multirow{2}{*}{$D_6$} & \multirow{2}{*}{$A_1^{\oplus 2}$} & \multirow{2}{*}{${A_1^*}^{\oplus 2}$} & \multirow{2}{*}{$\frac{3}{2}$} & \multirow{2}{*}{$1$} & \multirow{2}{*}{$\frac{1}{2}$}\\
&\\
\multirow{2}{*}{27} & \multirow{2}{*}{} & \multirow{2}{*}{$E_6$} & \multirow{2}{*}{$A_2$} & \multirow{2}{*}{$A_2^*$} & \multirow{2}{*}{$\frac{4}{3}$} & \multirow{2}{*}{$\frac{4}{3}$} & \multirow{2}{*}{$0$}\\
&\\
\multirow{2}{*}{28} & \multirow{2}{*}{} & \multirow{2}{*}{$A_5\oplus A_1$} & \multirow{2}{*}{$A_2$} & \multirow{2}{*}{$A_2^*\oplus\Z/2\Z$} & \multirow{2}{*}{$2$} & \multirow{2}{*}{$\frac{1}{2}$} & \multirow{2}{*}{$\frac{3}{2}$}\\
&\\
\multirow{2}{*}{29} & \multirow{2}{*}{} & \multirow{2}{*}{$A_5\oplus A_1$} & \multirow{2}{*}{$A_1\oplus\langle 6\rangle$} & \multirow{2}{*}{$A_1^*\oplus\langle 1/6\rangle$} & \multirow{2}{*}{$2$} & \multirow{2}{*}{$\frac{1}{2}$} & \multirow{2}{*}{$\frac{3}{2}$}\\
&\\
\multirow{2}{*}{30} & \multirow{2}{*}{} & \multirow{2}{*}{$D_5\oplus A_1$} & \multirow{2}{*}{$A_1\oplus\langle 4\rangle$} & \multirow{2}{*}{$A_1^*\oplus\langle 1/4\rangle$} & \multirow{2}{*}{$\frac{7}{4}$} & \multirow{2}{*}{$\frac{1}{2}$} & \multirow{2}{*}{$\frac{5}{4}$}\\
&\\
\multirow{3}{*}{31} & \multirow{3}{*}{} & \multirow{3}{*}{$A_4\oplus A_2$} & \multirow{3}{*}{$\left(\begin{matrix}8 & -1\\-1 & 2\end{matrix}\right)$} & \multirow{3}{*}{$\frac{1}{15}\left(\begin{matrix}2 & 1\\1 & 8\end{matrix}\right)$} & \multirow{3}{*}{$\frac{28}{15}$} & \multirow{3}{*}{$\frac{2}{3}$} & \multirow{3}{*}{$\frac{6}{5}$}\\
&\\
&\\
\multirow{3}{*}{32} & \multirow{3}{*}{} & \multirow{3}{*}{$D_4\oplus A_2$} & \multirow{3}{*}{$\left(\begin{matrix}4 & -2\\-2 & 4\end{matrix}\right)$} & \multirow{3}{*}{$\frac{1}{6}\left(\begin{matrix}2 & 1\\1 & 2\end{matrix}\right)$} & \multirow{3}{*}{$\frac{5}{3}$} & \multirow{3}{*}{$\frac{2}{3}$} & \multirow{3}{*}{$1$}\\
&\\
&\\
\multirow{2}{*}{33} & \multirow{2}{*}{} & \multirow{2}{*}{$A_4\oplus A_1^{\oplus 2}$} & \multirow{2}{*}{$\left(\begin{matrix}6 & -2\\-2 & 4\end{matrix}\right)$} & \multirow{2}{*}{$\frac{1}{10}\left(\begin{matrix}2 & 1\\1 & 3\end{matrix}\right)$} & \multirow{2}{*}{$\frac{11}{5}$} & \multirow{2}{*}{$\frac{1}{2}$} & \multirow{2}{*}{$\frac{17}{10}$}\\
&\\
\multirow{2}{*}{34} & \multirow{2}{*}{} & \multirow{2}{*}{$D_4\oplus A_1^{\oplus 2}$} & \multirow{2}{*}{$A_1^{\oplus 2}$} & \multirow{2}{*}{${A_1^*}^{\oplus 2}$} & \multirow{2}{*}{$2$} & \multirow{2}{*}{$\frac{1}{2}$} & \multirow{2}{*}{$\frac{3}{2}$}\\
&\\
\multirow{2}{*}{35} & \multirow{2}{*}{} & \multirow{2}{*}{$A_3^{\oplus 2}$} & \multirow{2}{*}{$A_1^{\oplus 2}$} & \multirow{2}{*}{${A_1^*}^{\oplus 2}\oplus \Z/2\Z$} & \multirow{2}{*}{$2$} & \multirow{2}{*}{$\frac{3}{4}$} & \multirow{2}{*}{$\frac{5}{4}$}\\
&\\
\multirow{2}{*}{36} & \multirow{2}{*}{} & \multirow{2}{*}{$A_3^{\oplus 2}$} & \multirow{2}{*}{$\langle 4\rangle^{\oplus 2}$} & \multirow{2}{*}{$\langle 1/4\rangle^{\oplus 2}$} & \multirow{2}{*}{$2$} & \multirow{2}{*}{$\frac{3}{4}$} & \multirow{2}{*}{$\frac{5}{4}$}\\
&\\
\multirow{2}{*}{37} & \multirow{2}{*}{} & \multirow{2}{*}{$A_3\oplus A_2\oplus A_1$} & \multirow{2}{*}{$A_1\oplus\langle 12\rangle$} & \multirow{2}{*}{$A_1^*\oplus\langle 1/12\rangle$} & \multirow{2}{*}{$\frac{13}{6}$} & \multirow{2}{*}{$\frac{1}{2}$} & \multirow{2}{*}{$\frac{5}{3}$}\\
&\\
\multirow{2}{*}{38} & \multirow{2}{*}{} & \multirow{2}{*}{$A_3\oplus A_1^{\oplus 3}$} & \multirow{2}{*}{$A_1\oplus\langle 4\rangle$} & \multirow{2}{*}{$A_1^*\oplus\langle 1/4\rangle\oplus\Z/2\Z$} & \multirow{2}{*}{$\frac{5}{2}$} & \multirow{2}{*}{$\frac{1}{2}$} & \multirow{2}{*}{$2$}\\
&\\
\multirow{2}{*}{39} & \multirow{2}{*}{} & \multirow{2}{*}{$A_2^{\oplus 3}$} & \multirow{2}{*}{$A_2$} & \multirow{2}{*}{$A_2^*\oplus\Z/3\Z$} & \multirow{2}{*}{$2$} & \multirow{2}{*}{$\frac{2}{3}$} & \multirow{2}{*}{$\frac{4}{3}$}\\
&\\
\multirow{2}{*}{40} & \multirow{2}{*}{} & \multirow{2}{*}{$A_2^{\oplus 2}\oplus A_1^{\oplus 2}$} & \multirow{2}{*}{$\langle 6\rangle^{\oplus 2}$} & \multirow{2}{*}{$\langle 1/6\rangle^{\oplus 2}$} & \multirow{2}{*}{$\frac{7}{3}$} & \multirow{2}{*}{$\frac{1}{2}$} & \multirow{2}{*}{$\frac{11}{6}$}\\
&\\
\multirow{2}{*}{41} & \multirow{2}{*}{} & \multirow{2}{*}{$A_2\oplus A_1^{\oplus 4}$} & \multirow{2}{*}{$\left(\begin{matrix}4 & -2\\-2 & 4\end{matrix}\right)$} & \multirow{2}{*}{$\frac{1}{6}\left(\begin{matrix}2 & 1\\1 & 2\end{matrix}\right)$} & \multirow{2}{*}{$\frac{8}{3}$} & \multirow{2}{*}{$\frac{1}{2}$} & \multirow{2}{*}{$\frac{13}{6}$}\\
&\\
\multirow{2}{*}{42} & \multirow{2}{*}{} & \multirow{2}{*}{$A_1^{\oplus 6}$} & \multirow{2}{*}{$A_1^{\oplus 2}$} & \multirow{2}{*}{${A_1^*}^{\oplus 2}\oplus(\Z/2\Z)^2$} & \multirow{2}{*}{$3$} & \multirow{2}{*}{$\frac{1}{2}$} & \multirow{2}{*}{$\frac{5}{2}$}\\
&\\
\hline
\multirow{2}{*}{43} & \multirow{2}{*}{$1$} & \multirow{2}{*}{$E_7$} & \multirow{2}{*}{$A_1$} & \multirow{2}{*}{$A_1^*$} & \multirow{2}{*}{$\frac{3}{2}$} & \multirow{2}{*}{$\frac{3}{2}$} & \multirow{2}{*}{$0$}\\
&\\
\multirow{2}{*}{44} & \multirow{2}{*}{} & \multirow{2}{*}{$A_7$} & \multirow{2}{*}{$A_1$} & \multirow{2}{*}{$A_1^*\oplus\Z/2\Z$} & \multirow{2}{*}{$2$} & \multirow{2}{*}{$\frac{7}{8}$} & \multirow{2}{*}{$\frac{11}{8}$}\\
&\\
\multirow{2}{*}{45} & \multirow{2}{*}{} & \multirow{2}{*}{$A_7$} & \multirow{2}{*}{$\langle 8\rangle$} & \multirow{2}{*}{$\langle 1/8\rangle$} & \multirow{2}{*}{$2$} & \multirow{2}{*}{$\frac{7}{8}$} & \multirow{2}{*}{$\frac{11}{8}$}\\
&\\
\multirow{2}{*}{46} & \multirow{2}{*}{} & \multirow{2}{*}{$D_7$} & \multirow{2}{*}{$\langle 4\rangle$} & \multirow{2}{*}{$\langle 1/4\rangle$} & \multirow{2}{*}{$\frac{7}{4}$} & \multirow{2}{*}{$1$} & \multirow{2}{*}{$\frac{3}{4}$}\\
&\\
\multirow{2}{*}{47} & \multirow{2}{*}{} & \multirow{2}{*}{$A_6\oplus A_1$} & \multirow{2}{*}{$\langle 14\rangle$} & \multirow{2}{*}{$\langle 1/14\rangle$} & \multirow{2}{*}{$\frac{31}{14}$} & \multirow{2}{*}{$\frac{1}{2}$} & \multirow{2}{*}{$\frac{12}{7}$}\\
&\\
\multirow{2}{*}{48} & \multirow{2}{*}{} & \multirow{2}{*}{$D_6\oplus A_1$} & \multirow{2}{*}{$A_1$} & \multirow{2}{*}{$A_1^*$} & \multirow{2}{*}{$2$} & \multirow{2}{*}{$\frac{3}{2}$} & \multirow{2}{*}{$\frac{1}{2}$}\\
&\\
\multirow{2}{*}{49} & \multirow{2}{*}{} & \multirow{2}{*}{$E_6\oplus A_1$} & \multirow{2}{*}{$\langle 6\rangle$} & \multirow{2}{*}{$\langle 1/6\rangle$} & \multirow{2}{*}{$\frac{11}{6}$} & \multirow{2}{*}{$\frac{1}{2}$} & \multirow{2}{*}{$\frac{4}{3}$}\\
&\\
\multirow{2}{*}{50} & \multirow{2}{*}{} & \multirow{2}{*}{$D_5\oplus A_2$} & \multirow{2}{*}{$\langle 12\rangle$} & \multirow{2}{*}{$\langle 1/12\rangle$} & \multirow{2}{*}{$\frac{23}{12}$} & \multirow{2}{*}{$\frac{2}{3}$} & \multirow{2}{*}{$\frac{5}{4}$}\\
&\\
\multirow{2}{*}{51} & \multirow{2}{*}{} & \multirow{2}{*}{$A_5\oplus A_2$} & \multirow{2}{*}{$A_1$} & \multirow{2}{*}{$A_1^*\oplus\Z/3\Z$} & \multirow{2}{*}{$\frac{13}{6}$} & \multirow{2}{*}{$\frac{2}{3}$} & \multirow{2}{*}{$\frac{3}{2}$}\\
&\\
\multirow{2}{*}{52} & \multirow{2}{*}{} & \multirow{2}{*}{$D_5\oplus A_1^{\oplus 2}$} & \multirow{2}{*}{$\langle 4\rangle$} & \multirow{2}{*}{$\langle 1/4\rangle\oplus\Z/2\Z$} & \multirow{2}{*}{$\frac{9}{4}$} & \multirow{2}{*}{$\frac{1}{2}$} & \multirow{2}{*}{$\frac{7}{4}$}\\
&\\
\multirow{2}{*}{53} & \multirow{2}{*}{} & \multirow{2}{*}{$A_5\oplus A_1^{\oplus 2}$} & \multirow{2}{*}{$\langle 6\rangle$} & \multirow{2}{*}{$\langle 1/6\rangle\oplus\Z/2\Z$} & \multirow{2}{*}{$\frac{5}{2}$} & \multirow{2}{*}{$\frac{1}{2}$} & \multirow{2}{*}{$2$}\\
&\\
\multirow{2}{*}{54} & \multirow{2}{*}{} & \multirow{2}{*}{$D_4\oplus A_3$} & \multirow{2}{*}{$\langle 4\rangle$} & \multirow{2}{*}{$\langle 1/4\rangle\oplus\Z/2\Z$} & \multirow{2}{*}{$2$} & \multirow{2}{*}{$\frac{3}{4}$} & \multirow{2}{*}{$\frac{5}{4}$}\\
&\\
\multirow{2}{*}{55} & \multirow{2}{*}{} & \multirow{2}{*}{$A_4\oplus A_3$} & \multirow{2}{*}{$\langle 20\rangle$} & \multirow{2}{*}{$\langle 1/20\rangle$} & \multirow{2}{*}{$\frac{11}{5}$} & \multirow{2}{*}{$\frac{3}{4}$} & \multirow{2}{*}{$\frac{29}{20}$}\\
&\\
\multirow{2}{*}{56} & \multirow{2}{*}{} & \multirow{2}{*}{$A_4\oplus A_2\oplus A_1$} & \multirow{2}{*}{$\langle 30\rangle$} & \multirow{2}{*}{$\langle 1/30\rangle$} & \multirow{2}{*}{$\frac{71}{30}$} & \multirow{2}{*}{$\frac{1}{2}$} & \multirow{2}{*}{$\frac{28}{15}$}\\
&\\
\multirow{2}{*}{57} & \multirow{2}{*}{} & \multirow{2}{*}{$D_4\oplus A_1^{\oplus 3}$} & \multirow{2}{*}{$A_1$} & \multirow{2}{*}{$A_1^*$} & \multirow{2}{*}{$\frac{5}{2}$} & \multirow{2}{*}{$\frac{1}{2}$} & \multirow{2}{*}{$2$}\\
&\\
\multirow{2}{*}{58} & \multirow{2}{*}{} & \multirow{2}{*}{$A_3^{\oplus 2}\oplus A_1$} & \multirow{2}{*}{$A_1$} & \multirow{2}{*}{$A_1^*\oplus\Z/4\Z$} & \multirow{2}{*}{$\frac{5}{2}$} & \multirow{2}{*}{$\frac{1}{2}$} & \multirow{2}{*}{$2$}\\
&\\
\multirow{2}{*}{59} & \multirow{2}{*}{} & \multirow{2}{*}{$A_3\oplus A_2\oplus A_1^{\oplus 2}$} & \multirow{2}{*}{$\langle 12\rangle$} & \multirow{2}{*}{$\langle 1/12\rangle\oplus\Z/2\Z$} & \multirow{2}{*}{$\frac{8}{3}$} & \multirow{2}{*}{$\frac{1}{2}$} & \multirow{2}{*}{$\frac{13}{6}$}\\
&\\
\multirow{2}{*}{60} & \multirow{2}{*}{} & \multirow{2}{*}{$A_3\oplus A_1^{\oplus 4}$} & \multirow{2}{*}{$\langle 4\rangle$} & \multirow{2}{*}{$\langle 1/4\rangle\oplus\Z/2\Z$} & \multirow{2}{*}{$3$} & \multirow{2}{*}{$\frac{1}{2}$} & \multirow{2}{*}{$\frac{5}{2}$}\\
&\\
\multirow{2}{*}{61} & \multirow{2}{*}{} & \multirow{2}{*}{$A_2^{\oplus 3}\oplus A_1$} & \multirow{2}{*}{$\langle 6\rangle$} & \multirow{2}{*}{$\langle 1/6\rangle\oplus\Z/3\Z$} & \multirow{2}{*}{$\frac{5}{2}$} & \multirow{2}{*}{$\frac{1}{2}$} & \multirow{2}{*}{$2$}\\
&\\
\hline
\caption{Mordell-Weil lattices of rational elliptic surfaces with Mordell-Weil rank $r\geq 1$.}
\end{longtable}
\end{center}

\newpage

\bibliographystyle{alpha}
\bibliography{references_tese}

\newcommand{\etalchar}[1]{$^{#1}$}
\begin{thebibliography}{SAV{\etalchar{+}}65}

\bibitem[AGL16]{ArtebaniGarbagnatiLaface}
M.~Artebani, A.~Garbagnati, and A.~Laface.
\newblock Cox rings of extremal rational elliptic surfaces.
\newblock {\em Transactions of the American Mathematical Society},
  368(3):1735--1757, 2016.

\bibitem[AL09]{ArtebaniLaface}
M.~Artebani and A.~Laface.
\newblock Cox rings of surfaces and the anticanonical {Iitaka} dimension.
\newblock {\em Advances in Mathematics}, 266(6):5252--5267, 2009.

\bibitem[Art24]{Artin}
E.~Artin.
\newblock Quadratische {Körper} im {Gebiete} der höheren {Kongruenzen}. {II}.
  ({Analytischer Teil}).
\newblock {\em Mathematische Zeitschrift}, 19:207--246, 1924.

\bibitem[Bea96]{Beauville}
A.~Beauville.
\newblock {\em Complex Algebraic Surfaces}.
\newblock Cambridge University Press, 1996.

\bibitem[Ber12]{Bernays}
P.~Bernays.
\newblock {\em Über die Darstellung von positiven, ganzen Zahlen durch die
  primitive, binären quadratischen Formen einer nicht-quadratischen
  Diskriminante.}
\newblock PhD thesis, Göttingen, 1912.

\bibitem[BH]{BhargavaHanke}
M.~Bhargava and J.~Hanke.
\newblock Universal quadratic forms and the 290-{Theorem}.
\newblock {Preprint} at
  \url{http://math.stanford.edu/~vakil/files/290-Theorem-preprint.pdf}.

\bibitem[Bil98]{Billard}
H.~Billard.
\newblock Sur la répartition des points rationnels de surfaces elliptiques.
\newblock {\em Journal für die reine und angewandte Mathematik}, 505:45--71,
  1998.

\bibitem[BSD65]{BS-D}
B.~Birch and P.~Swinnerton-Dyer.
\newblock Notes on elliptic curves {(II)}.
\newblock {\em Journal für die Reine und angewandte Mathematik}, 218:79--108,
  1965.

\bibitem[CD89]{Dolga}
F.~R. Cossec and I.~V. Dolgachev.
\newblock {\em Enriques Surfaces I}, volume~76 of {\em Progress in
  Mathematics}.
\newblock Birkhäuser, 1989.

\bibitem[Cosa]{ConicBundles}
R.D. Costa.
\newblock Classification of conic bundles on a rational elliptic surface in any
  characteristic.
\newblock arXiv:2206.03549.

\bibitem[Cosb]{Gaps}
R.D. Costa.
\newblock Gaps on the intersection number of sections on a rational elliptic
  surface.
\newblock ArXiv: to appear.

\bibitem[CS]{DiasCostaSalgado}
R.D. Costa and C.~Salgado.
\newblock Large rank jumps on elliptic surfaces and the {Hilbert} property.
\newblock ArXiv: 2205.07801.

\bibitem[CT20]{ColliotThelene20}
J.-L. Colliot-Thélène.
\newblock Point générique et saut du rang du groupe de {Mordell-Weil}.
\newblock {\em Acta Arithmetica}, 196:93--108, 2020.

\bibitem[CZ79]{CoxZ}
D.~Cox and S.~Zucker.
\newblock Intersection numbers of sections of elliptic surfaces.
\newblock {\em Inventiones Mathematicae}, 53:1--44, 1979.

\bibitem[DD13]{Dokchitser}
T.~Dokchitser and V.~Dokchitser.
\newblock A remark on {Tate’s} algorithm and {Kodaira} types.
\newblock {\em Acta Arithmetica}, 160:95--100, 2013.

\bibitem[Dok13]{DokchitserParity}
T.~Dokchitser.
\newblock Notes on the parity conjecture.
\newblock In {\em Elliptic Curves, Hilbert Modular Forms and Galois
  Deformations}, pages 201--249. Springer Basel, 2013.

\bibitem[Ebe13]{Ebeling}
W.~Ebeling.
\newblock {\em Lattices and Codes}.
\newblock Springer Spektrum Wiesbaden, 2013.

\bibitem[Elk90]{Elkies90}
N.D. Elkies.
\newblock The {Mordell-Weil} lattice of a rational elliptic surface.
\newblock {\em Arbeitstagung Bonn}, 1990.

\bibitem[Elk06]{Elkies06}
N.D. Elkies.
\newblock $\mathbb{Z}^{28}$ in {$E(\mathbb{Q})$}.
\newblock {\em Number Theory Listserver}, 2006.

\bibitem[Enr49]{Enriques}
F.~Enriques.
\newblock {\em Le Superficie Algebriche}.
\newblock Zanichelli, 1949.

\bibitem[Fal83]{Faltings}
G.~Faltings.
\newblock Endlichkeitssätze für abelsche varietäten über zahlkörpern.
\newblock {\em Inventiones Mathematicae}, 73:349--366, 1983.

\bibitem[Fer92]{Fermigier92}
S.~Fermigier.
\newblock Un exemple de courbe elliptique definie sur $\mathbb{Q}(t)$ de rang
  $\geq 19$.
\newblock {\em Comptes rendus de l'Academie de Sciences de Paris},
  315(6):719--722, 1992.

\bibitem[Fer97]{Fermigier97}
S.~Fermigier.
\newblock Une courbe elliptique définie sur $\mathbb{Q}$ de rang $\geq 22$.
\newblock {\em Acta Arithmetica}, 82:359--363, 1997.

\bibitem[GS17]{GarbagnatiSalgado17}
A.~Garbagnati and C.~Salgado.
\newblock Linear systems on rational elliptic surfaces and elliptic fibrations
  on {K3} surfaces.
\newblock {\em Journal of Pure and Applied Algebra}, 223(1):277--300, 2017.

\bibitem[GS20]{GarbagnatiSalgado20}
A.~Garbagnati and C.~Salgado.
\newblock Elliptic fibrations on {K3} surfaces with a non-symplectic involution
  fixing rational curves and a curve of positive genus.
\newblock {\em Revista Matematica Iberoamericana}, 36(4):1167--1206, 2020.

\bibitem[Har77]{Hartshorne}
R.~Hartshorne.
\newblock {\em Algebraic Geometry}, volume~52 of {\em Graduate Texts in
  Mathematics}.
\newblock Springer {Verlag New York Inc}, 1977.

\bibitem[Har95]{Har}
B.~Harbourne.
\newblock Anticanonical rational surfaces.
\newblock {\em Transactions of the American Mathematical Society},
  349(3):1191--1208, 1995.

\bibitem[HK00]{HuKeel}
Y.~Hu and S.~Keel.
\newblock Mori dream spaces and {GIT}.
\newblock {\em Michigan Mathematical Journal}, 48(1):331--348, 2000.

\bibitem[HS19]{HindrySalgado}
M.~Hindry and C.~Salgado.
\newblock Lower bounds for the rank of families of {Abelian} varieties under
  base change.
\newblock {\em Acta Arithmetica}, 189:263--282, 2019.

\bibitem[Hum90]{Humphreys}
J.~E. Humphreys.
\newblock {\em Reflection Groups and Coxeter Groups}.
\newblock Cambridge University Press, 1990.

\bibitem[HW79]{HardyWright}
G.~H. Hardy and E.~M. Wright.
\newblock {\em An Introduction to the Theory of Numbers}.
\newblock Clarendon Press, 1979.

\bibitem[Isk79]{Isko}
V.~A. Iskovskikh.
\newblock Minimal models of rational surfaces over arbitrary fields.
\newblock {\em Izvestiya Akademii Nauk SSSR. Seriya Matematicheskaya},
  43(1):19--43, 1979.

\bibitem[Isk87]{IskoRationalityProblem}
V.~A. Iskovskikh.
\newblock On the rationality problem for conic bundles.
\newblock {\em Duke Mathematical Journal}, 54(2):271--294, 1987.

\bibitem[KM17]{KollarMella}
J.~Kollar and M.~Mella.
\newblock Quadratic families of elliptic curves and unirationality of degree 1
  conic bundles.
\newblock {\em American Journal of Mathematics}, 139(4):915--936, 2017.

\bibitem[Kod63a]{KodII}
K.~Kodaira.
\newblock On compact analytic surfaces {II}.
\newblock {\em Annals of Mathematics}, 77:563--626, 1963.

\bibitem[Kod63b]{KodIII}
K.~Kodaira.
\newblock On compact analytic surfaces {III}.
\newblock {\em Annals of Mathematics}, 78:1--40, 1963.

\bibitem[Kur14]{Kurumadani}
Y.~Kurumadani.
\newblock Pencil of cubic curves and rational elliptic surfaces.
\newblock Master's thesis, Kyoto University, 2014.

\bibitem[LS22]{LoughSalgado}
D.~Loughran and C.~Salgado.
\newblock Rank jumps on elliptic surfaces and the {Hilbert} property.
\newblock {\em Annales de l'Institut Fourier}, 72(2):617--638, 2022.

\bibitem[Man64]{Manin}
Y.~Manin.
\newblock The {Tate} height of points on an {Abelian} variety, its variants and
  applications.
\newblock {\em Izvestiya Akademii Nauk SSSR. Seriya Matematicheskaya},
  28(6):1363--1390, 1964.

\bibitem[Mes92]{Mestre}
J.-F. Mestre.
\newblock Un exemple de courbe elliptique sur $\mathbb{Q}$ de rang $\geq 15$.
\newblock {\em Comptes rendus de l'Academie de Sciences de Paris},
  314:453--455, 1992.

\bibitem[Mir89]{Miranda}
R.~Miranda.
\newblock The basic theory of elliptic surfaces (lecture notes), 1989.

\bibitem[MM98]{MartinMcMillen98}
R.~Martin and W.~McMillen.
\newblock An elliptic curve over $\mathbb{Q}$ with rank at least $23$.
\newblock {\em Number Theory Listserver}, 1998.

\bibitem[MM00]{MartinMcMillen00}
R.~Martin and W.~McMillen.
\newblock An elliptic curve over $\mathbb{Q}$ with rank at least $24$.
\newblock {\em Number Theory Listserver}, 2000.

\bibitem[MP89]{MirandaPersson}
R.~Miranda and U.~Persson.
\newblock Torsion groups of elliptic surfaces.
\newblock {\em Compositio Mathematica}, 72(3):249--267, 1989.

\bibitem[MT86]{ManTsfa}
Yu.~I. Manin and M.~A. Tsfasman.
\newblock Rational varieties: algebra, geometry and arithmetic.
\newblock {\em Russian Mathematical Surveys}, 41(2):51--116, 1986.

\bibitem[Né52]{Neron52}
A.~Néron.
\newblock Problèmes arithmétiques et géométriques rattachés à la notion
  de rang d'une courbe algébrique dans un corps.
\newblock {\em Bulletin de la Société Mathématique de France}, 80:101--166,
  1952.

\bibitem[Né56]{Neron56}
A.~Néron.
\newblock Propriétés arithmétiques de certaines familles de courbes
  algébriques.
\newblock {\em Proceedings of the International Congress of Mathematicians},
  3:481--488, 1956.

\bibitem[Né64]{Neron64}
A.~Néron.
\newblock Modèles minimaux des variétés abéliennes sur les corps locaux et
  globaux.
\newblock {\em Publications Mathématiques de l'IHÉS}, 21:5--128, 1964.

\bibitem[Nag92]{Nagao92}
K.I. Nagao.
\newblock Examples of elliptic curves over $\mathbb{Q}$ with rank $\geq 17$.
\newblock {\em Proceedings of the Japan Academy, Series A, Mathematical
  Sciences}, 68(9):287--289, 1992.

\bibitem[Nag93]{Nagao93}
K.I. Nagao.
\newblock Examples of elliptic curves over $\mathbb{Q}$ with rank $\geq 20$.
\newblock {\em Proceedings of the Japan Academy, Series A, Mathematical
  Sciences}, 69:291--293, 1993.

\bibitem[Nis96]{Nishiyama}
K.~Nishiyama.
\newblock The {Jacobian} fibrations on some {K3} surfaces and their
  {Mordell-Weil} groups.
\newblock {\em Japanese Journal of Mathematics}, 22(2):293--347, 1996.

\bibitem[NK94]{Nagao94}
K.I. Nagao and T.~Kouya.
\newblock Examples of elliptic curves over $\mathbb{Q}$ with rank $ \geq 21$.
\newblock 1994.

\bibitem[OS91]{OguisoShioda}
K.~Oguiso and T.~Shioda.
\newblock The {Mordell-Weil} lattice of a rational elliptic surface.
\newblock {\em Commentarii Mathematici Universitatis Sancti Pauli}, 40:83--99,
  1991.

\bibitem[Per90]{Persson}
U.~Persson.
\newblock Configurations of {Kodaira} fibres on rational elliptic surfaces.
\newblock {\em Mathematische Zeitschrift}, 205(1):1--47, 1990.

\bibitem[Pet95]{Peters}
C.~Peters.
\newblock {\em An introduction to complex algebraic geometry with emphasis on
  the theory of surfaces}, volume~23 of {\em Cours de l'Institute Fourier}.
\newblock 1995.

\bibitem[Pro18]{Prokho}
Yu.~G. Prokhorov.
\newblock The rationality problem for conic bundles.
\newblock {\em Russian Mathematical Surveys}, 73(3):375--456, 2018.

\bibitem[Riz03]{Rizzo}
O.~G. Rizzo.
\newblock Average root numbers for a nonconstant family of elliptic curves.
\newblock {\em Compositio Mathematica}, 136(1):1--23, 2003.

\bibitem[Roh93]{Rohrlich}
D.E. Rohrlich.
\newblock Variation of the root number in families of elliptic curves.
\newblock {\em Composition Mathematica}, 87(2):119--151, 1993.

\bibitem[Sal09]{SalgadoPhD}
C.~Salgado.
\newblock {\em Rang de surfaces elliptiques: théorèmes de comparaison}.
\newblock PhD thesis, Université Paris-Diderot (Paris-VII), 2009.

\bibitem[Sal12]{Salgado12}
C.~Salgado.
\newblock On the rank of the fibres of rational elliptic surfaces.
\newblock {\em Algebra \& Number Theory}, 6(7):1289--1314, 2012.

\bibitem[Sal15]{Salgado15}
C.~Salgado.
\newblock On the rank of the fibers of elliptic {K3} surfaces.
\newblock {\em Bulletin of the Brazilian Mathematical Society}, 43:7--16, 2015.

\bibitem[Sar80]{Sarkisov}
V.~G. Sarkisov.
\newblock Birational automorphisms of conic bundles.
\newblock {\em Izvestiya Akademii Nauk SSSR. Seriya Matematicheskaya},
  44(4):918--945, 1980.

\bibitem[SAV{\etalchar{+}}65]{Shafarevich65}
I.R. Shafarevich, B.G. Averbukh, J.R. Vainberg, A.B. Zhizhchenko, J.I. Manin,
  B.G. Moishezon, G.~N. Tyurin, and A.N. Tyurina.
\newblock Algebraic surfaces.
\newblock {\em Trudy Matematicheskogo Instituta imeni V.A. Steklova},
  75:1--215, 1965.

\bibitem[Ser08]{Serre}
J.-P. Serre.
\newblock {\em Topics in Galois Theory}, volume~1 of {\em Research Notes in
  Mathematics}.
\newblock A K Peters, 2 edition, 2008.

\bibitem[Shi72]{Shioda72}
T.~Shioda.
\newblock On elliptic modular surfaces.
\newblock {\em Journal of the Mathematical Society of Japan}, 24:20--59, 1972.

\bibitem[Shi89]{Shioda89}
T.~Shioda.
\newblock The {Mordell-Weil} lattice and {Galois} representation, {I}, {II},
  {III}.
\newblock {\em Proceedings of the Japan Academy, Series A, Mathematical
  Sciences}, 65:268--271; 296--299; 300--303, 1989.

\bibitem[Shi90]{Shioda90}
T.~Shioda.
\newblock On the {Mordell-Weil} lattices.
\newblock {\em Commentarii Mathematici Universitatis Sancti Pauli},
  39:211--240, 1990.

\bibitem[Shi91]{Shioda91}
T.~Shioda.
\newblock An infinite family of elliptic curves over $\mathbb{Q}$ with large
  rank via {Néron's} method.
\newblock {\em Inventiones Mathematicae}, 106(1):109--120, 1991.

\bibitem[Sil83]{Silverman}
J.~Silverman.
\newblock Heights and the specialization map for families of abelian varieties.
\newblock {\em Journal für die reine und angewandte Mathematik}, 342:197--211,
  1983.

\bibitem[Sil94]{SilvermanAdv}
J.~Silverman.
\newblock {\em Advanced Topics in the Arithmetic of Elliptic Curves}.
\newblock Springer New York, NY, 1994.

\bibitem[SS10]{Schuett-Shioda}
M.~Schuett and T.~Shioda.
\newblock Elliptic surfaces.
\newblock {\em Advanced Studies in Pure Mathematics}, 60:51--160, 2010.

\bibitem[SS19]{MWL}
M.~Schuett and T.~Shioda.
\newblock {\em Mordell-Weil Lattices}, volume~70 of {\em Ergebnisse der
  Mathematik und ihrer Grenzgebiete}.
\newblock Springer, 2019.

\bibitem[Szy04]{Szydlo}
M.~Szydlo.
\newblock Elliptic fibers over non-perfect residue fields.
\newblock {\em Journal of {Number Theory}}, 104(1):75--99, 2004.

\bibitem[Tat66]{Tate66}
J.~Tate.
\newblock On the conjectures of {Birch} and {Swinnerton-Dyer} and a geometric
  analog.
\newblock In {\em Séminaire Bourbaki : années 1964/65 1965/66, exposés
  277-312}, number~9 in Séminaire Bourbaki. Société mathématique de France,
  1966.
\newblock talk 306.

\bibitem[Tat75]{Tate75}
J.~Tate.
\newblock Algorithm for determining the type of a singular fibre in an elliptic
  pencil.
\newblock In B.~J. Birch and W.~Kuyk, editors, {\em Modular Functions of One
  Variable IV}, pages 33--52. Springer, 1975.

\bibitem[Was87]{Washington}
L.C. Washington.
\newblock Class numbers of the simplest cubic fields.
\newblock {\em Mathematics of Computation}, 48(177):371--384, 1987.

\bibitem[Wei48]{Weil}
A.~Weil.
\newblock {\em Sur les courbes algébriques et les variétés, qui s'en
  déduisent}.
\newblock Actualités scientifiques et industrielles 1041. Institute de
  mathématique de l'université de Strasbourg, 1948.

\end{thebibliography}
\end{document}